\documentclass[11pt,leqno]{amsart}

\usepackage{amsmath}
\usepackage{amssymb}
\usepackage{a4wide}
\usepackage{graphics}
\usepackage{epsfig}
\usepackage{stmaryrd}

\newtheorem{theorem}{Theorem}[section]
\newtheorem{lemma}[theorem]{Lemma}
\newtheorem{proposition}[theorem]{Proposition}
\newtheorem{definition}[theorem]{Definition}
\newtheorem{remark}[theorem]{Remark} 
\newtheorem{corollary}[theorem]{Corollary}

\newcommand{\Subsection}[1]{\subsection{ #1} ${}^{}$}

\newcounter{hypo}

\newenvironment{hyp}{  \begin{enumerate} \setcounter{enumi}{\value{hypo}} \item}{\stepcounter{hypo} \end{enumerate}}

\def\C{{\mathbb C}}
\def\N{{\mathbb N}} 
\def\R{{\mathbb R}}

\def\S{{\mathbb S}}

\def\CA{\mathcal {A}}

\def\CF{\mathcal {F}}

\def\CH{\mathcal {H}}

\def\CM{\mathcal {M}}

\def\CO{\mathcal {O}}

\def\CR{\mathcal {R}}
\def\CS{\mathcal {S}}
\def\CT{\mathcal {T}}

\def\CI{{\mathcal I}}

\def\ker{\mathop{\rm Ker}\nolimits}

\def\re{\mathop{\rm Re}\nolimits}
 \def\im{\mathop{\rm Im}\nolimits}

\def\Op{\mathop{\rm Op}\nolimits}
\def\Vect{\mathop{\rm Vect}\nolimits}
\def\oph{\mathop{\rm Op}_{h}\nolimits}

\def\supp{\mathop {\rm supp}\nolimits}

\def\bra{\langle} \def\ket{\rangle}
\def\diag{\mathop{\rm diag}\nolimits}
\def\sgn{\mathop{\rm sgn}\nolimits}
\def\Hess{\mathop{\rm Hess}\nolimits}
\def\<{\langle}
\def\>{\rangle}

\def\MS{\mathop{\rm MS}\nolimits}

\def\jj{\widehat\jmath}

\def\ds{\displaystyle}

\newcommand{\fract}[2]{\genfrac{}{}{0pt}{}{\scriptstyle #1}{\scriptstyle #2}}
\newcommand{\fractt}[3]{\fract{\fract{\scriptstyle #1}{\scriptstyle #2}}{\scriptstyle #3}}

\makeatletter
 \@addtoreset{equation}{section}
 \makeatother

\author{Ivana Alexandrova} 
\address{Ivana Alexandrova, Department of Mathematics, East Carolina University, Greenville, NC 27858, USA}
\email{alexandrovai@ecu.edu}
\author{Jean-Fran\c{c}ois Bony}
\address{Jean-Fran\c{c}ois Bony, Institut de Math\'ematiques de Bordeaux, (UMR CNRS 5251), Universit\'e de Bordeaux I, 33405 Talence, France}
\email{bony@math.u-bordeaux1.fr} 
\author{Thierry Ramond}
\address{Thierry Ramond, Math\'ematiques, Universit\'e Paris Sud, (UMR CNRS 8628), 91405 Orsay, France}
\email{thierry.ramond@math.u-psud.fr}

\title{Semiclassical scattering amplitude at the maximum point of the potential}

\keywords{Scattering amplitude, critical energy, Schr\"odinger equation}
\subjclass[2000]{81U20,35P25,35B38,35C20}

\thanks{\textbf{Acknowledgments:}  We would like to thank Johannes Sj\"ostrand for helpful discussions during the preparation of this paper.  The first author also thanks Victor Ivrii for supporting visits to Universit\'{e} Paris Sud, Orsay, and the Department of Mathematics at Orsay for the extended hospitality.}

\date{April 12, 2007}

\begin{document}

\begin{abstract} We compute the scattering amplitude for Schr\"odinger operators at a critical energy level, corresponding to the maximum point of the potential. We follow \cite{RoTa89_01},  using Isozaki-Kitada's representation formula for the scattering amplitude, together with results from \cite{bfrz} in order to analyze the contribution of trapped trajectories.
\end{abstract}

\maketitle

\tableofcontents

\section{Introduction}

\def\lll{{\ell\!\!\!\!\;\ell\!\!\!\!\;\ell}}

We study the semiclassical behavior of scattering amplitude at energy $E>0$ for Schr\"o\-dinger operators 
\begin{equation}  \label{q28}
P(x,hD) = -\frac{h^{2}}2 \Delta + V(x)
\end{equation}
where $V$ is a real valued $C^\infty$ function on $\R^{n}$, which vanishes at infinity.
We shall suppose here that $E$ is close to a critical energy level $E_0$ for $P$, which corresponds to a non-degenerate global maximum of the potential. Here, we address the case where this maximum is unique.

Let us recall that, if  $V(x)= \CO (\bra x\ket^{-\rho})$ for some $\rho>(n+1)/2$, then  for any $\omega\neq \theta\in \S^{n-1}$ and $E>0$, the problem
\begin{equation*}
\left\{
\begin{aligned}
&P(x,hD)u=E u,    \\
&u(x,h)=e^{i\sqrt{2E}x\cdot \omega/h}+\CA(\omega, \theta, E,h)
\frac{
e^{i\sqrt{2E}\vert x\vert/h}}
{\vert x\vert^{(n-1)/2}} 
+o(\vert x\vert^{(1-n)/2}) \mbox{ as } x\to +\infty, \ \frac x{\vert x\vert}=\theta,
\end{aligned}
\right.
\end{equation*}
has a unique solution.
The scattering amplitude at energy $E$ for the incoming direction $\omega$ and the outgoing direction $\theta$ is the real number $\CA(\omega, \theta, E,h)$.

For potentials that are not decaying  that fast at infinity, it is not that easy to  write down  a stationary formula for the scattering amplitude: If  $V(x)=\CO(\bra x\ket^{-\rho})$ for some $\rho>1$, one can define the scattering matrix at energy $E$ using wave operators (see Section \ref{sec4} below).  Then, writing
\begin{equation}
\CS(E,h)=Id-2i\pi \CT(E, h),
\end{equation}
one can see that  $\CT(E,h)$ is a  compact operator on $L^2(\S^{n-1})$, which kernel $\CT(\omega,\theta,E,h)$ is smooth out of the diagonal in $\S^{n-1}\times \S^{n-1}$. Then, the scattering amplitude  is given  for $\theta\neq \omega$, by 
\begin{equation}
\CA(\omega, \theta, E,h)=c(E))h^{(n-1)/2}\CT(\omega, \theta, E,h),
\end{equation}
 where
\begin{equation}
c(E)=-2\pi (2E)^{-\frac{n-1}{4}}(2\pi 
)^{\frac{n-1}{2}}e^{-i\frac{(n-3)\pi}{4}}.
\label{intro1}
\end{equation}

We proceed here as in \cite{RoTa89_01},  where D. Robert and H. Tamura have studied the semiclassical behavior of the scattering amplitude for short range potentials at  a non-trapping energy $E$ . 
 An energy $E$ is said to be non-trapping when $K(E)$, the trapped set $K(E)$ at energy $E$, is empty. This trapped set is defined as
\begin{equation}
K(E)=\left\{  (x,\xi)\in p^{-1}(E), \; \exp(tH_p)(x,\xi)\not\to \infty \mbox{ as } t\to \pm \infty  \right \},
\end{equation}
where $H_p$ is the Hamiltonian vector field associated to the principal symbol $p(x,\xi)=\frac 12\xi^2+V(x)$ of the operator $P$. Notice that the scattering amplitude has been first studied, in the semiclassical regime, by B. Vainberg \cite{Va77_01} and Y. Protas \cite{Pr82_01}
in the case of compactly supported potential, and for non-trapping energies, where they obtained the same type of result.

 Under the non-trapping assumption, and some other non-degeneracy condition (in fact our assumption (A4) below), D. Robert and H. Tamura  have shown that  the scattering amplitude has an asymptotic expansion with respect to $h$. The non-degeneracy assumption implies in particular that there is a finite number $N_\infty$ of classical trajectories for the Hamiltonian $p$, with asymptotic  direction $\omega$ for $t\to -\infty$ and asymptotic direction $\theta$ as $t\to +\infty$. Robert and Tamura's result is the following asymptotic expansion for the scattering amplitude:
\begin{equation}\label{rt}
\CA (\omega , \theta , E , h) = \sum_{j=1}^{N_\infty}e^{iS^\infty_j/h}
  \sum_{m\geq 0}a_{j,m}(\omega ,\theta ,E) h^m+ \CO (h^\infty),\ h\to 0,
\end{equation}
where $S^\infty_j$ is the classical action along the corresponding trajectory.
Also, they have computed  the first term in this expansion, showing that it can be given in terms of quantities  attached to the corresponding classical trajectory only.

There are also some few works concerning  the scattering amplitude when the non-trapping assumption is not fulfilled. In his paper \cite{Mi04_01}, L. Michel has shown that, if  there is no trapped trajectory  with incoming direction  $\omega$ and outgoing direction $\theta$ (see the discussion after \eqref{imr1} below), and if there is a complex neighborhood of $E$ of size $\sim h^N$  for some $N\in \N$ possibly large, which is free of resonances, then $\CA(\omega,\theta, E,h)$  is still given  by Robert and Tamura's formula. The potential is also supposed to be analytic in a sector out of a compact set, and the assumption on the existence of a resonance free domain around $E$ amounts to an estimate on boundary value of the meromorphic extension of the truncated resolvent  of the for
\begin{equation}
\label{intr1}
\Vert \chi(P -(E \pm i0))^{-1}\chi\Vert= \CO ( h^{-N}), \ \chi\in C^\infty_0(\R^{n}).
\end{equation}
Of course, these assumptions allow the existence of a non-empty trapped set. 

In \cite{ac} and \cite{asr}, the first author has shown that at non-trapping energies or in L. Michel's setting, the scattering amplitude is an h-Fourier Integral Operator associated to a natural scattering relation.  These results imply that the scattering amplitude admits an asymptotic expansion even without the non-degeneracy assumption, and in the sense of oscillatory integrals.  In particular, the expansion \eqref{rt} is recovered under the non-degeneracy assumption and as an oscillatory integral.

In \cite{LaMa99_01}, A. Lahmar-Benbernou and A. Martinez have computed 
 the scattering amplitude at energy $E\sim E_0$, in the case where the trapped set $K(E_0)$ consists in one single point corresponding to a local minimum of the potential (a well in the island situation). In that case, the estimate \eqref{intr1} is not true, and their result is obtained through a construction of the resonant states.

In the present work, we compute the scattering amplitude at energy $E\sim E_0$  in the case where the trapped set $K(E_0)$ corresponds to the unique  global maximum of the potential. The one-dimensional case has been studied in \cite{Ra96_01,FuRa98_01,FuRa03_01}, with specific techniques, and we consider  here the general $n> 1$ dimensional case. 

Notice that J. Sj\"ostrand in \cite{Sj87_01},  and P. Briet, J.-M. Combes and P. Duclos in \cite{BCD87_01,BCD87_02} have described the resonances close to $E_0$ in the case where $V$ is analytic in a sector around $\R^n$. From their result, it follows that Michel's assumption on the existence of a not too small resonance-free neighborhood of $E_0$ is satisfied. However, we show below (see Proposition \ref{app1}) that  for any $\omega\in \S^{n-1}$,  there is at least one half-trapped trajectory with incoming direction $\omega$, so that L. Michel's result never applies here.

Here, we do not assume analyticity for  $V$. We compute the contributions to the scattering amplitude arising from the classical trajectories reaching the unstable equilibrium point,  which corresponds to the top of the potential barrier. At the quantum level, tunnel effect occurs, which permits the particle to pass through this point. Our computation here relies heavily on \cite{bfrz}, where a precise description of this phenomena has been obtained.
In a forthcoming paper, we shall show that in this case also, the scattering amplitude is an $h$-Fourier Integral Operator.

This paper is organized in the following way. In Section \ref{sec2}, we describe our assumptions, and state our  main results:  a resolvent estimate, and the asymptotic expansion of the scattering amplitude in the semiclassical regime.
Section \ref{sec3} is devoted to the proof of the resolvent estimate, from which we deduce in Section  \ref{sec4} estimates similar to those in \cite{RoTa89_01}. In that section, we also recall briefly the representation formula for the scattering amplitude proved by Isozaki and Kitada, and introduce notations from \cite{RoTa89_01}.
The computation of the asymptotic expansion of the scattering amplitude is conducted in sections  \ref{sec5},  \ref{a20} and \ref{sec7}, following the classical trajectories.
Eventually, we have put in four appendices the proofs of some  side results  or technicalities.

\section{Assumptions and main results}
\label{sec2}

We suppose that the potential $V$ satisfies the following assumptions

\begin{hyp} \label{a1}
$V$ is a  $C^\infty$ function on $\R^{n}$,  and, for some $\rho>1$,
\begin{equation*}
\partial^\alpha V(x)= \CO (\bra x\ket^{-\rho-\vert\alpha\vert}).
\end{equation*}
\end{hyp}

\begin{hyp}  \label{a2}
$V$ has a non-degenerate maximum point at $x=0$, with $E_{0} =V (0) >0$ and
\begin{equation*}
\nabla^{2} V (0) =
\left( 
\begin{array}{ccc}
  \lambda_{1}^{2} & & \\
  & \ddots & \\
  & & \lambda_{n}^{2}
\end{array}
\right),
\quad 0 < \lambda_{1} \leq \lambda_{2} \leq \ldots \leq \lambda_{n}.
\end{equation*}
\end{hyp}

\begin{hyp}  \label{a3}
The trapped set at energy  $E_0$ is $K(E_0) = \{ (0,0 ) \}$.
\end{hyp}

Notice that the assumptions \ref{a1}--\ref{a3} imply that $V$ has an absolute global maximum at $x=0$. Indeed, if ${\mathcal L} = \{ x\neq 0; \ V (x) \geq E_{0} \}$ was non empty, the geodesic, for the Agmon distance $(E_{0} - V(x))_{+}^{1/2} dx$, between $0$ and ${\mathcal L}$ would be  the projection of a trapped bicharacteristic (see \cite[Theorem 3.7.7]{AbMa}).

As in  D. Robert and H. Tamura in \cite{RoTa89_01}, one of the key ingredient for the study of the scattering amplitude is a suitable estimate for the resolvent. Using the ideas in \cite[Section 4]{bfrz},  we have obtained the following result, that we think to be of independent interest. 

\begin{theorem}  \label{r1} \sl
Suppose assumptions \ref{a1}, \ref{a2} and \ref{a3} hold, and let $\alpha>\frac12$ be a fixed real number. We have
\begin{equation}
\Vert P-(E \pm i0)) ^{-1}\Vert_{\alpha , - \alpha} \lesssim h^{-1} \vert \ln h \vert ,
\end{equation}
uniformly for $\vert E - E_{0} \vert \leq \delta$, with $\delta >0$ small enough. Here $\Vert Q\Vert_{\alpha,\beta}$ denotes the norm of the bounded operator $Q$ from $L^2(\bra x\ket^\alpha\, dx)$ to $L^2(\bra x\ket^\beta\, dx)$.
\end{theorem}

Moreover, we prove in the Appendix \ref{d45} that our estimate is not far from optimal. Indeed, we have the

\begin{proposition}\label{r222}\sl
Under the assumptions \ref{a1} and \ref{a2}, we have
\begin{equation}  \label{q25}
\Vert (P - E_{0} \pm i0)^{-1} \Vert_{\alpha , - \alpha} \gtrsim h^{-1} \sqrt{\vert \ln h \vert} .
\end{equation}
\end{proposition}

We would like to mention that in the case of a closed hyperbolic orbit, the same upper bound has been obtained by N. Burq \cite{Bu04_01}  in the analytic category, and in a recent paper \cite{christ}  by H. Christianson in the $C^\infty$ setting.

As a matter of fact, in the present setting, S. Nakamura has proved in \cite{Na91_01} an $\CO(h^{-2})$ bound for the resolvent. Nakamura's estimate would be  sufficient for our proof of Theorem \ref{main}, but it is not  sharp enough for the computation of the total scattering cross section along the lines of  D. Robert and H. Tamura in \cite{RoTa87}. In that paper, the proof relies on a bound  $\CO(h^{-1})$ for the resolvent, but it is easy to see that an estimate like $\CO(h^{-1-\varepsilon})$ for any small enough $\varepsilon>0$ is sufficient. If we denote
\begin{equation}\label{zz2}
\sigma(\omega,E_0,h)=\int_{\S^{n-1}} \vert \CA(\omega, \theta, E,h)\vert^2 d\theta,
\end{equation}
the total scattering cross-section, and following D. Robert and H. Tamura's work,  our resolvent estimates gives the

\begin{theorem}\sl \label{robtam} Suppose assumptions \ref{a1}, \ref{a2} and \ref{a3} hold, and that $\rho>\frac{n+1}2$, $n\geq 2$. If $\vert E-E_0\vert<\delta$ for some $\delta>0$ small enough, then
\begin{equation}\label{zz3}
\sigma(\omega,E,h)=4\int_{\omega^\perp}
 \sin^2
 \left\{ 2^{-1}(2E)^{-1/2}h^{-1}\int_{\R} V(y+s\omega)ds
 \right\} dy +\CO(h^{-(n-1)/(\rho-1)}).
\end{equation}

\end{theorem}

Now we state our assumptions concerning the classical trajectories associated with the Hamiltonian $p$, that is curves $t\mapsto \gamma(t,x,\xi)=\exp(tH_p)(x,\xi)$ for some initial data $(x,\xi)\in T^*\R^{n}$.
Let us recall that, thanks to the decay of $V$ at infinity,  for given $\alpha\in\S ^{n-1}$ and
$z\in\alpha^{\perp}\sim \R^{n-1}$ (the impact plane), there is a unique bicharacteristic curve
\begin{equation}
\gamma_{\pm}(t,z,\alpha,E)=(x_{\pm}(t,z,\alpha,E),\xi_\pm(t,z,\alpha,E))
\end{equation}
such that
\begin{equation}
\begin{array}{l}
\ds\lim_{t\to \pm\infty}\vert x_{\pm}(t,z,\alpha,E)-\sqrt{2E}\alpha t-z\vert=0,
\\[10pt]
\ds\lim_{t\to \pm\infty}\vert \xi_{\pm}(t,z,\alpha,E)-\sqrt{2E}\alpha\vert=0.
\end{array}
\label{imr1}
\end{equation}
We  shall denote by $\Lambda^-_\omega$ the set of points in $T^*\R^{n}$ lying on trajectories going to infinity with direction $\omega$ as $t\to -\infty$, and  $\Lambda^+_\theta$ the set of those which lie on trajectories going to infinity with direction $\theta$ as $t\to +\infty$:
\begin{equation}
\begin{array}{l}
\Lambda^-_\omega=
\left \{
\gamma_-(t,z,\omega,E)\in T^*\R^{n}, z\in \omega^\perp, \ t\in \R 
\right\},\\[10pt]
\Lambda^+_\theta=
\left \{
\gamma_+(t,z,\theta,E)\in T^*\R^{n}, z\in \theta^\perp, \ t\in \R 
\right\}.
\end{array}
\end{equation}
We shall see that $\Lambda^-_\omega$ and $\Lambda^+_\theta$ are in fact  Lagrangian submanifolds of $T^*\R^{n}$.

Under the assumptions \ref{a1}, \ref{a2} and \ref{a3} there are only two possible behaviors for $x_{\pm}(t,z,\alpha,E)$ as $t\to\mp\infty$: either  it escapes to $\infty$, or  it goes to 0.

First we state our assumptions for the first kind of trajectories. For these, we also have, for some $\xi_\infty(z,\omega,E)$,
$$
\ds\lim_{t\to +\infty} \xi_-(t,z,\omega)=\xi_\infty(z,\omega,E),
$$
and we shall say that the trajectory $\gamma_-(t,z,\omega,E)$ has initial direction $\omega$ and final direction $\theta=\xi_\infty(z,\omega,E)/2\sqrt{E}$.  As in \cite{RoTa89_01} we shall suppose that  there is only a finite  number of trajectories with initial direction $\omega$ and final direction $\theta$. This assumption can be given in terms of the angular density 
\begin{equation}
\widehat{\sigma} (z)=\vert\det (\xi_\infty(z,\omega,E), \partial_{z_1}\xi_\infty(z,\omega,E),\dots,\partial_{z_{n-1}}\xi_\infty(z,\omega,E))\vert.
\end{equation}

\begin{definition}\sl
The outgoing direction $\theta\in\S ^{n-1}$ is called 
 regular  for
the incoming direction $\omega\in\S ^{n-1}$, or $\omega$-regular,  if $\theta\ne\omega$ and, for all $z'\in\omega^{\perp}$ with
$\xi_{\infty}(z', \omega,E)=2\sqrt{E}\theta$, the map $\omega^{\perp}\ni z\mapsto \xi_{\infty}(z, \omega,E)\in\S ^{n-1}$ is non-degenerate at $z'$, i.e. $\widehat{\sigma} (z')\neq 0$.
\end{definition}

\noindent We fix the incoming direction $\omega \in \S^{n-1}$, and we assume that 

\begin{hyp}\label{hyp4}
the direction $\theta\in  \S^{n-1}$ is $\omega$-regular.
\end{hyp}
Then, one can show  that  $\Lambda^-_{\omega} \cap \Lambda_{\theta}^+$ is a finite set of Hamiltonian trajectories  $(\gamma_j^\infty)_{1 \leq j \leq N_\infty}$, $\gamma_j^\infty(t)=\gamma^\infty(t,z^\infty_j)=(x_j^\infty(t), \xi_j^\infty(t))$, with  transverse  intersection along each of these curves.

We turn to  trapped  trajectories. Let us notice that the  linearization $F_p$ at $(0,0)$ of  the Hamilton vector field $H_p$  has eigenvalues $-\lambda_n,\dots,-\lambda_1,\lambda_1,\dots,\lambda_n$. Thus $(0,0)$ is a hyperbolic fixed point for  $H_p$, and 
the stable/unstable manifold Theorem  gives the existence of a stable incoming Lagrangian manifold $\Lambda_-$ and a stable outgoing Lagrangian manifold $\Lambda_+$ characterized by
\begin{equation}
\Lambda_\pm=\left\{(x,\xi)\in T^*\R^{n},\ \exp(tH_p)(x,\xi)\to 0 \mbox{ as }   t\to\mp\infty\right \}.
\end{equation}

In this paper, we shall describe the contribution to the scattering amplitude of the trapped trajectories, that is those going from infinity to the fixed point $(0,0)$. We have proved in Appendix A the following result, which shows that there are always such  trajectories.

\begin{proposition}\label{app1}
\sl For every $\omega , \theta \in \S^{n-1}$, we have
\begin{equation}
\Lambda_{\omega}^{-} \cap \Lambda_{-} \neq \emptyset \quad \text{and}
\quad \Lambda_{\theta}^{+} \cap \Lambda_{+} \neq \emptyset .
\end{equation}
\end{proposition}

We suppose that
\begin{hyp}  \label{a6m}
$\Lambda_{\omega}^{-}$ and $\Lambda_-$  (resp. $\Lambda^+_\theta$ and $\Lambda_+$) intersect transversally.
\end{hyp}
Under this assumption, $\Lambda_{\omega}^{-} \cap\Lambda_-$ and $\Lambda^+_\theta\cap\Lambda_+$ are finite sets of bicharacteristic curves. We denote them, respectively,
\begin{equation}\label{zzz10}
\gamma_k^-:t\mapsto \gamma^-(t,z^-_k)=(x_k^-(t),\xi^-(t) ), \quad 1\leq k\leq N_-, 
\end{equation}
and 
\begin{equation}\label{zzz11}
\gamma_\ell^+:t\mapsto \gamma^+(t,z^+_\ell)=(x^+(t),\xi^+(t) ), \quad 1\leq \ell\leq N_+.
\end{equation}
Here, the $z^-_k$ (resp. the $z^+_\ell$) belong to $\omega^\perp$ (resp. $\theta^\perp$) and determine the corresponding  curve by \eqref{imr1}.

We recall from \cite[Section 3]{HeSj85_01} (see also \cite[Section 5]{bfrz}),  that each integral curve $\gamma^{\pm}(t) = ( x^{\pm} (t) , \xi^{\pm} (t) ) \in \Lambda_\pm$ satisfies, in the sense of expandible functions (see Definition \ref{expdeb} below),
\begin{equation}
\gamma^\pm(t)\sim\sum_{j\geq 1}\gamma^\pm_{j}(t)e^{\pm\mu_jt}, \mbox{ as } t\to \mp\infty,
\end{equation}
where $\mu_1=\lambda_1<\mu_2<\dots$ is the strictly increasing sequence of linear combinations over $\N$ of the $\lambda_j$'s.
Here, the functions $\gamma_{j}^\pm:\R\to \R^{2n}$ are polynomials, that we write
\begin{equation}\label{zzz1}
\gamma_{j}^\pm(t)=\sum_{m=0}^{M'_j}\gamma^\pm_{j,m}t^m.
\end{equation}
Considering the base space projection of these trajectories, we denote
\begin{equation}\label{fzz1}
x^{\pm}(t) \sim\sum_{j = 1}^{+ \infty} g_{j}^{\pm} (t)e^{\pm\mu_{j} t}, \mbox{ as } t\to \mp\infty,\qquad
 g_{j}^{\pm} (t)=\sum_{m=0}^{M'_j}g^\pm_{j,m}t^m.
 \end{equation}
 Let us denote $\jj$ the (only) integer such that $\mu_{\jj}=2\lambda_1$. We prove in Proposition \ref{tr1000} below  that if $j<\jj$, then $M'_j=0$, or more precisely, that $ g_{j}^{\pm} (t)=g_{j}^{\pm} $ is a constant vector in $\ker (F_p\mp\lambda_j)$. We also have $M_{\jj}'\leq 1$, and $g^-_{\jj,1}$ can be computed in terms  of $g^-_1$.

In this paper, concerning the incoming trajectories, we shall assume that, 

\begin{hyp}  \label{a5m} For each $k\in \{1,\dots,N_-\}$, 
$g^-_{1}(z^-_k)\neq 0$.
\end{hyp}

Finally, we state our assumptions for the outgoing trajectories $\gamma_\ell^+\subset\Lambda_+\cap \Lambda^\theta_+$. First of all, it is easy to see, using Hartman's linearization theorem, that there exists always a $m\in\N$ such that $g_{m}^+(z^+_\ell)\neq 0$. We denote
\begin{equation}\label{zj1}
\lll=\lll(\ell)=\min\{m,\; g_{m}^+(z^+_\ell)\neq 0\}
\end{equation}
 the smallest of these $m$'s. We know that $\mu_\lll$ is one of the $\lambda_j$'s, and  that $M'_\lll=0$.

In \cite{bfrz}, we have been able to describe the branching process between an incoming curve $\gamma^-\subset\Lambda_-$ and an outgoing curve $\gamma^+\subset\Lambda_+$ provided $\langle g^-_{1}\vert g^+_{1}\rangle\neq 0$ (see the definition for $\widetilde\Lambda_+(\rho_-)$ before \cite[Theorem 2.6]{bfrz}).  Here, for the computation of the scattering amplitude, we can relax a lot this assumption, and analyze the branching in other cases that we describe now. Let us denote, for a given pair of paths $(\gamma^-(z^-_k),\gamma^+(z^+_\ell))$  in $(\Lambda^-_\omega\cap\Lambda_-)\times
(\Lambda^+_\theta\cap\Lambda_+)$,
\begin{equation}\label{zzzm2}
\CM_2(k,\ell)=-\frac1{8\lambda_1}
  \sum_{\fract{j \in\CI_1(2\lambda_1)}{\alpha,\beta \in\CI_2(\lambda_1)}}
\partial_j\partial^{\beta} V(0)\frac{(g^-_1(z_k^-))^{\beta}}{\beta!}\;
\partial_j \partial^{\alpha}V(0) \frac{(g^+_1(z_\ell^+))^\alpha}{\alpha !}\; ,
\end{equation}
and
\begin{align}\label{zzzm1}
\nonumber
\CM_1(k,\ell)=&
- \sum_{\fract{j \in \CI_{1}}{\alpha \in \CI_{2} (\lambda_{1})}}  \frac{\partial_j\partial^{\alpha } V(0)}{\alpha !} \big ( (g_1^-(z^-))^\alpha (g_{\widehat{\jmath} ,0}^{+} (z^{+} ))_{j}
+(g_{\jj,0}^-(z^-))_{j} (g_{1}^{+} (z^{+} ))^{\alpha}\big )
 \\
 & + \sum_{\alpha , \beta \in \CI_{2} (\lambda_{1})} \frac{(g_1^-(z^-))^\alpha}{\alpha !} \frac{(g_1^+(z^+))^\beta}{\beta !}C_{\alpha,\beta},
 \end{align}
 where
\begin{align}  
C_{\alpha,\beta}=&
 -\partial^{\alpha+\beta}V(0) + \sum_{j \in \CI_1 \setminus \CI_1(2\lambda_1)} \frac{4 \lambda_1^2}{\lambda_{j}^2 (4 \lambda_1^2 - \lambda^2_{j})}
\partial^{\alpha+\gamma}V(0)\partial^{\beta+\gamma}V(0)   \nonumber  \\
&
-\sum_{ \fractt{j \in \CI_{1}}{\gamma,\delta\in\CI_2(\lambda_1)}{\gamma+ \delta =\alpha+\beta}
} \frac{(\gamma+\delta)!}{\gamma !\; \delta !} \frac{1}{2\lambda_j^2} {\partial_j\partial^{\gamma}V(0)} {\partial_j\partial^{\delta}V(0)} . 
\end{align}
Here, we have set $\CI_1=\{1,\dots, n\}$, $1_j= (\delta_{ij})_{i=1 , \ldots , n}\in \N^n$ and
\begin{equation}\label{zh1}
\CI_m(\mu) = \{\beta\in \N^n,\  \beta=1_{k_1}+\cdots+ 1_{k_m} \text{ with } \lambda_{k_1} = \cdots = \lambda_{k_m} = \mu \},
\end{equation}
the set of multi-indices $\beta$ of length $\vert\beta\vert=m$ with each index of its non-vanishing components in the set $\{ j \in \N , \ \lambda_j = \mu \}$. We also denote $\CI_{m} \subset \N^{n}$ the set of multi-indices of length~$m$. 

We will suppose that

\begin{hyp}\label{a7} For each pair  of paths $(\gamma^-(z_k^-),\gamma^+(z^+_\ell))$, $k\in\{1,\dots,N_-\}$, $\ell\in\{1,\dots,N_+\}$, one of the three following cases occurs:\\

\begin{itemize}
\item[\textbf{(a)}] The set $\big\{m<\jj,\;  \langle g^-_{m}(z^-_k)\vert g^+_{m}(z^+_\ell)\rangle\neq 0\big\}$ is not empty. Then we denote
\begin{equation*}
\mathbf{k}=\min\big\{m<\jj,\;  \<g^-_{m}(z^-_k)\vert g^+_{m}(z^+_\ell)\>\neq 0\big\}.
\end{equation*}

\item[\textbf{(b)}] For all $m<\jj$, we have $\<g^-_{m}(z^-_k)\vert g^+_{m}(z^+_\ell)\>=0$, and $\CM_2(k,\ell)\neq 0$.
\\

\item[\textbf{(c)}] For all $m<\jj$, we have $\<g^-_{m}(z^-_k)\vert g^+_{m}(z^+_\ell)\>=0$, $\CM_2(k,\ell)= 0$ and $\CM_1(k,\ell)\neq 0$.
\end{itemize}
\end{hyp}

As one could expect (see \cite{RoTa89_01}, \cite{Ra96_01} or \cite{FuRa03_01}), action integrals appear in our formula for the scattering amplitude. We shall denote
\begin{eqnarray}
&&
S^\infty_j=\int_{-\infty}^ {+\infty}(\vert\xi^\infty_j(t)\vert^2-2E_0) dt,
\ j\in \{1,\dots,N_\infty\},
\\
&&
S^-_k=\int_{- \infty}^{+\infty}\big ( \vert \xi^{-}_k (t) \vert^2-2E_01_{t<0}\big ) dt, 
\quad k\in \{1,\dots,N_-\},
\\
&&
S^+_\ell=\int_{- \infty}^{+\infty}\big ( \vert \xi^{+}_\ell (t) \vert^2-2E_01_{t>0}\big ) dt, 
\quad \ell\in \{1,\dots,N_+\},
\end{eqnarray}
and  $\nu^\infty_j$, $\nu_\ell^+$, $\nu_k^-$  the Maslov indexes of the curves $\gamma^\infty_j$, $\gamma^+_\ell$, $\gamma^-_k$ respectively.
Let also 
\begin{align}\label{zm2}
D_{k}^- &=\lim_{t\to +\infty} \Big\vert\det\frac{\partial x_-(t,z,\omega,E_0)}{\partial(t,z)}\vert_{z=z^-_k}\Big\vert\; e^{-(\Sigma\lambda_j-2\lambda_1)t},\\
D_{\ell}^+ &=\lim_{t\to -\infty}\Big\vert\det\frac{\partial x_+(t,z,\omega,E_0)}{\partial(t,z)}\vert_{z=z^+_\ell}\Big\vert\; e^{(\Sigma\lambda_j-2\lambda_\lll)t},
\end{align}
be the Maslov determinants for $\gamma^-_k$, and $\gamma^+_\ell$respectively. We show  below that 
$0<D_{k}^-,D_{\ell}^+<+\infty$.
Eventually we  set
\begin{equation}\label{zzz2}
\Sigma(E,h)=\sum_{j=1}^n\frac{\lambda_j}{2} -  i\frac{E-E_0}{h}\cdotp
\end{equation}

Then, the main result of this paper is the

\begin{theorem}\sl \label{main} Suppose assumptions \ref{a1} to \ref{a7}
hold, and that $E\in \R$ is such that $E-E_0= \CO (h)$. Then
\begin{align}
\CA (\omega , \theta , E , h) =&
\sum_{j=1}^{N_\infty} 
\CA^{\rm reg}_j(\omega, \theta , E , h ) 
+ \sum_{k=1}^{N_-}
\sum_{\ell=1}^{N_+}
\CA_{k,\ell}^{\rm sing}
(\omega , \theta , E  ,h) 
 + \CO (h^{\infty}),
 \label{a8}
\end{align}
where
\begin{equation}\label{zz1}
\CA^{\rm reg}_j(\omega , \theta , E , h ) =  e^{iS^\infty_j/h}
  \sum_{m\geq 0}a_{j,m}^{\rm reg} (\omega ,\theta ,E) h^m, \quad
  a_{j,0}^{\rm reg} (\omega , \theta ,E)= \frac{e^{-i\nu^\infty_j\pi/2}}{\widehat{\sigma} (z_j)^{1/2}}\cdotp
\end{equation}
 Moreover we have
 \begin{itemize}
 \item In case \textbf{(a)}
\begin{align}\label{zzzt1a}
&\CA_{k,\ell}^{\rm sing}(\omega,\theta,E,h)=
e^{i(S^-_k+S^+_\ell)/h}\sum_{m\geq 0}a^{{\rm sing}}_{k,\ell,m}(\omega,\theta,E,\ln h)h^{(\Sigma(E)+\widehat\mu_m)/\mu_k-1/2} ,
\end{align}
where the $a^{{\rm sing}}_{k,\ell,m}(\omega,\theta,E,\ln h)$ are polynomials with respect to $\ln h$, and
\begin{align}\label{zzzt1a2}
a^{{\rm sing}}_{k,\ell,0}(\omega,\theta,E,\ln h)&=\frac{c(E)\sqrt{E}}{\pi^{1 - n/2}} \frac{e^{i(n\pi/4-\pi/2)}}{\mu_{{\bf k}}}  \Big( \prod_{j=1}^n  \lambda_{j} \Big)^{-1/2}\Gamma \big(\frac{ \Sigma(E)}{ \mu_{{\bf k}} }\big) 
 (2\lambda_1\lambda_{\lll})^{3/2}   \nonumber  
 \\
&\times  e^{- i \nu_{\ell}^{+} \pi /2} e^{-i\nu_{k}^{-}\pi/2} (D_k^-D^+_\ell)^{-1/2}\nonumber  \\
&\times  \vert g^{-}_{1} (z_{k}^{-})\vert\; \vert g^{+}_{\lll} (z_{\ell}^{+}) \vert \; \big( 2 i \mu_{{\bf k}} \big\< g_{{\bf k}}^{-} (z_{k}^{-}) \big\vert g_{{\bf k}}^{+} (z_{\ell}^{+}) \big\> \big)^{-\Sigma(E)/ \mu_{{\bf k}}}.
\end{align}

 \item In case \textbf{(b)}
\begin{equation}\label{zzzt1b}
\CA_{k,\ell}^{\rm sing}
(\omega , \theta , E , h)= e^{i(S^+_\ell+S^-_k)/h}a^{\rm sing}_{k,\ell}(\omega,\theta,E)
\frac{h^{\Sigma(E)/ 2 \lambda_{1}-1/2}}{\vert\ln h\vert^{\Sigma(E)/ \lambda_{1}}}(1+o(1)),
\end{equation}
where
\begin{align}
a^{{\rm sing}}_{k,\ell}(\omega,\theta,E)=
& \frac{c(E)\sqrt{E} }{ \pi^{1 - n/2}}e^{i(n\pi/4-\pi/2)} \Big( \prod_{j=1}^n  \lambda_{j} \Big)^{-1/2}
\Gamma \big(\frac{ \Sigma(E)}{2\lambda_1}\big) 
 (2\lambda_1 \lambda_{\lll})^{3/2} (2 \lambda_{1})^{\Sigma(E)/\lambda_{1} -1}
  \nonumber
     \\
 &\times  e^{- i \nu_{\ell}^{+} \pi /2} e^{-i\nu_{k}^{-}\pi/2} (D_k^-D^+_\ell)^{-1/2}\nonumber  \\
&\times  \vert g^{-}_{1} (z_{k}^{-})\vert\; \vert g^{+}_{\lll} (z_{\ell}^{+}) \vert \big( -i{\mathcal M}_{2}(k, \ell )\big)^{-\Sigma(E) / 2\lambda_{1}} .
\end{align}

 \item In case \textbf{(c)}
\begin{equation}\label{zzzt1c}
\CA_{k,\ell}^{\rm sing}
(\omega , \theta , E , h)= e^{i(S^+_\ell+S^-_k)/h}a^{{\rm sing}}_{k,\ell}(\omega,\theta,E)
\frac{h^{\Sigma(E)/ 2 \lambda_{1}-1/2}}{\vert\ln h\vert^{\Sigma(E)/2\lambda_{1}}}(1+o(1)),
\end{equation}
where
\begin{align}
a^{{\rm sing}}_{k,\ell}(\omega,\theta,E)=
& \frac{c(E)\sqrt{E} }{ \pi^{1 - n/2}}e^{i(n\pi/4-\pi/2)} \Big( \prod_{j=1}^n  \lambda_{j} \Big)^{-1/2} \Gamma\big (\frac{\Sigma(E)}{ 2 \lambda_{1}}\big ) 
(2\lambda_1 \lambda_{\lll})^{3/2} (2 \lambda_{1})^{\Sigma(E)/2\lambda_{1} -1}
    \nonumber
     \\
&\times  e^{- i \nu_{\ell}^{+} \pi /2} e^{-i\nu_{k}^{-}\pi/2} (D_k^-D^+_\ell)^{-1/2}\nonumber  \\
&\times  \vert g^{-}_{1} (z_{k}^{-})\vert\; \vert g^{+}_{\lll} (z_{\ell}^{+}) \vert \big( -i{\mathcal M}_{1}(k, \ell )\big)^{-\Sigma(E) / 2\lambda_{1}}.
\end{align}

\end{itemize}
Here,  the $\widehat{\mu}_{j}$ are the linear combinations over $\N$ of the $\lambda_k$'s and $\lambda_k-\lambda_1$'s, and the function $z\mapsto z^{-\Sigma(E)/\mu_k}$ is defined on $\C\setminus ]-\infty,0]$ and real positive on $]0,+\infty[$.


\end{theorem}

Of course the assumption that $\langle g^-_{1}\vert g^+_{1}\rangle\neq 0$ (a subcase of  \textbf{(a)}) is generic. Without the assumption \ref{hyp4}, the regular part $\CA^{reg}$ of the scattering  amplitude has an integral representation as in \cite{asr}. When the assumption \ref{a7} is not fulfilled, that is when the terms corresponding to the $\mu_j$ with $j\leq\jj\;$ do not contribute, we don't know if the scattering amplitude can be given  only in terms of the $g^\pm_1$'s and of the derivatives of the potential.

\section{Proof of the main resolvent estimate}
\label{sec3}

Here we prove Theorem  \ref{r1} using  Mourre's Theory. We start with the construction  of an escape function close to the stationary point $(0,0)$ in the spirit of \cite{BuZw04_01} and \cite{bfrz}. Since $\Lambda_{+}$ and $\Lambda_{-}$ are Lagrangian manifolds,  one can choose local symplectic coordinates $(y,\eta )$ such that
\begin{equation}  \label{aa1}
p (x, \xi) = B(y, \eta ) y\cdot \eta ,
\end{equation}
where $(y,\eta)\mapsto B(y,\eta)$ is a $C^\infty$ mapping from a neighborhood of (0,0) in $T^*\R^n$ to the space $\CM_{n}(\R)$ of $n \times n$ matrices with real entries, such that,
\begin{equation}
B (0,0) = \left( \begin{array}{ccc}
\lambda_{1} /2 & & \\
& \ddots & \\
& & \lambda_{n} /2
\end{array} \right).
\end{equation}
We denote $U$  a unitary Fourier Integral Operator (FIO) microlocally defined in a neighborhood of $(0,0)$, which canonical transformation is the map $(x, \xi )\mapsto(y, \eta )$, and we set
\begin{equation}
\widehat{P} = U P U^{*}.
\end{equation}
Here the FIO $U^{*}$ is the adjoint of $U$, and we have $U U^{*} = {\rm Id} + \CO (h^{\infty})$ and $U^{*} U = {\rm Id} + \CO (h^{\infty})$ microlocally near $(0,0)$.
Then $\widehat{P}$ is a pseudodifferential operator, with a real (modulo $O (h^{\infty})$) symbol $\widehat{p} (y, \eta )= \sum_{j} \widehat{p}_{j} (y, \eta ) h^{j}$, such that
\begin{equation}
\widehat{p}_{0} = B(y , \eta ) y \cdot \eta .
\label{hatp0}
\end{equation}
We set  $B_{1} = \oph (b_{1})$, 
\begin{equation}
b_{1} (y, \eta ) = \Big( \ln \Big\< \frac{y}{\sqrt{h M}} \Big\> - \ln \Big< \frac{\eta}{\sqrt{h M}} \Big\> \Big) \widetilde{\chi}_{2} (y, \eta ) ,
\end{equation}
where $M > 1$ will be fixed later and $\widetilde{\chi}_{1} \prec
\widetilde{\chi}_{2} \in C_{0}^{\infty} ( T^{*} \R^{n} )$
with $\widetilde{\chi}_{1} =1$ near $(0,0)$. In what follows, we will
assume that $h M <1$. In particular, $b_{1} \in S^{1/2} ( \vert \ln h \vert )$.
Here and in what follows,  we use the usual notation for classes of symbols. For $m$ an order function, a function $a (x,\xi,h) \in C^{\infty} (T^{*} \R^{n} )$  belongs to 
$S_{h}^{\delta} (m)$ when 
\begin{equation}
\forall \alpha\in\N^{2n},\  \exists C_{\alpha}>0 ,\  \forall h\in ]0,1] , \   |\partial_{x,\xi}^\alpha a(x,\xi,h)| \leq C_{\alpha} h^{-\delta |\alpha |} m(x,\xi ).
 \label{shdelta1}
\end{equation}
Let us also recall that,
if $a \in S^{\alpha} (1)$ and $b \in S^{\beta} (1)$, with $\alpha$,
$\beta < 1/2$, we have
\begin{equation}
\big[ \oph (a) , \oph (b) \big] = \oph \big( i h \{ b,a \} \big) + h^{3( 1- \alpha - \beta)} \oph (r) ,
\end{equation}
with $r \in S^{\min (\alpha , \beta ) } (1)$: In particular the term of order 2 vanishes. 

Hence, we have here
\begin{equation}  \label{r4}
[ B_{1} , \widehat{P} ] = \oph \big( i h \{ \widehat{p}_{0} , b_{1} \} \big) + \vert \ln h \vert \, h^{3/2} \oph (r_{M}) ,
\end{equation}
with $r_{M} \in S^{1/2} (1)$. The semi-norms of $r_{M}$ may depend on $M$. We have
\begin{equation}  \label{r5}
\{ \widehat{p}_{0} , b_{1} \} = c_{1} + c_{2} ,
\end{equation}
with
\begin{align}
c_{1} =& \Big( \ln \Big\< \frac{y}{\sqrt{h M}} \Big\> - \ln \Big< \frac{\eta}{\sqrt{h M}} \Big\> \Big) \{ \widehat{p}_{0} , \widetilde{\chi}_{2} \}   \label{r15} \\
c_{2} =& \Big\{ \widehat{p}_{0} , \ln \Big\< \frac{y}{\sqrt{h M}} \Big\> - \ln \Big< \frac{\eta}{\sqrt{h M}} \Big\> \Big\} \widetilde{\chi}_{2}   \nonumber \\
=& \Big( \big( By+ (\partial_{\eta} B) y \cdot \eta \big) \cdot \frac{y}{hM + y^{2}} + \big( B \eta + (\partial_{y} B) y \cdot \eta \big) \cdot \frac{\eta}{hM + \eta^{2}} \Big) \widetilde{\chi}_{2}. \label{r16}
\end{align}
The symbols $c_{1} \in S^{1/2} (\vert \ln h \vert )$, $c_{2} \in S^{1/2} (1)$ satisfy $\supp (
c_{1} ) \subset \supp ( \nabla \widetilde{\chi}_{2} )$. Let
$\widetilde{\varphi} \in C_{0}^{\infty} ( T^{*} \R^{n} )$ be a
function such that $\widetilde{\varphi} =0$ near $(0,0)$ and
$\widetilde{\varphi} =1$ near the support of $\nabla
\widetilde{\chi}_{2}$. We have
\begin{align}
\oph (c_{1} ) =& \oph ( \widetilde{\varphi} ) \oph (c_{1} ) \oph (
\widetilde{\varphi} ) + \CO (h^{\infty})    \nonumber \\
\geq & - C_{1} h \vert \ln h \vert \oph ( \widetilde{\varphi} ) \oph (
\widetilde{\varphi} ) + \CO (h^{\infty})    \nonumber  \\
\geq & - C_{1} h \vert \ln h \vert \oph ( \widetilde{\varphi}^{2} ) + \CO
(h^{2} \vert \ln h \vert ),  \label{r6}
\end{align}
for some $C_{1} >0$. On the other hand, using \cite[(4.96)--(4.97)]{bfrz}, we get
\begin{equation}
\oph (c_{2}) \geq \varepsilon M^{-1} \oph ( \widetilde{\chi}_{1} ) + \CO
(M^{-2}) , \label{r7}
\end{equation}
for some $\varepsilon >0$. With the notation $A_{1} = U^{*} B_{1} U$,
the formulas \eqref{r4}, \eqref{r5}, \eqref{r6} and \eqref{r7} imply
\begin{align}
-i [A_{1} , P] = & -i U^{*} [ B_{1} , P ] U   \nonumber \\
\geq & \varepsilon h M^{-1} U^{*} \oph ( \widetilde{\chi}_{1} ) U 
- C_{1} h \vert \ln h \vert U^{*} \oph ( \widetilde{\varphi}^{2} ) U
\nonumber \\
&+ \CO (h M^{-2}) + \CO_{M} ( h^{3/2} \vert \ln h \vert ) .   \label{r3}
\end{align}
If $\kappa$ is the canonical transformation associated to $U$, then $\chi_{j} = \widetilde{\chi}_{j} \circ \kappa$, $j=1,2$ and $\varphi = \widetilde{\varphi} \circ \kappa$ are $C_{0}^{\infty} (T^{*} (\R^{n}) , [0,1] )$ functions which satisfy $\chi_{1} =1$ near $(0,0)$ and $\varphi = 0$ near $(0,0)$. Using Egorov's Theorem, \eqref{r3} becomes
\begin{equation}
-i [A_{1} , P] \geq \varepsilon h M^{-1} \oph ( \chi_{1} ) - C_{1} h \vert \ln h \vert
\oph ( \varphi ) + \CO (h M^{-2}) + \CO_{M} ( h^{3/2} \vert \ln h \vert
) .  \label{r2}
\end{equation}

Now, we build an escape function outside of $\supp ( \chi_{1} )$
as in  \cite{Ma02_01}. Let ${\bf 1}_{(0,0)} \prec \chi_{0} \prec \chi_{1} \prec \chi_{2} \prec
\chi_{3} \prec \chi_{4} \prec \chi_{5}$ be
$C_{0}^{\infty} ( T^{*} (\R^{n}) , [0,1] )$ functions with $\varphi
\prec \chi_{4}$. We define $a_{3} = g ( \xi ) ( 1
- \chi_{3} (x, \xi )) x \cdot \xi$ where $g \in C_{0}^{\infty}
(\R^{n} )$ satisfies ${\bf 1}_{p^{-1} ( [E_{0} - \delta , E_{0} +
\delta ])} \prec g$. Using \cite[Lemma 3.1]{BoMi04_01}, we
can find a bounded,  $C^{\infty}$ function $a_{2}(x, \xi )$ such
that
\begin{equation}  \label{o1}
H_{p} a_{2} \geq \left\{ \begin{aligned}
&0 \qquad \text{for all } (x, \xi ) \in p^{-1} ([E_{0} - \delta , E_{0} + \delta ]) , \\
&1 \qquad \text{for all } (x, \xi ) \in \supp (\chi_{4} - \chi_{0} ) \cap p^{-1} ([E_{0} - \delta , E_{0} + \delta ]),
\end{aligned} \right.
\end{equation}
and we set $A_{2}= \oph ( a_{2} \chi_{5})$. We denote
\begin{equation}  \label{a16}
A = A_{1} + C_{2} \vert \ln h \vert A_{2} + \vert \ln h \vert A_{3},
\end{equation}
where $C_{2} >1$ will be fixed later. Now let $\widetilde{\psi} \in
C_{0}^{\infty} ( [E_{0} - \delta , E_{0} + \delta ] , [0,1])$ with
$\widetilde{\psi} =1$ near $E_{0}$. We recall that $\widetilde{\psi} (P)$ is a
classical pseudodifferential operator of class $\Psi^{0} (\< \xi \>^{-
\infty})$ with principal symbol $\widetilde{\psi} (p)$. Then, from
\eqref{r2}, we obtain
\begin{align}
-i \widetilde{\psi} (P) [A , P] \widetilde{\psi} (P) \geq& \varepsilon
h M^{-1} \widetilde{\psi} (P) \oph ( \chi_{1} ) \widetilde{\psi} (P) - C_{1} h \vert \ln
h \vert \widetilde{\psi} (P) \oph ( \varphi ) \widetilde{\psi} (P)   \nonumber  \\
&+ C_{2} h \vert \ln h \vert \oph \big( \widetilde{\psi}^{2} (p)
(\chi_{4} - \chi_{0} ) \big) +  C_{2} h \vert \ln h \vert \oph
\big( \widetilde{\psi}^{2} (p) a_{2} H_{p} \chi_{5} \big) \nonumber    \\ 
&+ h \vert \ln h \vert \oph \big( \widetilde{\psi}^{2} (p) ( \xi^{2} -x \cdot \nabla V ) ( 1 -
\chi_{3}) \big) \nonumber  \\ 
&+ h \vert \ln h \vert \oph \big( \widetilde{\psi}^{2}
(p) x \cdot \xi  H_{p} ( g \chi_{3} ) \big) + \CO (h M^{-2}) + \CO_{M} ( h^{3/2} \vert \ln h \vert ) .  \label{r11}
\end{align}
From {\bf (A1)}, we have $x \cdot \nabla V (x) \to 0$ as $x \to \infty$. In particular, if $\chi_{3}$ is equal to $1$ in a sufficiently large zone, we have
\begin{equation}  \label{r8}
\widetilde{\psi}^{2} (p) ( \xi^{2} -x \cdot \nabla V ) ( 1 - \chi_{3}) \geq E_{0} \widetilde{\psi}^{2} (p) ( 1 - \chi_{3}) .
\end{equation}
If $C_{2}>0$ is large enough, the G{\aa}rding inequality implies
\begin{equation}  \label{r10}
\begin{aligned}
C_{2} \oph \big( \widetilde{\psi}^{2} (p) (\chi_{4} - \chi_{0} ) \big)
- C_{1} \oph \big( \widetilde{\psi}^{2} (p) \varphi \big) + & \oph
\big( \widetilde{\psi}^{2} (p) x \cdot \xi  H_{p} ( g \chi_{3} ) \big)  \\
&\geq \oph \big( \widetilde{\psi}^{2} (p) (\chi_{4} - \chi_{0} ) \big) + \CO (h) .
\end{aligned}
\end{equation}
As in \cite{Ma02_01}, we take $\chi_{5} (x) = \widetilde{\chi}_{5} (\mu x)$ with $\mu$ small and $\widetilde{\chi}_{5} \in C_{0}^{\infty} ( [E_{0} - \delta , E_{0} + \delta ] , [0,1])$. Since $a_{2}$ is bounded, we get
\begin{equation}
\big\vert C_{2} \widetilde{\psi}^{2} (p) a_{2} H_{p} \chi_{5} \big\vert \leq \mu C_{2} \Vert a_{2} \Vert_{L^{\infty}} \Vert H_{p} \widetilde{\chi}_{5} \Vert_{L^{\infty}} \lesssim \mu .
\end{equation}
Therefore, if $\mu$ is small enough, \eqref{r8} implies
\begin{equation}  \label{r9}
\oph \big( \widetilde{\psi}^{2} (p) ( \xi^{2} -x \cdot \nabla V ) ( 1 - \chi_{3}) \big) + C_{2} \oph \big( \widetilde{\psi}^{2} (p) a_{2} H_{p} \chi_{5} \big) \geq \frac{E_{0}}{2} \oph \big( \widetilde{\psi}^{2} (p) ( 1 - \chi_{3}) \big) .
\end{equation}
Then \eqref{r11}, \eqref{r10}, \eqref{r9} and the G{\aa}rding inequality give 
\begin{align}
-i \widetilde{\psi} (P) [A , P] \widetilde{\psi} (P) \geq & \varepsilon
h M^{-1} \oph \big( \widetilde{\psi}^{2} (p) \chi_{1} \big) + h \vert \ln h \vert \oph \big( \widetilde{\psi}^{2} (p) (\chi_{4} - \chi_{0} ) \big)   \nonumber    \\ 
&+ \frac{E_{0}}{2} h \vert \ln h \vert \oph \big( \widetilde{\psi}^{2} (p) ( 1 - \chi_{3}) \big) + \CO (h M^{-2}) + \CO_{M} ( h^{3/2} \vert \ln h \vert )    \nonumber   \\
\geq & \varepsilon h M^{-1} \oph \big( \widetilde{\psi}^{2} (p) \big) + \CO (h M^{-2}) + \CO_{M} ( h^{3/2} \vert \ln h \vert ) .  \label{r13}
\end{align}
Choosing $M$ large enough and ${\bf 1}_{E_{0}} \prec \psi \prec \widetilde{\psi}$, we have proved the

\begin{lemma}  \label{m1}\sl 
Let $M$ be large enough and $\psi \in C^{\infty}_{0} ( [ E_{0} - \delta , E_{0} + \delta
])$, $\delta >0$ small enough, with $\psi = 1$ near $E_{0}$. Then,
we have
\begin{equation}  \label{r14}
-i \psi (P) [A, P] \psi (P) \geq \varepsilon h^{-1} \psi^{2} (P)  .
\end{equation}
Moreover
\begin{equation}
[A, P] = \CO (h \vert \ln h \vert ).
\end{equation}
\end{lemma}

From the properties of the support of the $\chi_{j}$, we have
\begin{align}
[ [P, A] , A] =& [[P, A_{1}] , A_{1}] + C_{2} \vert \ln h \vert [[P, A_{1}] , A_{2}]  \nonumber  \\
&+ C_{2} \vert \ln h \vert [[P, A_{2}] , A_{1}] + C_{2}^{2} \vert \ln h \vert^{2} [[P, A_{2}] , A_{2}] + C_{2} \vert \ln h \vert^{2} [[P, A_{2}] , A_{3}]   \nonumber   \\
&+ C_{2} \vert \ln h \vert^{2} [[P, A_{3}] , A_{2}] + \vert \ln h \vert^{2} [[P, A_{3}] , A_{3}] + \CO (h^{\infty}) . \label{r20}
\end{align}
We also know that $P \in \Psi^{0} ( \< \xi \>^{2} )$, $A_{2} \in
\Psi^{0} ( \< \xi \>^{- \infty} )$ and $A_{3} \in \Psi^{0} ( \< x \> \< \xi \>^{- \infty} )$. Then, we can show that all the terms in \eqref{r20} with $j,k=2,3$ satisfy
\begin{equation}  \label{r21}
[[P, A_{j}] , A_{k}] \in \Psi^{0} (h^{2} ) .
\end{equation}
On the other hand,
\begin{equation}
[[P, A_{1}] , A_{2}] = U^{*} [ [ \widehat{P} , B_{1} ] , U A_{2} U^{*} ] U + \CO ( h^{\infty}) ,
\end{equation}
with $U A_{2} U^{*} \in \Psi^{0} (1)$. From \eqref{r4} -- \eqref{r16}, we have $[ \widehat{P} , B_{1} ] \in \Psi^{1/2} (h  \vert \ln h \vert )$ and then
\begin{equation}  \label{r22}
[[P, A_{1}] , A_{2}] = \CO ( h^{3/2} \vert \ln h \vert ) .
\end{equation}
The term $[[P, A_{2}] , A_{1}]$ gives the same type of contribution. It remains to study
\begin{equation}
[[P, A_{1}] , A_{1}] = U^{*} [ [ \widehat{P} , B_{1} ] , B_{1} ] U + \CO ( h^{\infty}) .
\end{equation}
Let $\widetilde{\chi}_{3} \in C_{0}^{\infty} ( T^{*}\R^{n}) , [0,1] )$ with $\widetilde{\chi}_{2} \prec \widetilde{\chi}_{3}$ and
\begin{equation}
f = \Big( \ln \Big\< \frac{y}{\sqrt{h M}} \Big\> - \ln \Big< \frac{\eta}{\sqrt{h M}} \Big\> \Big) \widetilde{\chi}_{3} (y, \eta ) \in S^{1/2} (\vert \ln h \vert ).
\end{equation}
Then, with a remainder $r_{M} \in S^{1/2}(1)$ which differs from line to line,
\begin{align}
i [ \widehat{P} , B_{1} ] =& h \oph \big( f \{ \widetilde{\chi}_{2} , \widehat{p}_{0} \} + c_{2} \big) - h^{3/2} \vert \ln h \vert \oph ( r_{M} )   \nonumber  \\
=& h \oph ( f ) \oph ( \{ \widetilde{\chi}_{2} , \widehat{p}_{0} \} ) + h \oph ( c_{2} ) + h^{3/2} \vert \ln h \vert \oph ( r_{M} ) .
\end{align}
In particular, since $[ \widehat{P} , B_{1} ] \in \Psi^{1/2} (h  \vert \ln h \vert )$, $c_{2} \in S^{1/2}(1)$ and $f \in S^{1/2} (\vert \ln h \vert )$,
\begin{align}
[ [ \widehat{P} , B_{1} ] , B_{1} ] =& [ [ \widehat{P} , B_{1} ] , \oph ( f \widetilde{\chi}_{2}) ]     \nonumber  \\
=& -i h [ \oph ( f ) \oph ( \{ \widetilde{\chi}_{2} , \widehat{p}_{0} \} ) , \oph ( f \widetilde{\chi}_{2}) ] -i h [ \oph ( c_{2} ) , \oph ( f \widetilde{\chi}_{2}) ]   \nonumber  \\
&+ \CO (h^{3/2} \vert \ln h \vert^{2} )    \nonumber   \\
=& -i h [ \oph ( f ) \oph ( \{ \widetilde{\chi}_{2} , \widehat{p}_{0} \} ) , \oph ( f ) \oph ( \widetilde{\chi}_{2}) ] + \CO (h \vert \ln h \vert )    \nonumber  \\
=& -i h \oph ( f ) [ \oph ( \{ \widetilde{\chi}_{2} , \widehat{p}_{0} \} ) , \oph ( f ) ] \oph ( \widetilde{\chi}_{2})    \nonumber  \\
&- i h [ \oph ( f ) , \oph ( f ) ] \oph ( \{ \widetilde{\chi}_{2} , \widehat{p}_{0} \} )  \oph ( \widetilde{\chi}_{2})  \nonumber \\
&-i h \oph ( f ) \oph ( f ) [ \oph ( \{ \widetilde{\chi}_{2} , \widehat{p}_{0} \} ) , \oph ( \widetilde{\chi}_{2}) ]   \nonumber  \\
&-i h \oph ( f ) [ \oph ( f ) , \oph ( \widetilde{\chi}_{2}) ] \oph ( \{ \widetilde{\chi}_{2} , \widehat{p}_{0} \} ) + \CO (h \vert \ln h \vert )   \nonumber  \\
=& \CO (h \vert \ln h \vert ) .   \label{r23}
\end{align}
From \eqref{r20}, \eqref{r21}, \eqref{r22} and \eqref{r23}, we get
\begin{equation}  \label{r24}
[ [P, A] , A] = \CO (h \vert \ln h \vert ) .
\end{equation}
As a matter of fact, using \cite{bfrz}, one can show that $[ [P, A] , A] = \CO
(h)$. Now we can use the following proposition which is an adaptation of the limiting absorption principle of Mourre \cite{Mo81_01} (see also \cite[Theorem 4.9]{CyFrKiSi87_01}, \cite[Proposition 2.1]{HiNa89_01} and \cite[Theorem 7.4.1]{AmBoGe96_01}).

\begin{proposition} \sl  \label{p1}
Let $(P , D (P))$ and $( \CA , D (\CA ) )$ be self-adjoint operators on a separable Hilbert space $\CH$. Assume the following assumptions:
\begin{enumerate}
\item $P$ is of class $C^{2} ( \CA )$. Recall that $P$ is of class $C^{r} (\CA )$ if there exists $z \in \C \setminus \sigma (P)$ such that
\begin{equation} \label{c1}
\R \ni t \to e^{i t \CA} ( P -z)^{-1} e^{-i t \CA},
\end{equation}
is $C^{r}$ for the strong topology of ${\mathcal L} ( \CH )$.
\item The form $[ P, \CA ]$ defined on $D ( \CA ) \cap
D (P)$ extends to a bounded operator on $\CH$ and
\begin{equation}  \label{r38}
\Vert [ P, \CA] \Vert \lesssim \beta .
\end{equation} \label{y15}
\item The form $[[ P , \CA ] , \CA ]$ defined on $D ( \CA )$ extends to a
bounded operator on $\CH$ and
\begin{equation}  \label{r41}
\Vert [[P, \CA ] , \CA ] \Vert \lesssim \gamma .
\end{equation}
\item There exist a compact interval $I \subset \R$ and $g \in C_{0}^{\infty} (\R )$ with ${\bf 1}_{I} \prec
g$ such that
\begin{equation}
i g (P) [ P , \CA ] g (P) \gtrsim \gamma g^{2} (P).
\end{equation} \label{y16}
\item $\beta^{2} \lesssim \gamma \lesssim 1$. \label{r39}
\end{enumerate}
Then, for all $\alpha >1/2$, $\lim_{\varepsilon \to 0}  \< \CA \>^{- \alpha} (P -E \pm i \varepsilon )^{-1} \< \CA \>^{-
\alpha}$ exists and 
\begin{equation}
\big\Vert \< \CA \>^{- \alpha} (P -E \pm i0 )^{-1} \< \CA \>^{-
\alpha} \big\Vert \lesssim \gamma^{-1} ,
\end{equation}
uniformly for $E \in I$.
\end{proposition}

\begin{remark}  \label{y17}
From Theorem 6.2.10 of \cite{AmBoGe96_01}, we have the following useful characterization of the regularity $C^{2} (\CA )$. Assume that \eqref{y15} and \eqref{y16} hold. Then, $P$ is of class $C^{2} ( \CA )$ if and only if, for some $z \in \C \setminus \sigma (P)$, the set $\{ u \in D (\CA ); \ (P-z)^{-1} u \in D (\CA ) \text{ and } (P - \overline{z})^{-1} u \in D (\CA ) \}$ is a core for $\CA$.
\end{remark}

\begin{proof}
The proof follows the work of Hislop and Nakamura \cite{HiNa89_01}. For $ \varepsilon >0$, we define $M^{2} = i g (P) [ P ,
\CA ] g (P)$ and $G_{\varepsilon} (z) = (P - i \varepsilon M^{2} -
z)^{-1}$ which is analytic for $\re z \in I$ and $\im z >0$. Following
\cite[Lemma 4.14]{CyFrKiSi87_01} with {\eqref{c1})}, we get
\begin{equation}  \label{r36}
\Vert g (P) G_{\varepsilon} (z) \varphi \Vert \lesssim ( \varepsilon
\gamma)^{-1/2} \vert ( \varphi , G_{\varepsilon} (z) \varphi )
\vert^{1/2},
\end{equation}
\begin{equation}  \label{r37}
\Vert (1 - g (P) )  G_{\varepsilon} (z) \Vert \lesssim 1 +
\varepsilon \beta \Vert G_{\varepsilon} (z) \Vert ,
\end{equation}
and then
\begin{equation}  \label{r35}
\Vert G_{\varepsilon} (z) \Vert \lesssim ( \varepsilon
\gamma)^{-1} ,
\end{equation}
for $\varepsilon < \varepsilon_{0}$ with $\varepsilon_{0}$ small
enough, but independent on $\beta , \gamma$.

As in \cite{HiNa89_01}, let $D_{\varepsilon} = ( 1 + \vert \CA \vert
)^{-\alpha} (1 + \varepsilon \vert \CA \vert )^{\alpha -1}$ for
$\alpha \in ] 1/2 ,1]$ and $F_{\varepsilon} (z) = D_{\varepsilon}
G_{\varepsilon} (z) D_{\varepsilon}$. Of course, from \eqref{r35},
\begin{equation}  \label{r42}
\Vert F_{\varepsilon} (z) \Vert \lesssim ( \varepsilon \gamma)^{-1} ,
\end{equation}
and \eqref{r36} and \eqref{r37} with $\varphi = D_{\varepsilon} \psi$
give
\begin{equation}   \label{r40}
\Vert G_{\varepsilon} (z) D_{\varepsilon} \Vert \lesssim 1 + (
\varepsilon \gamma)^{-1/2} \Vert F_{\varepsilon} \Vert^{1/2} .
\end{equation}

The derivative of $F_{\varepsilon} (z)$ is given by (see \cite[Lemma 4.15]{CyFrKiSi87_01})
\begin{equation}
\partial_{\varepsilon} F_{\varepsilon} (z) = i D_{\varepsilon}
G_{\varepsilon} M^{2} G_{\varepsilon} D_{\varepsilon}= Q_{0} + Q_{1} + Q_{2} + Q_{3},
\end{equation}
with
\begin{align}
Q_{0} =&  ( \alpha -1) \vert \CA \vert ( 1 + \vert \CA \vert
)^{-\alpha} (1 + \varepsilon \vert \CA \vert )^{\alpha -2}
G_{\varepsilon} (z) D_{\varepsilon}     \nonumber  \\
&+ ( \alpha -1)  D_{\varepsilon} G_{\varepsilon} (z)  \vert \CA \vert ( 1 + \vert \CA \vert
)^{-\alpha} (1 + \varepsilon \vert \CA \vert )^{\alpha -2}   \\
Q_{1} =& D_{\varepsilon} G_{\varepsilon} (1- g(P)) [ P , \CA ] (1-
g(P)) G_{\varepsilon} D_{\varepsilon}  \\
Q_{2} =& D_{\varepsilon} G_{\varepsilon} (1- g(P)) [ P , \CA ] g(P)
G_{\varepsilon} D_{\varepsilon} + D_{\varepsilon} G_{\varepsilon} g(P)
[ P , \CA ] (1- g(P)) G_{\varepsilon} D_{\varepsilon}  \\
Q_{3}  =& - D_{\varepsilon} G_{\varepsilon} [ P , \CA ] G_{\varepsilon} D_{\varepsilon}.
\end{align}
From \eqref{r40}, we obtain
\begin{equation}
\Vert Q_{0} \Vert \lesssim \varepsilon^{\alpha -1} \big( 1 + (
\varepsilon \gamma)^{-1/2} \Vert F_{\varepsilon} \Vert^{1/2} \big),
\end{equation}
and from \eqref{r38}, {\ref{r39})} of Proposition \ref{p1}, \eqref{r37} and \eqref{r35}, we get
\begin{equation}
\Vert Q_{1} \Vert \lesssim \gamma^{-1} .
\end{equation}
Using in addition \eqref{r40}, we obtain
\begin{equation}
\Vert Q_{2} \Vert \lesssim 1 + ( \varepsilon \gamma)^{-1/2} \Vert F_{\varepsilon} \Vert^{1/2} .
\end{equation}
Now we write $Q_{3} = Q_{4} + Q_{5}$ with
\begin{align}
Q_{4} &= - D_{\varepsilon} G_{\varepsilon} [ P -i \varepsilon M^{2} -z ,
\CA ] G_{\varepsilon} D_{\varepsilon}  \\
Q_{5} &= - i \varepsilon D_{\varepsilon} G_{\varepsilon} [  M^{2} ,
\CA ] G_{\varepsilon} D_{\varepsilon} .
\end{align}
For $Q_{4}$, we have the estimate 
\begin{equation}
\Vert Q_{4} \Vert \lesssim \varepsilon^{\alpha -1} \big(  1 + (
\varepsilon \gamma)^{-1/2} \Vert F_{\varepsilon} \Vert^{1/2} \big)
\end{equation}
On the other hand, \eqref{r38}, \eqref{r41} and {\ref{r39})} imply
\begin{equation}
\Vert [ M^{2} , \CA ] \Vert \lesssim \gamma.
\end{equation}
Then \eqref{r40} gives
\begin{equation}
\Vert Q_{5} \Vert \lesssim 1 + \Vert F_{\varepsilon} \Vert .
\end{equation}
Using the estimates on the $Q_{j}$, we get
\begin{equation}  \label{r43}
\Vert \partial_{\varepsilon} F_{\varepsilon} \Vert \lesssim
\varepsilon^{\alpha -1 } \big( \gamma^{-1} + (\varepsilon
\gamma)^{-1/2} \Vert F_{\varepsilon} \Vert^{1/2} + \Vert
F_{\varepsilon} \Vert\big) .
\end{equation}
Using \eqref{r42} and integrating \eqref{r41} $N$ times with respect to
$\varepsilon$, we get
\begin{equation}
\Vert  F_{\varepsilon} \Vert \lesssim
\gamma^{-1} \big( 1 + \varepsilon^{2 \alpha (1 - 2^{-N})-1} \big),
\end{equation}
so that, for $N$ large enough, 
\begin{equation}  \label{r44}
\limsup_{\delta \to 0} \sup_{E \in I} \Vert \< \CA \>^{- \alpha} (P -E \pm i \delta )^{-1} \< \CA \>^{-
\alpha} \Vert \lesssim \gamma^{-1}.
\end{equation}
Using, as in \cite{HiNa89_01}, that $z \mapsto F_{0} (z)$ is
H\"{o}lder continuous, we prove the existence of the limit $\lim_{\im
z \to 0} F_{0} (z)$ for $\re z \in I$ and the proposition follows from
\eqref{r44}.
\end{proof}

\noindent
From Lemma \ref{m1} and \eqref{r24}, we can apply Proposition \ref{p1} with $\CA = A / \vert \ln h \vert$,
$\beta = h$ and $\gamma = h / \vert \ln h
\vert$. Therefore we have the estimate
\begin{equation}  \label{r31}
\big\Vert \< \CA \>^{- \alpha} (P -E \pm i0 )^{-1} \< \CA \>^{-
\alpha} \big\Vert \lesssim h^{-1} \vert \ln h \vert ,
\end{equation}
for $E \in [ E_{0} - \delta , E_{0} + \delta ]$. As usual, we have
\begin{equation} \label{r30}
\Vert \< x \>^{-\alpha} \< \CA \>^{\alpha} \Vert  = \CO (1) ,
\end{equation}
for $\alpha \geq 0$. Indeed, \eqref{r30} is clear for $\alpha \in 2
\N$, and the general case follows by complex interpolation. Then,
\eqref{r31} and \eqref{r20} imply Theorem \ref{r1}.

\section{Representation of the Scattering Amplitude}
\label{sec4}

As in \cite{RoTa89_01}, our starting point for the computation of the scattering amplitude is the representation given by Isozaki and Kitada in \cite{IsKi85_01}. We recall briefly their formula, that they obtained writing parametrices for the wave operators $W_\pm$ as Fourier Integral Operators,  taking advantage of the well-known intertwining property $W_\pm P=P_0W_\pm$, $P=P_0+V$.
The wave operators are defined by
\begin{equation}\label{ze1}
W_\pm=s-\lim_{t\to \pm \infty} e^{itP/h}e^{-itP_0/h},
\end{equation}
where the limit exist thanks to the short-range assumption \ref{a1}. The scattering operator is  by definition $\CS=(W_+)^*W_-$, and the scattering matrix $\CS(E,h)$ is then given by the decompostion of 
$\CS$ with respect to the spectral measure of $P_0=-h^2\Delta$. 
Now we recall  briefly the discussion in \cite[Section 1,2]{RoTa89_01} (see also \cite{asr}), and we start  with some notations.

If $\Omega$ is an open subset of $T^{*}\R^{n}$ , we denote by 
 $A_{m}(\Omega)$ the class of symbols $a$ such that $(x, \xi)\mapsto a(x, \xi,
h)$ belongs to $C^{\infty}(\Omega)$ and
\begin{equation}
\left|\partial_{x}^{\alpha}\partial_{\xi}^{\beta}a(x, \xi)\right|\leq C_{\alpha\beta}\langle
x \rangle^{m-|\alpha|}\langle\xi\rangle^{-L}, \text{ for all }   L>0, (x, \xi)\in\Omega, (\alpha,
\beta)\in\N^{d}\times\N^{d}.
\end{equation}
We also denote by
\begin{equation}
\Gamma_{\pm}(R, d, \sigma)=\left\{(x, \xi)\in\R^{n}\times\R^{n}:|x|>R,
\frac{1}{d}<|\xi|<d, \pm \cos(x, \xi)>\pm\sigma\right\}
\end{equation}
 with $R>1$,
$d>1$, $\sigma\in(-1,
1)$, and $\cos(x, \xi)=\frac{\langle x, \xi\rangle}{|x|\;|\xi|}$, the
outgoing and incoming  
subsets of $T^*\R^{n}$, respectively.
Eventually, for $\alpha>\frac{1}{2}$, we denote the bounded operator $\CF_{0}(E, h):
L^{2}_{\alpha}(\R^{n})\rightarrow L^{2}(\S^{n-1})$ given by
\begin{equation}
\left(\CF_{0}(E, h)f\right)(\omega)=(2\pi
h)^{-\frac{n}{2}}(2E)^{\frac{n-2}{4}}\int_{\R^{n}}
e^{-\frac{i}{h}\sqrt{2E}\langle \omega, x\rangle}f(x)dx, E>0.
\end{equation}

Isozaki and Kitada have constructed
 phase functions
$\Phi_{\pm}$ and symbols $a_{\pm}$ and $b_{\pm}$ such that, for some $R_{0}>>0$, $1<d_4<d_3<d_2<d_1<d_0$, and
$0<\sigma_4<\sigma_3<\sigma_2<\sigma_1<\sigma_0<1$:
\begin{enumerate}
\item  $\Phi_{\pm}\in C^{\infty}(T^{*}\R^{n})$ solve the eikonal equation
\begin{equation}
\frac12\left|\nabla_{x}\Phi_{\pm}(x, \xi)\right|^{2}+V(x)=\frac12|\xi|^{2}
\end{equation}
in $(x, \xi)\in\Gamma_{\pm}(R_0, d_0, \pm\sigma_0)$, respectively.
\item $(x,\xi)\mapsto \Phi_{\pm}(x, \xi)-x\cdot\xi\in
A_{0}\left(\Gamma_{\pm}(R_0, d_0, \pm\sigma_{0})\right).$
\item  For all $(x, \xi)\in T^{*}\R^{n}$
\begin{equation}
\left|\frac{\partial^{2}\Phi_{\pm}}{\partial_{x_{j}}\partial_{\xi_{k}}}(x,
\xi)-\delta_{jk}\right|<\varepsilon(R_0),
\end{equation}
where $\delta_{jk}$ is the Kronecker delta and $\varepsilon(R_0)\to 0$ as $R_{0}\to+\infty.$
\item  $a_{\pm}\sim\sum_{j=0}^{\infty}h^j a_{\pm j}$, where $a_{\pm j}\in
A_{-j}(\Gamma_{\pm}(3R_0, d_1, \mp\sigma_1))$, $\supp a_{\pm j}\subset\Gamma_{\pm}(3R_0,
d_1, \mp\sigma_1)$, $a_{\pm j}$ solve
\begin{equation}\label{aeq1}
\langle \nabla_{x}\Phi_{\pm}\vert\nabla_{x}a_{\pm 0}\rangle
+\frac12\left(\Delta_{x}\Phi_{\pm}\right)a_{\pm 0}=0
\end{equation}
\begin{equation}\label{aeq2}
\langle \nabla_{x}\Phi_{\pm}\vert \nabla_{x}a_{\pm
j}\rangle+\frac12\left(\Delta_{x}\Phi_{\pm}\right)a_{\pm j}=\frac{i}2\Delta_{x}a_{\pm
j-1}, j\geq 1,
\end{equation}
with the conditions at infinity
\begin{equation}\label{condinfty}
a_{\pm 0}\to 1, a_{\pm j}\to 0, j\geq 1, \text{ as } |x|\to \infty.
\end{equation}
in $\Gamma_{\pm}(2R_0, d_2, \mp\sigma_2),$
and solve \eqref{aeq1} and \eqref{aeq2} in $\Gamma_{\pm}(4R_0, d_1, \mp\sigma_2).$
\item $b_{\pm}\sim\sum_{j=0}^{\infty}h^j b_{\pm j},$ where $b_{\pm j}\in
A_{-j}(\Gamma_{\pm}(5R_0, d_3, \pm\sigma_{4}),$ $\supp b_{\pm j}\subset\Gamma_{\pm}(5R_0,
d_3, \pm\sigma_{4}),$ $b_{\pm j}$ solve \eqref{aeq1} and \eqref{aeq2} with the conditions
at infinity \eqref{condinfty} in $\Gamma_{\pm}(6R_0, d_4, \pm\sigma_3),$ and solve
\eqref{aeq1} and \eqref{aeq2} in $\Gamma_{\pm}(6R_0, d_3, \pm\sigma_{3}).$
\end{enumerate}

For a symbol $c$ and a phase function $\varphi$, we denote by $I_{h}(c, \varphi)$ the oscillatory integral
\begin{equation}
I_{h}(c, \varphi)=\frac{1}{(2\pi h)^{n}}\int_{\R^{n}} e^{\frac{i}{h}(\varphi(x, \xi)-\langle y,
\xi\rangle)}c(x, \xi) d\xi
\end{equation}
and we set
\begin{equation}
\begin{aligned}
K_{\pm a}(h) & =P(h)I_{h}(a_{\pm}, \Phi_{\pm})-I_{h}(a_{\pm}, \Phi_{\pm})P_{0}(h),\\
K_{\pm b}(h) & =P(h)I_{h}(b_{\pm}, \Phi_{\pm})-I_{h}(b_{\pm}, \Phi_{\pm})P_{0}(h).
\end{aligned}
\end{equation}

The operator $\CT(E, h)$ for $E\in]\frac{1}{2d_{4}^{2}},
\frac{d_{4}^{2}}{2}[$ is then given by (see \cite[Theorem 3.3]{IsKi85_01})
\begin{equation}
\CT(E, h)=\CT_{+1}(E, h)+\CT_{-1}(E, h)-\CT_{2}(E, h),
\end{equation}
where
\begin{equation}
\CT_{\pm 1}(E, h)=\CF_{0}(E, h)I_{h}(a_{\pm}, \Phi_{\pm})^{*}K_{\pm b}(h)
\CF_{0}^{*}(E, h)
\end{equation}
and
\begin{equation}
\CT_{2}(E, h)=\CF_{0}(E,
h)K_{+a}^{*}(h)\CR(E+i0,
h)\left(K_{+b}(h)+K_{-b}(h)\right)\CF_{0}^{*}(E, h),
\end{equation}
where we denote from now on $\CR (E \pm i 0,h )=(P-(E\pm i0))^{-1}$.

Writing explicitly their kernel, it is easy to see, by a non-stationary phase argument, that the operators $\CT_{\pm 1}$ are $\CO(h^\infty)$ when $\theta\neq\omega$. Therefore we have
\begin{equation}
\CA(\omega,\theta,E,h)=-c(E)h^{(n-1)/2}\CT_2(\omega,\theta,E,h) +\CO(h^\infty),
\end{equation}
where $c(E)$ is given in \eqref{intro1}.

As in \cite{RoTa89_01}, we shall use our resolvent  estimate (Theorem \ref{r1})  in a particular form. It was noticed by L. Michel in \cite[Proposition 3.1]{Mi04_01} that, in the present trapping case,  the following proposition follows easily from the corresponding one in the non-trapping setting. Indeed, if $\varphi$ is a compactly supported smooth function, it is clear that $\widetilde{P} = -h^2 \Delta + (1- \varphi (x/R)) V(x)$ satisfies the non-trapping assumption for $R$ large enough, thanks to the decay of $V$ at $\infty$. Writing \cite[Lemma 2.3]{RoTa89_01} for $\widetilde{P}$, one gets the

\begin{proposition}\sl \label{est}
Let $\omega_{\pm} \in A_{0}$ has support in $\Gamma_{\pm} (R,d,\sigma_{\pm})$ for $R > R_{0}$. For $E \in [E_0- \delta , E_0 +  \delta ]$, we have
\begin{itemize}
\item[(i)] For any $\alpha >1/2$ and $M >1$, then, for any $\varepsilon>0$,
\begin{equation}
\Vert \CR (E \pm i 0,h ) \omega_{\pm} (x, h D_{x}) \Vert_{- \alpha + M , - \alpha } = {\mathcal O} (h^{-3-\varepsilon}) .
\end{equation}

\item[(ii)] If $\sigma_{+} > \sigma_{-} $, then for any $\alpha \gg 1$,
\begin{equation}
\Vert \omega_{\mp} (x, h D_{x}) \CR (E \pm i 0,h ) \omega_{\pm} (x, h D_{x}) \Vert_{- \alpha + \delta , - \alpha } = {\mathcal O} (h^{\infty}) .
\end{equation}

\item[(iii)] If $\omega (x, \xi ) \in A_{0}$ has support in $\vert x \vert < (9/10) R$, then for any $\alpha \gg 1$
\begin{equation}
\Vert \omega (x, h D_{x}) \CR (E \pm i 0,h ) \omega_\pm (x, h D_{x}) \Vert_{- \alpha + \delta , - \alpha } = {\mathcal O} (h^{\infty}) .
\end{equation}
\end{itemize}
\end{proposition}

Then we can follow line by line the discussion after Lemma 2.1 of D. Robert and H. Tamura, and we obtain (see Equations 2.2-2.4 there):

\begin{equation}  \label{q6}
\CA(\omega,\theta, E,h)=c(E)h^{-(n+1)/2}\bra \CR(E+i0,h)g_- e^{i\psi_-/h}, g_+ e^{i\psi_+/h}\ket + \CO (h^\infty),
\end{equation}
where
\begin{equation} \label{d1}
g_\pm=e^{-i\psi_\pm/h}[\chi_\pm, P]a_\pm(x,h) e^{i\psi_\pm/h},
\end{equation}
and 
\begin{equation}  \label{q22}
\psi_+(x)=\Phi_+(x, \sqrt{2E}\theta),\qquad  \psi_-(x)=\Phi_-(x, \sqrt{2E}\omega).
\end{equation}
Moreover the functions $\chi_\pm$ are $C^\infty_0 (\R^{n})$ functions  such that $\chi_\pm =1$ on some ball $B(0,R_\pm)$, with  support in $B(0,R_\pm+1)$.

Eventually, we shall need the following version of Egorov's Theorem, which is also used in Robert and Tamura's paper.

\begin{proposition}[{\cite[Proposition 3.1]{RoTa89_01}}]  \label{b1} \sl
Let $\omega (x, \xi ) \in A_{0}$ be of compact support. Assume that, for some fixed $t\in \R$,  $\omega_{t}$ is a function in $A_{0}$ which vanishes in a small neighborhood of
\begin{equation*}
\{ (x, \xi ) ; \ (x, \xi ) = \exp (t H_{p} ) (y, \eta ), \ (y, \eta ) \in \supp \omega \} .
\end{equation*}
Then
\begin{equation*}
\Vert \oph (\omega_{t} ) e^{-i t P /h} \oph (\omega ) \Vert_{- \alpha , \alpha} = \CO (h^{\infty}),
\end{equation*}
for any $\alpha \gg 1$. Moreover,  the order relation is uniform in $t$ when $t$ ranges over a compact interval of $\R$.
\end{proposition}

In the three next sections, we prove Theorem \ref{main} using \eqref{q6}. We set
\begin{equation}
u_-=u_-^h=\CR(E+i0,h)g_- e^{i\psi_-/h},
\end{equation}
and our proof consists in the computation of $u_-$ in different region of the phase
space, following the classical trajectories $\gamma_j^\infty$, or $\gamma_k^-$ and
$\gamma_\ell^+$. It is important to notice that we have  $(P-E)u_-=0$ out
of the support of $g_-$.

\section{Computations before the critical point}
\label{sec5}

\Subsection{Computation of $u_-$ in the incoming region}

We start with the computation of $u_{-}$ in an incoming region which contains the microsupport of $g_{-}$. Notice that, thanks to Theorem \ref{r1},  $\< x \>^{- \alpha } u_- (x)$ is a semiclassical family of distributions for $\alpha>1/2$.

\begin{lemma}\label{b2}  \sl Let $P$ be a Schr\"{o}dinger operator as in \eqref{q28} satisfying only  \ref{a1}. Suppose that $I$ is a compact interval of $]0 , + \infty[$, and $d>0$ is such that  $I \subset ] \frac{1}{2 d^{2}} ,  \frac{d^{2}}{2} [$. Suppose also that $0< \sigma_{+} <1$, $R$ is large enough and $K \subset T^{*} \R^{n}$ is a compact subset of $\{ \vert x \vert > R \} \cap p^{-1}(I)$. Then there exists $T_0>0$ such that, if $\rho \in K$ and $t>T_0$,
\begin{equation}
\exp (t H_{p}) (\rho ) \in \Gamma_+(R/2,d,\sigma_+) \cup (B(0,R/2)\times \R^n ).
\end{equation}
\end{lemma}

\begin{proof}
Let $\delta >0$. From the construction of C. G\'erard and J. Sj\"{o}strand \cite{GeSj87_01}, there exists a function $G (x, \xi) \in C^{\infty} (\R^{2n})$ such that, 
\begin{align}
&(H_{p} G) (x, \xi) \geq 0 \quad \text{for all $(x, \xi)\in p^{-1} (] \frac{1}{2 d^{2}} ,  \frac{d^{2}}{2} [)$}, \\
&(H_{p} G) (x, \xi) > 2 E (1- \delta)  \quad \text{for $\vert  x \vert  > R_{0}$ and $p (x, \xi ) = E\in ] \frac{1}{2 d^{2}} ,  \frac{d^{2}}{2} [$},  \\
&G(x, \xi ) = x \cdot \xi \quad \text{for $\vert x\vert  > R_{0}$}.
\end{align}
Let $\rho \in K$, and $\gamma (t) = ( x(t), \xi (t) ) =
\exp (t H_{p}) (\rho)$ be the corresponding Hamiltonian curve. We distinguish between 2 cases:

\medskip
1) {\sl For all $t>0$, we have $\vert  x(t) \vert  > R_{0}$}.

\noindent
Then $G ( \gamma (t) ) > 2 E (1-\delta) t + G (\rho)$ and, for $t > T_{1}$ with $T_{1}$ large enough, 
\begin{equation}\label{ajoute1}
G ( \gamma (t)
) > 2 \sup_{\fract{x \in B(0, R_{0})}{p (x, \xi ) \in I}} G(x, \xi).
\end{equation}
By continuity, there exists a neighborhood ${\mathcal U}$ of $\gamma$ such that, for all
$\widetilde{\gamma} \in {\mathcal U}$, we have
\begin{equation}
G ( \widetilde{\gamma} (T_{1}) ) > \sup_{\fract{x \in B(0, R_{0})}{p (x, \xi ) \in I}} G(x, \xi) .
\end{equation}
Since $G$ is non-decreasing on $\widetilde{\gamma} (t)$, we have $\vert \widetilde{x} (t) \vert > R_{0}$ for
all $t>T_{1}$, and then
\begin{equation}  \label{w1}
G (\widetilde{\gamma} (t)) > 2 E (1-\delta )  ( t - T_{1} ) + G (\widetilde{\gamma}
(T_{1})) > 2 E (1-\delta ) t - C.
\end{equation}
On the other hand, by uniformly finite propagation, we have $\vert 
\widetilde{x} (t) \vert < \sqrt{2 E} (1+\delta ) t + C$. From \eqref{w1}, we get $\vert \widetilde{x} (t) \vert > \frac1C t -C$ for all $\widetilde{\gamma} \in {\mathcal U}$, and then $\vert \widetilde{\xi} (t) \vert = \sqrt{2 E}+ o_{t \to \infty} (1)$. In particular, the previous estimates gives
\begin{equation}
\vert x(t)\vert  > R/2,
\end{equation}
\begin{equation}
\cos \big( \widetilde{x}, \widetilde{\xi} \big) (t) > \frac{2 E (1-\delta ) t - C}{(\sqrt{2 E} (1+\delta ) t + C) (\sqrt{2 E}+ o_{t \to \infty} (1) )} = \frac{1-\delta}{1+\delta} + o_{t \to \infty} (1)
> 1-3 \delta ,
\end{equation}
for $t > T_{0}$ with $T_{0}$ large enough but independent on $\widetilde{\gamma} \in {\mathcal U}$. Thus, for $t > T_{0}$ and $\widetilde{\gamma} \in {\mathcal U}$, we have
\begin{equation}  \label{w2}
\widetilde{\gamma} (t) \in \Gamma_{+} (R/2, d, \sigma_{+}),
\end{equation}
with $\sigma_{+} = 1- 3 \delta$.

\medskip
2) {\sl There exist $T_{2} >0$ such that $\vert  x(T_{2}) \vert  = R_{0}$}.

\noindent
Then there exists ${\mathcal V}$ a neighborhood of $\gamma$ such that,
for all $\widetilde{\gamma} \in {\mathcal V}$, we have $\vert  \widetilde{x}
(T_{2}) \vert  < 2 R_{0}$. Let $t > T_{2}$.

a) If $\vert \widetilde{x} (t) \vert \leq R/2$, then $\widetilde{\gamma}(t) \in B(0,R/2)\times \R^n$.

b) Assume now $\vert  \widetilde{x} (t) \vert  > R/2$. Denote
by $T_{3}$ ($> T_{2}$) the last time (before $t$) such that $\vert  \widetilde{x} (T_{3})
\vert = 2 R_{0}$. Then
\begin{align}
G ( \widetilde{\gamma} (t) ) >& 2 E (1- \delta) (t-T_{3}) + G (\widetilde{\gamma}
(T_{3}))  \\
>& 2 E (1- \delta) (t-T_{3}) - C,
\end{align}
where $C$ depend only on $R_{0}$. On the other hand, the have $\vert \widetilde{x} (t) \vert < \sqrt{2 E} (1+ \delta ) (t-T_{3}) + C$ (where the constant $C$ depend only on $R_{0}$). Then,
\begin{equation}
t -T_{3} > \frac{\vert \widetilde{x} (t) \vert}{\sqrt{2 E} (1+ \delta )} - \frac{C}{\sqrt{2 E} (1+ \delta )} ,
\end{equation}
\begin{equation}
\vert \widetilde{\xi} (t) \vert = \sqrt{2 E} + o_{R \to \infty} (1) ,
\end{equation}
\begin{align}
\cos \big( \widetilde{x}, \widetilde{\xi} \big) (t) >& \frac{2 E (1- \delta ) \vert \widetilde{x} (t) \vert}{\vert \widetilde{x} (t) \vert (\sqrt{2 E}(1 + \delta ) ) ( \sqrt{2 E} + o_{R \to \infty} (1) )} + \CO (R^{-1})   \nonumber \\
>& \frac{1- \delta}{1 + \delta} + o_{R \to \infty} (1) > 1- 2 \delta + o_{R \to \infty} (1).
\end{align}
So, if $R$ is large enough, $\widetilde{\gamma} (t) \in \Gamma_{+}
(R/2, d, \sigma_{+})$, $\sigma_{+} = 1 - 3 \delta$.

Then a) and b) imply that, for all $\widetilde{\gamma} \in {\mathcal
V}$ and $t> T_{0} := T_{2}$, we have
\begin{equation}  \label{w3}
\widetilde{\gamma} (t) \in \Gamma_{+} (R/2, d, \sigma_{+}) \cup ( B (0, R/2) \times \R^{n} ) .
\end{equation}

The lemma follows from \eqref{w2}, \eqref{w3} and a compactness argument.
\end{proof}

Recall that the microsupport of $g_- (x) e^{i\psi_- (x)/h} \in C_{0}^{\infty} (\R^{n})$ is contained in $\Gamma_{-} (R_{-} , d_{1} , \sigma_{1} )$. Let $\omega_{-} (x, \xi ) \in A_{0}$ with $\omega_{-} =1$ near $\Gamma_{-} (R_{-} /2 , d_{1} , \sigma_{1} )$ and $\supp (\omega_{-}) \subset \Gamma_{-} (R_{-} /3 , d_{0} , \sigma_{0} )$. Using the identity
\begin{equation}
u_{-} = \frac{i}{h}\int_0^{T} e^{-it(P-E)/h} (g_- e^{i\psi_-/h})dt + \CR(E+i0,h) e^{-i T (P-E)/h} (g_- e^{i\psi_-/h}) ,
\end{equation}
and Proposition \ref{est}, Proposition \ref{b1} and Lemma \ref{b2}, we get
\begin{equation}  \label{p1m}
\oph ( \omega_{-} ) u_{-} = \oph ( \omega_{-} ) \frac{i}{h}\int_0^{T} e^{-it(P-E)/h} (g_- e^{i\psi_-/h})dt + \CO (h^{\infty}) ,
\end{equation}
for some $T>0$ large enough. In particular,
\begin{equation}
\MS (\oph ( \omega_{-} ) u_{-} ) \subset \Lambda_{\omega}^{-} \cap ( B(0,R_{-}+1) \times \R^{n} ) .
\end{equation}

\Subsection{Computation of $u_-$ along $\gamma^{-}_{k}$}  \label{q9}

Now we want to compute $u_-$ microlocally along a trajectory $\gamma^{-}_{k}$. We recall that  $\gamma^{-}_{k}$ is a bicharacteristic curve $(x^{-}_{k} (t),\xi^{-}_{k} (t))$ such that $(x^{-}_{k} (t), \xi^{-}_{k} (t))\to (0,0)$ as $t\to +\infty$, and such that, as $t\to -\infty$,
\begin{equation}
\begin{array}{l}
\vert x^{-}_{k} (t)-\sqrt{2E_{0}}\omega t-z^{-}_{k} \vert\to 0 , \\[6pt]
\vert \xi^{-}_{k} (t)-\sqrt{2E_{0}}\omega\vert \to 0 .  
\end{array}\label{ma4}
\end{equation}
If $R_{-}$ is large enough, $a_{-}$ solves \eqref{aeq1} and \eqref{aeq2} microlocally near $\gamma^{-}_{k} \cap \MS (g_{-} e^{i \psi_{-} /h})$. In particular, microlocally near $\gamma^{-}_{k} \cap \Gamma_{-} (R_{-} /2 , d_{1} , \sigma_{1} ) \cap ( B (0, R_{-}) \times \R^{n} )$, $u_{-}$ is given by \eqref{p1m} and
\begin{align}
u_{-} =& \frac{i}{h}\int_0^{T} e^{-it(P-E)/h} ( [\chi_{-} ,P] a_- e^{i\psi_-/h})dt + \CO (h^{\infty})  \nonumber \\
=& \frac{i}{h}\int_0^{T} e^{-it(P-E)/h} ( \chi_{-} (P-E) a_- e^{i\psi_-/h})dt \nonumber \\
&+ \frac{i}{h}\int_0^{T} e^{-it(P-E)/h} ( (P-E) \chi_{-} a_- e^{i\psi_-/h})dt  + \CO (h^{\infty})  \nonumber \\
=& \frac{i}{h}\int_0^{T} (P-E) e^{-it(P-E)/h} ( \chi_{-} a_- e^{i\psi_-/h})dt  + \CO (h^{\infty})  \nonumber \\
=& (P-E) \CR(E+i0,h) a_- e^{i\psi_-/h} + \CO (h^{\infty})  \nonumber \\
=& a_- e^{i\psi_-/h} + \CO (h^{\infty}).  \label{q1}
\end{align}
Now, using \eqref{q1},  and the fact that $u_-$ is a semiclassical distribution satisfying
\begin{equation}
(P-E) u_{-} = 0,
\end{equation}
we can compute $u_{-}$ microlocally near $\gamma_{k}^{-} \cap B
(0,R_{-})$ using  Maslov's theory (see \cite{MaFe81_01} for more
details). Moreover, it is proved in Proposition \ref{a15} (see also \cite[Lemma 5.8]{bfrz}) that the Lagrangian
manifold $\Lambda_{\omega}^{-}$ has a nice projection with respect to
$x$ in a neighborhood of $\gamma_{k}^{-}$ close to $(0,0)$. Then, in
such a neighborhood, $u_{-}$ is given by
\begin{equation}  \label{q2}
u_{-} (x) = a_{-} (x,h) e^{-i \nu_{k}^{-} \pi /2} e^{i \psi_{-} (x) /h} ,
\end{equation}
where $\nu_{k}^{-}$ denotes the Maslov index of $\gamma_{k}^{-}$, and $\psi_{-}$ satisfies the usual eikonal equation
\begin{equation}  \label{q3}
p (x , \nabla \psi_{-} ) = E_{0} .
\end{equation}
Here, to the contrary of \eqref{q22}, we have written  $E = E_{0} + z h$ with $z = \CO (1)$, and we choose to work with $z$ in the amplitudes instead of the phases.
As usual, we have
\begin{equation}\label{ma5}
\partial_t(\psi_{-}(x_{k}^{-}(t)))= \nabla\psi_{-}(x_{k}^{-}(t))\cdot \partial_tx_{k}^{-}(t)=\nabla\psi_{-}(x_{k}^{-}(t))\cdot\xi_{k}^{-}(t)=\vert\xi_{k}^{-}(t)\vert^2, 
\end{equation}
so that
\begin{equation}\label{ma6}
\psi_{-}(x_{k}^{-}(t))=\psi_{-}(x_{k}^{-}(s))+\int_s^t\vert\xi_{k}^{-}(u)\vert^2 du
\end{equation}
We also have $\psi_{-}(x_{k}^{-}(s))=(\sqrt {2E_{0}}\omega s+z_{k}^{-})\cdot \sqrt {2E_{0}}\omega +o(1)$ as $s\to -\infty$, and then
\begin{equation}\label{ma7}
\psi_{-}(x_{k}^{-}(t))=2E_{0}s+\int_s^t\vert\xi_{k}^{-}(u)\vert^2 du +o(1),\quad s\to -\infty.
\end{equation}
We have obtained in particular that
\begin{equation}\label{ma8}
\psi_{-}(x_{k}^{-}(t))=\int_{-\infty}^t \vert \xi_{k}^{-}(u) \vert^2 - 2E_{0} 1_{u<0} \, du=
\int_{-\infty}^{t} \frac{1}{2} \vert \xi_{k}^{-} (u) \vert^2 -V(x_{k}^{-}(u))+E_{0} \sgn(u) \, du.
\end{equation}

We turn to the computation of the symbol. The function $a_{-} (x,h) \sim \sum_{k=0}^{\infty} a_{-,k} (x) h^{k}$ satisfies the usual transport equations:
\begin{equation}
\left\{ \begin{aligned}  \label{q4}
&\nabla \psi_{-} \cdot \nabla a_{-,0} + \frac{1}{2}( \Delta
\psi_{-} -2 i z) a_{-,0}  = 0 , \\
&\nabla \psi_{-} \cdot \nabla a_{-,k} + \frac{1}{2}( \Delta
\psi_{-} - 2 i z) a_{-,k} = i \frac{1}{2}\Delta a_{-,k-1} , \quad k \geq 1 ,
\end{aligned} \right.
\end{equation}
In particular, we get for the principal symbol
\begin{equation}\label{ma9}
\partial_t(a_{-,0}(x_{k}^{-} (t)))=\nabla a_{-,0}(x_{k}^{-} (t))\cdot \xi_{k}^{-} (t)=\nabla a_{-,0}(x_{k}^{-} (t))\cdot \nabla\psi_-(x_{k}^{-} (t)) ,
\end{equation}
so that, 
\begin{equation} \label{ma10}
\partial_{t} (a_{-,0}(x_{k}^{-} (t))) =-\frac{1}{2} \big(  \Delta \psi_-(x_{k}^{-} (t)) - 2 i z \big) a_{-,0}(x_{k}^{-} (t))
\end{equation}
and then
\begin{equation}\label{mk2}
a_{-,0}(x_{k}^{-} (t)) = a_{-,0}(x_{k}^{-} (s)) \exp\left( - \frac{1}{2} \int_{s}^{t} \Delta\psi_-(x_-(u)) \, du + i (t-s) z \right) .
\end{equation}
On the other hand, from \cite[Lemma 4.3]{RoTa89_01}, based on Maslov theory, we have 
\begin{equation}\label{ma11}
a_{-,0}(x_{k}^{-} (t)) = (2E_0)^{1/4}D^-_k(t)^{-1/2}e^{itz},
\end{equation}
where 
\begin{equation}\label{zp2}
D^-_k(t)=\big\vert\det \frac{\partial x_-(t,z,\omega,E_0)}{\partial(t,z)}\vert_{z=z^-_k}\big\vert.
\end{equation}


\section{Computation of $u_-$ at the critical point}  \label{a20}

Now we use the results of  \cite{bfrz} to get a representation of $u_-$ in a whole neighborhood of the critical point. Indeed we saw already that $(P-E)u_-=0$  out of the support of $g_-$, in particular in a neighborhood of the critical point. First, we need to recall some terminology of \cite{HeSj85_01} and \cite{bfrz}.

We recall from Section \ref{sec2} that $(\mu_{j})_{j\geq 0}$ is the strictly growing  sequence of linear 
combinations over $\N$ of the $\lambda_{j}$'s. Let $u (t,x)$ be a function defined on $[0,+\infty[\times U$, $U \subset\R^{m}$.

\begin{definition}\label{expdeb}\sl We say that $u:[0,+\infty[\times U \to \R$, a smooth function, is expandible, if, for any $N\in\N$, $\varepsilon >0$, $\alpha , \beta \in \N^{1 + m}$,
\begin{equation}
\partial^{\alpha}_{t} \partial_{x}^\beta \Big( u(t,x)-\sum_{j= 1}^N 
u_{j}(t,x) e^{-\mu_{j}t} \Big)= \CO \big( e^{-(\mu_{N+1} - \varepsilon ) t} \big) ,
\label{defexpandible2}
\end{equation}
for a sequence of  $(u_{j})_{j}$ smooth functions, which are 
polynomials in $t$.
We shall write 
\begin{equation*}
u(t,x) \sim \sum_{j\geq 1} u_{j}(t,x) e^{-\mu_{j}t} ,
\end{equation*}
when (\ref{defexpandible2}) holds.
\label{def1}
\end{definition}

\noindent
We say that $f (t,x) = \widetilde{\CO} (e^{-\mu t})$ if for all $\alpha  , \beta \in \N^{1+m}$ and $\varepsilon >0$ we have
\begin{equation}
\partial_{t}^{\alpha} \partial_{x}^{\beta} f (t,x) = \CO (e^{-( \mu - \varepsilon) t}).
\end{equation}

\begin{definition}\sl  \label{a13}
We say that $u (t,x ,h)$, a smooth function, is of class $\CS^{A,B}$ if, for any $\varepsilon >0$, $\alpha , \beta \in \N^{1 + m}$,
\begin{equation} \label{defexpandible3}
\partial^{\alpha}_{t} \partial_{x}^\beta u(t,x, h) = \CO \big( h^{A} e^{-(B - \varepsilon ) t} \big) .
\end{equation}
Let $\CS^{\infty , B} = \bigcap_{A} \CS^{A,B}$. We say that $u (t,x ,h)$ is a classical expandible function of order $(A,B)$, if, for any $K \in \N$,
\begin{equation}
u(t,x,h) - \sum_{k=A}^{K} u_{k} (t,x) h^{k} \in \CS^{K+1,B},
\end{equation}
for a sequence of  $(u_{k})_{k}$ expandible functions. We shall write 
\begin{equation*}
u(t,x ,h) \sim \sum_{k\geq A} u_{k}(t,x) h^{k} ,
\end{equation*}
in that case.
\end{definition}

Since the intersection between $\Lambda_{\omega}^{-}$ and $\Lambda_{-}$ is transverse along the trajectories $\gamma^{-}_k(z_k^-)$, and since $g_{1}^-( z^{-}_{k}) \neq 0$,  Theorem 2.1 and Theorem 5.4 of \cite{bfrz}  implies that one can write, microlocally near $(0,0)$,
\begin{equation}  \label{a4}
u_-= \frac{1}{\sqrt{2\pi h}} \int \sum_{k=1}^{N_{-}} \alpha^{k} (t,x,h)e^{i\varphi^{k} (t,x)/h} d t,
\end{equation}
where the $\alpha^{k} (t,x,h)$'s are classical expandible functions in $\CS^{0,2\re \Sigma(E)}$:
\begin{equation}
\begin{aligned}
\alpha^{k} (t,x,h) &\sim \sum_{m \geq 0} \alpha_{m}^{k}  (t,x) h^{m} ,    \\
\alpha_{m}^{k} (t,x) &\sim \sum_{j \geq 0} \alpha_{m,j}^{k} (t,x) e^{- 2 ( \Sigma(E)+\mu_{j} )t} ,
\end{aligned}
\end{equation}
and where the $\alpha_{m,j}^{k} (t,x)$'s are polynomial with respect to $t$. We recall from \eqref{zzz2} that, for $E=E_0+hz$, 
\begin{equation}
 \Sigma(E)=\sum_{j=1}^n \frac{\lambda_j}{2} - i z .
\end{equation}
Following line by line Section 6 of \cite{bfrz}, we obtain (see \cite[(6.26)]{bfrz})
\begin{align}
\alpha_{0,0}^{k} (0)= e^{i\pi/4} ( 2 \lambda_1)^{3/2} e^{-i\nu_{k}^{-}\pi/2} 
\vert g (\gamma^{-}_{k})\vert 
(D_k^-)^{-1/2}(2E_0)^{1/4}.
 \label{a35}
\end{align}
Notice that  from  \eqref{mk2} and Proposition  \ref{a15}, we have $0< D_k^- <+\infty$.

From \cite[Section 5]{bfrz}, we recall that the phases $\varphi^k (t,x)$ satisfies the eikonal equation
\begin{equation}  \label{n6}
\partial_{t} \varphi^{k}+ p (x, \nabla_{x} \varphi^{k}) = E_{0},
\end{equation}
and that they have the asymptotic expansion
\begin{equation}  \label{d2}
\varphi^{k} (t,x) \sim \sum_{j=0}^{+\infty} \sum_{m=0}^{M^k_{j}} \varphi_{j,m}^{k} (x) t^{m} e^{- \mu_{j} t} ,
\end{equation}
with $M^k_{j} < + \infty$. In the following, we denote
\begin{equation}\label{s2}
\varphi_{j}^{k} (t,x)=\sum_{m=0}^{M^k_j} \varphi_{j,m}^{k} (x) t^m ,
\end{equation}
and the first $\varphi_{j}^{k}$'s are of the form
\begin{align}
\varphi_{0}^{k} (t,x) =& \varphi_{+} (x) + c_{k}   \label{a5} \\
\varphi_{1}^{k} (t,x)=& - 2 \lambda_1 g_-(z_{k}^{-})\cdot x + \CO (x^2) ,  \label{d4}
\end{align}
where $c_{k} \in \R$ is the constant depending on $k$ given by
\begin{equation}  \label{y18}
c_{k} = `` \psi_{-} (0) " = \lim_{t \to + \infty} \psi_{-} (x_{k}^{-} (t)) = S^{-}_{k},
\end{equation}
thanks to \eqref{ma8} (see also \cite[Lemma 5.10]{bfrz}). Moreover $\varphi_{+}$ is the generating function of the outgoing stable Lagrangian manifold $\Lambda_+$ with $\varphi_{+} (0) = 0$. We have
\begin{equation}
\varphi_{+} (x) = \sum_{j} \frac{\lambda_{j}}{2} x^{2}_{j} + \CO (x^{3}).
\end{equation}
The fact that $\varphi^{k}_{1} (t,x)$ does not depend on $t$ and the expression \eqref{d4} follows also from  Corollary \ref{tr500} and \eqref{tr1100}.

\Subsection{Study of the transport equations for the phases} \label{d3}

Now, we examine the equations satisfied by the functions $\varphi^k_{j} (t,x)$, defined in \eqref{d2},  for the integers $j \leq \widehat{\jmath}$ (recall that $\widehat{\jmath}$ is defined by $\mu_{\widehat{\jmath}} = 2 \lambda_{1}$). For clearer notations, we omit the superscript $k$ until further notice.

Let us recall that the function $\varphi(t,x)$ satisfies the eikonal equation \eqref{n6},
which implies (see \eqref{d2})
\begin{equation}
\sum_{j} \sum_{m=0}^{M_{j}} e^{- \mu_{j} t} \varphi_{j,m}^{}(x) (- \mu_{j} t^{m} + m t^{m-1}) + \frac{1}{2} \Big( \sum_{j} \sum_{m=0}^{M_{j}} \nabla \varphi_{j,m}^{}(x) t^{m} e^{- \mu_{j} t} \Big)^{2} + V(x) \sim E_{0} ,
\end{equation}
and then
\begin{align}
\sum_{j} \sum_{m=0}^{M_{j}} e^{- \mu_{j} t} \varphi_{j,m}^{}(x) (- \mu_{j} t^{m} + m t^{m-1}) + \frac{1}{2} \sum_{j, \widetilde{\jmath}} \sum_{m=0}^{M_{j}} \sum_{\widetilde{m}=0}^{M_{\widetilde{\jmath}}} \nabla \varphi_{j,m}^{}\nabla \varphi_{\widetilde{\jmath} , \widetilde{m}}^{}(x) e^{- (\mu_{j} + \mu_{\widetilde{\jmath}}) t} t^{m + \widetilde{m}}   \nonumber \\
+ V(x) \sim E_{0} .  \label{d7}
\end{align}
When $\mu_j<2\lambda_1$, the double product of the previous formula provides a term of the form $e^{- \mu_{j} t}$ if and only if $\mu_{j} =0$ or $\mu_{\widetilde{\jmath}}=0$. In particular, the term in $e^{- \mu_{j} t}$ in \eqref{d7} gives
\begin{equation} \label{d8}
\sum_{m=0}^{M_{j}} \varphi_{j,m}^{}(x) (- \mu_{j} t^{m} + m t^{m-1}) + \nabla \varphi_{+} (x)\cdot \sum_{m=0}^{M_{j}} \nabla \varphi_{j,m}^{}(x) t^{m} =0 .
\end{equation}
When $\mu_{j} = 2 \lambda_{1}$, one gets  also a term in $e^{- 2 \lambda_{1} t}$ for $\mu_{j} = \mu_{\widetilde{\jmath}} = \lambda_{1}$ and then
\begin{align}
\sum_{m=0}^{M_{j}} \varphi_{j,m}^{}(x) (- \mu_{j} t^{m} + m t^{m-1}) + \nabla & \varphi_{+} (x)\cdot \sum_{m=0}^{M_{j}} \nabla \varphi_{j,m}^{}(x) t^{m}  \nonumber \\
+& \frac{1}{2} \sum_{m=0}^{M_{1}} \sum_{\widetilde{m}=0}^{M_{1}} t^{m+\widetilde{m}} \nabla \varphi_{1,m}^{}(x) \nabla \varphi_{ 1, \widetilde{m}}^{}(x) =0 .   \label{d9}
\end{align}

We  denote 
\begin{equation}\label{zf1}
L = \nabla\varphi_+(x)\cdot\nabla
\end{equation}
 the vector field that appears in \eqref{d8} and \eqref{d9}. We set also  $L_{0}=  \sum_{j} \lambda_{j} x_{j} \partial_{j}$  its linear part at $x=0$, and we  begin with the study of the solution of
\begin{equation}  \label{q11}
( L - \mu ) f = g,
\end{equation}
with $\mu \in \R$ and $f$, $g \in C^{\infty}(\R^{n})$. First of all, we show that it is sufficient to solve \eqref{q11} for formal series.

\begin{proposition} \sl  \label{d16}
Let $g \in C^{\infty} (\R^{n})$ and $g_{0}$ the the Taylor expansion of $g$ at $0$. For each formal series $f_{0}$ such that $(L-\mu)f_{0} =g_{0}$, there exists one and only one function $f \in C^{\infty} (\R^{n})$ defined near $0$ such that $f =f_{0} + \CO (x^{\infty})$ and
\begin{equation}  \label{d13}
( L - \mu ) f =g ,
\end{equation}
near $0$.
\end{proposition}

\begin{proof}
Let $\widetilde{f_{0}}$ be a $C^{\infty}$ function having $f_{0}$ has Taylor expansion at $0$. With the notation $f = \widetilde{f_{0}} + r$, the problem \eqref{d13} is equivalent to find $r = \CO ( x^{\infty} )$ with
\begin{equation} \label{d14}
( L - \mu ) r = g - ( L - \mu ) \widetilde{f_{0}} = \widetilde{r} ,
\end{equation}
where $\widetilde{r} \in C^{\infty}$ has $g_{0} - ( L - \mu ) f_{0} =0$ as Taylor expansion at $0$. Let $y (t,x)$ be the solution of
\begin{equation}
\left\{ \begin{aligned}
&\partial_{t} y ( t, x) = \nabla \varphi_{+} (y(t,x)) ,  \\
&y (0,x) = x .
\end{aligned} \right.
\end{equation}
Thus, \eqref{d14} is equivalent to
\begin{equation}
r (x) = \int_{t}^{0} e^{- \mu s} \widetilde{r} ( y (s,x)) d s + e^{- \mu t} r (y(t,x)) .
\end{equation}
Since $r (x)$, $\widetilde{r} (x) = \CO (x^{\infty})$ and $y (s,x) = \CO (e^{\lambda_{1} t} \vert x \vert)$ for $t <0$, the functions $e^{- \mu t} r (y(t,x))$, $e^{- \mu t} \widetilde{r} ( y (t,x))$ are $\CO (e^{N t})$ as $t \to - \infty$ for all $N >0$. Then
\begin{equation}
r (x) = \int_{- \infty}^{0} e^{- \mu s} \widetilde{r} ( y (s,x)) d s ,  \label{d15}
\end{equation}
and $r (x) = \CO (x^{\infty})$. The uniqueness follows and it is enough to prove that $r$ given by \eqref{d15} is $C^{\infty}$. We have
\begin{equation}
\partial_{t} (\nabla_{x} y) = ( \nabla^{2}_{x} \varphi_{+} (y) ) (\nabla_{x} y),
\end{equation}
and since $\nabla^{2}_x\varphi_{+}$ is bounded, there exists $C>0$ such that
\begin{equation}
\vert \nabla_{x} y (t,x) \vert \lesssim e^{-Ct} ,
\end{equation}
has $t \to - \infty$. Then, $e^{- \mu s} ( \nabla  \widetilde{r} ) ( y (s,x)) (\partial_{j} y (t,x)) = \CO (e^{N t})$ as $t \to - \infty$ for all $N>0$ and $\partial_{j} r (x) = \int_{- \infty}^{0} e^{- \mu s} ( \nabla  \widetilde{r} ) ( y (s,x)) (\partial_{j} y (t,x)) d s$. The derivatives of order greater than $1$ can be treated the same way.
\end{proof}

We denote
\begin{equation}\label{d2222bis}
L_\mu=L-\mu:\C \llbracket x \rrbracket \to \C \llbracket x \rrbracket,
\end{equation}
where we use the standard notation $\C \llbracket x \rrbracket$ for formal series, and $\C_p \llbracket x \rrbracket$ for formal series of degree $\geq p$. We notice  that
\begin{equation}   \label{d17}
L_\mu  x^{\alpha} = (L_{0} - \mu ) x^{\alpha} +\C_{\vert \alpha \vert +1} \llbracket x \rrbracket= (\lambda \cdot \alpha - \mu ) x^{\alpha} + \C_{\vert \alpha \vert +1} \llbracket x \rrbracket.
\end{equation}
Recall that
$\CI_\ell(\mu)$ has been defined in \eqref{zh1}. The number of elements in $\CI_{\ell} ( \mu )$ will be denoted 
\begin{equation}
n_{\ell} (\mu ) = \# \CI_{\ell} ( \mu ).
\end{equation}
One has for example $n_{2} (\mu ) = \frac{n_{1} (\mu ) ( n_{1} (\mu ) +1)}{2}$.

\begin{proposition} \sl  \label{d22bis} 
Suppose $\mu\in ]0,2\lambda_1[$. With the above notations, one has $\ker L_\mu\oplus \im L_\mu=\C\llbracket x\rrbracket$. More precisely:
\begin{enumerate}
\item The kernel of $L_\mu$ has dimension $n_1(\mu)$, and one can  find  a basis  $(E_{j_1},\dots , E_{j_{n_1(\mu)}})$  of  $\ker L_\mu$ such that $E_j(x)=x_j+ \C_2\llbracket x\rrbracket$, $j\in\CI_1(\mu)$.
\item A formal series $\ds F=F_0+\sum_{j=1}^n F_jx_j+\C_2\llbracket x\rrbracket$ belongs to $\im L_\mu$ if and only if $F_j=0$ for all $j\in \CI_1(\mu)$.
\end{enumerate}
\end{proposition}

\begin{remark}\sl
Thanks to Propostion \ref{d16}, the same result is true for germs of $C^\infty$ functions at 0. Notice that when $\mu\neq \mu_j$ for all $j$, $L_\mu$ is invertible.
\end{remark}

\begin{proof}
For a given $F=\sum_{\alpha }F_\alpha x^\alpha\in\C\llbracket x\rrbracket$, we look for solutions $E=\sum_{\alpha }E_\alpha x^\alpha\in\C\llbracket x\rrbracket$ to the equation 
\begin{equation}\label{dr1}
L_\mu \Big( \sum_{\alpha }E_\alpha x^\alpha \Big)=\sum_{\alpha }F_\alpha x^\alpha.
\end{equation}
The calculus of the term of order $x^{0}$ in \eqref{dr1} leads to the equation
\begin{equation}
E_0=-\frac{F_0}{\mu} .
\end{equation}
With this value for $E_0$, \eqref{dr1} becomes, using again \eqref{d17},
\begin{equation}\label{dr2}
\sum_{\vert \alpha\vert=1}(\lambda\cdot \alpha-\mu)E_\alpha x^\alpha=\sum_{\vert \alpha\vert=1}F_\alpha x^\alpha +\C_2\llbracket x\rrbracket.
\end{equation}
We have two cases:

If $\alpha\notin \CI_1(\mu)$, one should have
\begin{equation}
E_\alpha=\frac{F_\alpha}{\lambda\cdot\alpha-\mu} .
\end{equation}

If $\alpha \in \CI_1(\mu)$, the formula \eqref{dr2} becomes $F_{\alpha} = 0$. In that case, the corresponding $E_\alpha$ can be chosen arbitrarily.

Now suppose that the $E_\alpha$ are fixed for any $\vert \alpha\vert \leq n-1$ (with $n\geq 2$), and  such that
\begin{equation}
L_{\mu} \Big( \sum_{\vert \alpha \vert \leq n-1} E_{\alpha} x^{\alpha} \Big) = \sum_{\alpha} F_{\alpha} x^{\alpha} + \C_{n} \llbracket x\rrbracket .
\end{equation}
We can write \eqref{dr1} as
\begin{equation}
L_\mu \Big( \sum_{\vert\alpha\vert= n }E_\alpha x^\alpha \Big) = \sum_{\alpha} F_\alpha x^\alpha - L_{\mu} \Big( \sum_{\vert\alpha\vert \leq n-1}E_\alpha x^\alpha \Big) +\C_{n+1}\llbracket x\rrbracket,
\end{equation}
or, using again \eqref{d17}, 
\begin{equation}\label{dr3}
\sum_{\vert\alpha\vert= n }(\lambda\cdot \alpha-\mu)E_\alpha x^\alpha=
\sum_{\vert\alpha\vert\leq n }F_\alpha x^\alpha- L_{\mu} \Big( \sum_{\vert\alpha\vert \leq n-1}E_\alpha x^\alpha \Big) + \C_{n+1} \llbracket x\rrbracket.
\end{equation}
Since $\vert \alpha\vert\geq 2$, one has $\lambda\cdot\alpha\geq 2\lambda_1> \mu$, so that \eqref{dr3} determines by induction all the $E_\alpha$'s for $\vert\alpha\vert=n$ in a unique way.
\end{proof}

\begin{corollary} \label{tr500}\sl
If $j < \widehat{\jmath}$, the function $\varphi_j(t,x)$  does not depend on $t$, i.e.\  we have $M_{j} =0$.
\end{corollary}

\begin{proof}
Suppose  that $M_{j} \geq 1$, then \eqref{d8} gives the system
\begin{equation}  \label{d23}
\left\{
\begin{aligned}
&(L - \mu_{j} ) \varphi^{}_{j,M_{j}} = 0,\\  
&(L - \mu_{j} ) \varphi^{}_{j,M_{j}-1} = - M_{j} \varphi^{}_{j,M_{j}}  , 
\end{aligned}
\right.
\end{equation}
with $\varphi^{}_{j,M_{j}} \neq 0$. But this would imply that $\varphi^{}_{j,M_{j}}\in \ker L_\mu\cap \im L_\mu$, a contradiction.
\end{proof}

As a consequence, for $j< \widehat{\jmath}$, the equation \eqref{d8} on $\varphi_{j}$ reduces to
\begin{equation}  \label{y1}
(L - \mu_{j} ) \varphi^{}_{j,0} =0,
\end{equation}
and, from Proposition \ref{d22bis},  we get that
\begin{equation}  \label{d42}
\varphi^{}_{j} (t,x)=\varphi^{}_{j,0} (x)= \sum_{k \in \CI_1( \mu ) } d_{j,k} x_{k} + \CO ( x^{2} ).
\end{equation}

Now we pass to the case $j = \widehat{\jmath}$, and we  study \eqref{d9}. First of all, we have seen that $\varphi_{1}$ does not depend on $t$, so that this equation can be written
\begin{equation}   \label{d10}
\sum_{m=0}^{M_{j}} \varphi_{j,m}^{}(x) (- \mu_{j} t^{m} + m t^{m-1}) + \nabla \varphi_{+} \cdot \sum_{m=0}^{M_{j}} \nabla \varphi_{j,m}^{}(x) t^{m} + \frac{1}{2}  \big\vert \nabla \varphi_{1}^{}(x) \big\vert^{2} =0 .
\end{equation}

As for the study of \eqref{d8}, we begin with that of \eqref{q11}, now in the case where $\mu=2\lambda_1$. We denote 
$\Psi: 
\R^{n_1(2\lambda_1)}\longrightarrow 
\R^{n_2(\lambda_1)}$ the linear map given by
\begin{equation}\label{tr106}
\Psi( E_{\beta_{1}} , \ldots , E_{\beta_{n_1(2\lambda_1)}} ) =\Big( \sum_{\beta\in\CI_1(2\lambda_1)} E_{\beta} \frac{1}{\alpha !} \big( \partial^{\alpha}(L - \mu ) x^\beta \big){\vert_{x=0}}\Big)_{\alpha\in \CI_2(\lambda_1)},
\end{equation}
and we set 
\begin{equation}\label{tr107}
n(\Psi)=\dim\ker\Psi.
\end{equation}
Recalling that $L=\nabla\varphi_+(x)\cdot \nabla$, we see that
\begin{equation}
\Psi(E_{\beta_{1}} , \ldots , E_{\beta_{n_1(2\lambda_1)}} ) =\Big (\sum_{\beta\in\CI_1(2\lambda_1)} E_{\beta}\frac{\partial^{\alpha} \partial^\beta\varphi_{+} (0)}{\alpha !}\Big)_{\alpha\in\CI_2(\lambda_1)}
.
\end{equation}
More generally,  for any $\vert \alpha\vert=2$, we denote 
\begin{equation}
\Psi_\alpha
(
(E_{\beta})_{\beta\in\CI_1(2\lambda_1)})=
\sum_{\beta\in\CI_1(2\lambda_1)} E_{\beta}\frac{\partial^{\alpha} \partial^\beta\varphi_{+} (0)}{\alpha !}\cdotp
\end{equation}
Then, at the level of formal series, we have the

\begin{proposition} \sl  \label{d26}
Suppose $\mu  = 2 \lambda_{1}$. Then
\begin{enumerate}
\item  $\ker L_{\mu}$ has dimension $n_2(\lambda_1)+n(\Psi)$.
 
\item A formal series $F=\sum_\alpha F_\alpha x^\alpha$ belongs to $\im L_\mu$ if and only if
\begin{align}
&\forall   \alpha\in\CI_1(2\lambda_1),\quad F_\alpha = 0, \label{d27}
\\
&\Big(\sum_{\fract{\vert\beta\vert=1}{\beta \notin\CI_1( 2 \lambda_{1})}}
\frac{\partial^\beta\partial^{\alpha} \varphi_{+} (0)}{\alpha !} \frac{F_\beta}{2 \lambda_{1} - \lambda\cdot\beta} + F_\alpha\Big)_{ \alpha\in\CI_2(\lambda_1)} \in \im \Psi.
 \label{d28}
\end{align}
\item If $F\in \im L_\mu$, any formal series $E=\sum_\alpha E_\alpha x^\alpha$ with $L_\mu E=F$  satisfies
\begin{align}
&E_0 = \frac{1}{- 2 \lambda_{1} }F_0, \label{d35}   \\
&E_\alpha= \frac{1}{\lambda\cdot\alpha- 2 \lambda_{1}} F_\alpha, \qquad \text{ for } \alpha\in \CI_1\setminus\CI_1(2\lambda_1),   \label{d36}  \\
&\Psi \big( (E_\beta \big)_{\beta\in\CI_1(2\lambda_1)})
 =
\Big( \sum_{\fract{\vert\beta\vert=1}{\beta \notin\CI_1( 2 \lambda_{1})}}
 \frac{\partial^{\beta} \partial^\alpha\varphi_{+} (0)}{\alpha !} \frac{F_\beta}{2 \lambda_{1} - \lambda\cdot\beta} +F_\alpha \Big)_{\alpha\in \CI_2(\lambda_1)}.
 \label{d40}
 \end{align}
Moreover for $\alpha\in \CI_2\setminus\CI_2(\lambda_1)$, one has
\begin{equation} \label{trtroups}
E_\alpha=\frac{1}{\lambda\cdot\alpha-2\lambda_1}\Big( F_\alpha-\Psi_\alpha((E_\beta)_{\beta\in\CI_1(2\lambda_1)})+\sum_{\fract{\vert\beta\vert=1}{\beta\notin\CI_1(2\lambda_1)}}\frac{F_\beta}{2\lambda_1-\lambda\cdot\beta} \frac{\partial^{\alpha+\beta} \varphi_+(0)}{\alpha !}\Big).
\end{equation}
Last, $E$ is completely determined by $F$ and  a choice of the $E_\alpha$ for $\vert\alpha\vert\leq 2$ such that \eqref{d35}-- \eqref{trtroups} are satisfied.
\item  $\ker L_\mu\cap \im (L_\mu)^2=\{0\}$.
\end{enumerate}
\end{proposition}

\begin{proof}  For a given $F=\sum_\alpha F_\alpha x^\alpha$ we look for a $E=\sum_\alpha E_\alpha x^\alpha$ such that 
$L_{2\lambda_1} E=F$. First of all, we must have 
\begin{equation}\label{tr100}
E_0=-\ds\frac {F_0}{2\lambda_1}. 
\end{equation} 
When this is true, we get
\begin{equation}\label{tr101}
\sum_{\vert\alpha\vert=1}E_\alpha (L_0-2\lambda_1)x^\alpha=\sum_{\vert\alpha\vert=1}F_\alpha (L-2\lambda_1)x^\alpha +\C_2\llbracket x\rrbracket,
\end{equation}
and we obtain as necessary condition that $F_\alpha=0$ for any $\alpha\in \CI_1(2\lambda_1)$. So far, the $E_\alpha$ for $\alpha\in \CI_1(2\lambda_1)$ can be chosen arbitrarily, and we must have
\begin{equation}\label{tr102}
E_\alpha=\frac{F_\alpha}{\lambda\cdot\alpha-2\lambda_1},\quad \alpha\in\CI_2\setminus\CI_1(2\lambda_1).
\end{equation}

We suppose that \eqref{tr100} and \eqref{tr102} hold. Then we should have
\begin{equation}\label{tr103}
\sum_{\vert\alpha\vert=2}E_\alpha (L_0-2\lambda_1)x^\alpha=
\sum_{\vert\alpha\vert=2}F_\alpha x^\alpha +
\Big (\sum_{\fract{\vert\alpha\vert=1}{\alpha \notin \CI_1(2\lambda_1)}}F_\alpha x^\alpha -\sum_{\vert\alpha\vert=1}E_\alpha  (L-2\lambda_1)x^\alpha\Big )
+\C_3\llbracket x\rrbracket.
\end{equation}
Notice that the second term in the R.H.S of \eqref{tr103} belongs to $\C_2\llbracket x\rrbracket$ thanks to 
\eqref{tr102}. Again, we have to cases:
\begin{itemize}
\item When $\alpha\in \CI_2(\lambda_1)$, the corresponding $E_\alpha$ can be chosen arbitrarily, but one must have
\begin{align}\label{tr104}
F_\alpha=&
\sum_{\vert\beta\vert=1}E_\beta  \big(\frac1{\alpha!}{\partial^\alpha}(L-2\lambda_1)x^\beta\big){\vert_{x=0}}\\
=&
\Psi_\alpha((E_\beta)_{\beta\in\CI_1(2\lambda_1)})+\sum_{\fract{\vert\beta\vert=1}{\beta\notin\CI_1(2\lambda_1)}}E_\beta \frac{\partial^{\alpha+\beta} \varphi_+(0)}{\alpha !},
\end{align}
and this, with \eqref{tr102}, gives \eqref{d40}.

\item When $\vert\alpha\vert=2$, $\alpha\notin \CI_2(\lambda_1)$, one obtains
\begin{align}\label{tr105}
\nonumber
E_\alpha=&
\frac{1}{\lambda\cdot\alpha-2\lambda_1}\Big( F_\alpha-\sum_{\vert\beta\vert=1}E_\beta  \big(\frac1{\alpha!}{\partial^\alpha}(L-2\lambda_1)x^\beta\big){\vert_{x=0}}\Big)\\
=&
\frac{1}{\lambda\cdot\alpha-2\lambda_1}\Big( F_\alpha-\Psi_\alpha((E_\beta)_{\beta\in\CI_1(2\lambda_1)})-\sum_{\fract{\vert\beta\vert=1}{\beta\notin\CI_1(2\lambda_1)}}E_\beta \frac{\partial^{\alpha+\beta} \varphi_+(0)}{\alpha !}\Big),
\end{align}
and this, with \eqref{tr102}, gives \eqref{trtroups}.
\end{itemize}

Now suppose that   \eqref{tr100}, \eqref{tr102}, \eqref{tr104} and \eqref{tr105} hold, and that we have chosen a value for the free variables $E_\alpha$ for $\alpha\in \CI_1(2\lambda_1)\cup\CI_2(\lambda_1)$. Thanks to the fact that $\lambda\cdot \alpha\neq 2\lambda_1$ for any $\alpha \in \N^n$ with $\vert \alpha\vert=3$, we see as in the proof of Propostion \ref{d22bis}, that the equation \eqref{tr101} has a unique solution, and the points (i), (ii) and (iii) follows easily.

We prove the last point of the proposition, and we suppose that 
\begin{equation}\label{tr110}
E=\ds \sum_{\alpha\in \N^n} E_\alpha x^\alpha\in\ker L_\mu\cap \im (L_\mu)^2.
\end{equation}
First, we have $E\in \ker L_\mu\cap \im L_\mu$. Thus,  $E_0=0$ by \eqref{d35}, $E_\alpha=0$ for $\alpha\in \CI_1(2\lambda_1)$ by \eqref{d27}, and $E_\alpha=0$ for $\alpha\in\CI_1\setminus \CI_1(2\lambda_1)$ by \eqref{d36}. Last, since $L_\mu E=0$, we also have $E_\alpha=0$ for $\alpha\in \CI_2\setminus\CI_2(\lambda_1)$, and finally,
\begin{equation}\label{tr111}
E=\sum_{\alpha\in\CI_2(\lambda_1)} E_\alpha x^\alpha+\C_3\llbracket x\rrbracket.
\end{equation}
Moreover, one can write  $E=L_\mu G$ for some $G\in \im L_\mu$. Since $E_0=0$, we must have $G_0=0$. Since $G\in \im L_\mu$, by \eqref{d27}, we have $G_\alpha=0$ for $\alpha\in \CI_1(2\lambda_1)$. Finally, since $E_\alpha=0$ for $\vert\alpha\vert=1, \alpha\notin \CI_1(2\lambda_1)$, the same is true for the corresponding $G_\alpha$, and
\begin{equation}\label{tr112}
G=\sum_{\vert\alpha\vert\geq 2} G_\alpha x^\alpha. 
\end{equation}
Then, since $L_\mu x^\alpha=0+\C_3[x]$ for $\alpha\in \CI_2(\lambda_1)$,  we obtain $E_\alpha=0$ for $\alpha\in \CI_2(\lambda_1)$.  As above, we then get that,  for $\vert \alpha\vert\geq 3$, $E_\alpha=0$, and this ends the proof.
\end{proof}

\begin{corollary}\label{tr501} \sl
We always have $M_{\widehat{\jmath}} \leq 2$. If, in addition, $\lambda_{k} \neq 2 \lambda_{1}$ for all $k \in \{ 1 , \ldots , n \}$, then $M_{\widehat{\jmath}} \leq 1$.
\end{corollary}

\begin{proof}
Suppose that $M_{\jj} \geq 3$. Then \eqref{d10} gives
\begin{align}
&(L - \mu_{\jj} ) \varphi_{\jj,M_{\jj}} = 0  \label{d29}  \\
&(L - \mu_{\jj} ) \varphi_{\jj,M_{\jj}\,-1} = - M_{\jj} \varphi_{\jj,M_{\jj}}   \label{d30}  \\
&(L - \mu_{\jj} ) \varphi_{\jj,M_{\jj}\,-2} = - (M_{\jj} -1) \varphi_{\jj,M_{\jj}\,-1} \label{d31},
\end{align}
with $\varphi_{\jj,M_{\jj}} \neq 0$. Notice that we have used the fact that $M_{\jj}-2>0$ in \eqref{d31}. But this gives $\varphi_{\jj,M_{\jj}}\in \ker (L-\mu_{\jj})$ and $(L-\mu_{\jj})^2\varphi_{\jj,M_{\jj}-2}=M_{\jj}(M_{\jj}-1)\varphi_{\jj,M_{\jj}}$,  so that  $\varphi_{\jj,M_{\jj}}\in \im (L-\mu_{\jj})^2$. This contradicts 
point (iv) of Proposition \ref{d26}.

Now we suppose  that $\lambda_{k} \neq 2 \lambda_{1}$ for all $k \in \{ 1 , \ldots n \}$, that is  $\CI_1(2\lambda_1)=\emptyset$, and that $M_{\jj}=2$.
Then \eqref{d10} gives
\begin{align}
&(L - \mu_{\jj} ) \varphi_{\jj,M_{\jj}} = 0  \label{tr120}  \\
&(L - \mu_{\jj} ) \varphi_{\jj,M_{\jj}\, -1} = - M_{\jj}\; \varphi_{\jj,M_{\jj}}   \label{tr121} 
\end{align}
with $\varphi_{\jj,M_{\jj}} \neq 0$. Therefore we have $\varphi_{\jj,M_{\jj}}\in \ker L_{\mu_{\jj}}\cap \im L_{\mu_{\jj}}$, and we get the same conclusion as in \eqref{tr111}: $\varphi_{\jj,M_{\jj}}(x)=\CO(x^2)$. Then, we write 
\begin{equation}\label{tr122}
\varphi_{\jj,M_{\jj}}=(L-\mu_{\jj})g,
\end{equation}
and we see, as in \eqref{tr112}, that $g=\CO(x^2)$, here  because  $\CI_1(2\lambda_1)=\emptyset$.
Finally,  we conclude also that $\varphi_{\jj,M_{\jj}}=0$, a contradiction.
\end{proof}

\Subsection{Taylor expansions of $\varphi_{+}$ and $\varphi^k_{1}$}  \label{d55}

Now we compute the Taylor expansions of the leading terms with respect to $t$, of the phase functions $\varphi (t,x)=\varphi^k (t,x)$. 

\begin{lemma}\sl   \label{d5}
The smooth function $\ds \varphi_{+} (x) = \sum_{j=1}^{n} \frac{\lambda_{j}}{2} x_{j}^{2} + \CO (x^{3})$ satisfies
\begin{equation}
\partial^{\alpha} \varphi_+(0) = -\frac1{\lambda \cdot \alpha}\partial^\alpha V(0) ,  \\
\end{equation}
for $\vert \alpha \vert=3$, and
\begin{equation}
\partial^\alpha\varphi_+(0) = -\frac1{2 ( \lambda \cdot \alpha )}
\sum_{j=1}^n \sum_{\fract{\beta , \gamma \in \CI_{2}}{\alpha = \beta + \gamma}} \frac{\alpha!}{\beta ! \, \gamma !} \frac{\partial_j\partial^\beta V(0)}{\lambda_{j} + \lambda \cdot \beta} \frac{\partial_j\partial^{\gamma}V(0)}{\lambda_{j} + \lambda \cdot \gamma} -\frac1{\lambda \cdot \alpha}\partial^\alpha V(0),
\end{equation}
for $\vert \alpha \vert=4$, where $\alpha , \beta , \gamma \in \N^{n}$ are multi-indices.
\end{lemma}

\begin{proof}
The smooth function $x\mapsto \varphi_+(x)$ is defined in a neighborhood of $0$, and it is characterized  (up to a constant: we have chosen $\varphi_+(0)=0$) by
\begin{equation}\label{q31}
\left\{ \begin{aligned}
&p(x,\nabla\varphi_+(x))=\frac12 \vert\nabla\varphi_+(x)\vert^2+V(x) =0 \\[8pt]
&\nabla\varphi_+(x)=\left(\lambda_jx_j\right )_{j=1 , \ldots , n}+ \CO ( x^2 )
\end{aligned} \right .
\end{equation}
The Taylor expansion of $\varphi_+$ at $x=0$ is
\begin{equation}\label{q32}
\varphi_+(x)=
\sum_{j=1}^{n} \frac{\lambda_{j}}{2} x_j^2+
\sum_{\vert\alpha\vert = 3 ,4} \frac{\partial^\alpha\varphi_+(0)}{\alpha !} x^\alpha
+ \CO ( x^{5} ),
\end{equation}
and we have
\begin{equation}\label{q33}
\partial_j\varphi_+(x)={\lambda_j} x_j+ \sum_{\vert\alpha\vert = 3 ,4} \alpha_j \frac{\partial^\alpha\varphi_+(0)}{\alpha !} x^{\alpha-1_j} + \CO( x^{4} ).
\end{equation}
Therefore
\begin{align}  \label{q34}
\vert\nabla\varphi_+(x)\vert^2=&
\sum_{j=1}^n\lambda_j^2x_j^2
+2\sum_{\vert\alpha\vert=3} \Big( \sum_{j=1}^n\lambda_j\alpha_j \Big) \frac{\partial^\alpha\varphi_+(0)}{\alpha !} x^\alpha + 2 \sum_{\vert\alpha\vert=4} \Big( \sum_{j=1}^{n} \lambda_j \alpha_j \Big) \frac{\partial^\alpha\varphi_+(0)}{\alpha !} x^\alpha \nonumber \\
&+ \sum_{j=1}^n \Big(\sum_{\vert\alpha\vert=3} \alpha_j\frac{\partial^\alpha\varphi_+(0)}{\alpha !} x^{\alpha-1_j} \Big)^2 +\CO( x^{5} ).
\end{align}
Let us compute further the last term in \eqref{q34}:
\begin{align}
\sum_{j=1}^{n} \Big( \sum_{\vert\alpha\vert=3} \alpha_j\frac{\partial^\alpha\varphi_+(0)}{\alpha !} x^{\alpha-1_j} \Big)^2=& \sum_{j=1}^n \sum_{\vert\beta\vert , \vert\gamma\vert=3}
\beta_j\gamma_j\frac{\partial^\beta\varphi_+(0)}{\beta !} \frac{\partial^\gamma\varphi_+(0)}{\gamma !} x^{\beta+\gamma-21_j}    \nonumber   \\
=& \sum_{j=1}^n \sum_{\vert\alpha\vert=4}x^\alpha \Big( \sum_{\fract{\alpha = \beta + \gamma}{\vert \beta \vert , \vert \gamma \vert =2}} \frac{\partial_j\partial^{\beta}\varphi_+(0)}{\beta !} \frac{\partial_j\partial^{\gamma}\varphi_+(0)}{\gamma !} \Big)  \cdotp  \label{s3}
\end{align}

Writing the Taylor expansion of $V$ at $x=0$ as
\begin{equation}\label{q36}
V(x)=\sum_{j=1}^n\frac{\lambda_j^2}{2}x_j^2+
\sum_{\vert\alpha\vert = 3,4} \frac{\partial^\alpha V(0)}{\alpha !} x^\alpha
+\CO(x^5),
\end{equation}
and using the eikonal equation \eqref{q31}, we obtain first, for any $\alpha\in \N^n$ with $\vert \alpha\vert=3$, 
\begin{equation}\label{s4}
\partial^\alpha\varphi_+(0)=-\frac1{\lambda\cdot\alpha}\partial^\alpha V(0) .
\end{equation}
Then, \eqref{q34} and \eqref{s3} give
\begin{equation} \label{d60}
\partial^\alpha\varphi_+(0)= - \frac{1}{\lambda\cdot\alpha}\partial^\alpha V(0) -\frac{1}{2 ( \lambda \cdot \alpha )} \sum_{j=1}^n \sum_{\fract{\beta , \gamma \in \CI_{2}}{\alpha = \beta + \gamma}} \frac{\alpha !}{\beta ! \gamma !}\frac{\partial_j \partial^{\beta} V(0)}{\lambda_{j} + \lambda \cdot \beta} \frac{\partial_j \partial^{\gamma} V(0)}{\lambda_{j} + \lambda \cdot \gamma} , 
\end{equation}
for $\vert \alpha\vert=4$.
\end{proof}

Now we  pass to the function $\varphi_{1}$. This function is a solution, in a neighborhood of $x=0$, of the transport equation
\begin{equation}\label{q40}
L \varphi_{1} (x)=\lambda_1 \varphi_{1} (x),
\end{equation}
where $L$ is given in \eqref{zf1}.

\begin{lemma}\sl   \label{d6}
The $C^{\infty}$ function $\varphi_{1} (x) = - 2 \lambda_1 g_{1}^{-} (z_{k}^{-})\cdot x + \CO (x^2)$ satisfies
\begin{equation}\label{tr1200}
\partial^{\alpha} \varphi_{1} (0) = \frac{2 \lambda_1 \alpha !}{( \lambda_1-\lambda\cdot \alpha ) (\lambda_{1} + \lambda \cdot \alpha )} \sum_{j=1}^n \frac{\partial_{j} \partial^{\alpha} V(0)}{\alpha !} \big( g_{1}^{-} (z_{k}^{-}) \big)_{j} ,
\end{equation}
for $\vert \alpha \vert =2$, and
\begin{align}
\partial^{\alpha} \varphi_{1} (0) =& - \frac{2 \lambda_{1}}{\lambda_{1} - \lambda \cdot \alpha}   \sum_{\fractt{k \in \CI_{1} ( \lambda_{1} ) , j \in \CI_{1}}{\beta , \gamma \in \CI_{2}}{\alpha + 1_{j} = \beta + \gamma}}   \frac{\alpha ! \gamma_{j}}{\beta ! \gamma !} \frac{\partial_{j} \partial^{\beta} V (0)}{\lambda_{j} + \lambda \cdot \beta}   \frac{\partial_{k} \partial^{\gamma} V (0)}{(\lambda_{1} - \lambda \cdot \gamma ) (\lambda_{1} + \lambda \cdot \gamma )} \big( g_{1}^{-} (z_{k}^{-}) \big)_{k} \nonumber \\
&+ \frac{\lambda_{1}}{( \lambda_1-\lambda \cdot \alpha ) (\lambda_{1} + \lambda \cdot \alpha )} \sum_{\fractt{k \in \CI_{1} , j \in \CI_{1} ( \lambda_{1} )}{\beta , \gamma \in \CI_{2}}{1_{j} + \alpha = \beta + \gamma}} \frac{(\alpha + 1_{j}) !}{\beta ! \gamma !} \frac{\partial_{k} \partial^{\beta} V (0)}{\lambda_{k} + \lambda \cdot \beta} \frac{\partial_{k} \partial^{\gamma} V (0)}{\lambda_{k} + \lambda \cdot \gamma} \big( g_{1}^{-} (z_{k}^{-}) \big)_{j}    \nonumber  \\
&+ \frac{2 \lambda_{1}}{( \lambda_1-\lambda \cdot \alpha ) (\lambda_{1} + \lambda \cdot \alpha )} \sum_{j \in \CI_{1} ( \lambda_{1} )} \partial_{j} \partial^{\alpha} V (0) \big( g_{1}^{-} (z_{k}^{-}) \big)_{j}.
\end{align}
for $\vert \alpha \vert =3$.
\end{lemma}

\begin{proof}
We write
\begin{equation}    \label{d57}
\varphi_{1} (x)=\sum_{j=1}^na_jx_j+\sum_{\vert\alpha\vert = 2,3} a_\alpha x^\alpha+\CO( x^{4} ),
\end{equation}
and Lemma \ref{d5} together with \eqref{q33} give all the coefficients in the expansion
\begin{equation}    \label{q42}
\nabla\varphi_+(x) = \Big( \lambda_jx_j+\sum_{\vert \alpha\vert= 2,3} A_{j,\alpha} x^\alpha +\CO( x^{4} ) \Big)_{j=1, \ldots , n}.
\end{equation}
In fact, we have
\begin{equation}\label{tr99}
A_{j, \alpha } = \frac{\partial^{\alpha + 1_{j}} \varphi_{+} (0)}{\alpha !} \quad \text{ and } \quad a_{\alpha} = \frac{\partial^{\alpha} \varphi_{1} (0)}{\alpha !} .
\end{equation}
We get
\begin{align}
L \varphi_{1} (x)=&\sum_{j=1}^n\partial_j\varphi_+(x)\partial_j\varphi_{1} (x)  \nonumber \\
=& \sum_{j=1}^n \Big( a_j\lambda_jx_j +\sum_{\vert\alpha\vert=2} \big( \alpha_j\lambda_j a_\alpha + a_j A_{j,\alpha} \big) x^\alpha \nonumber \\
& + \sum_{\vert\alpha\vert=3} \alpha_j\lambda_j a_\alpha x^\alpha + \sum_{\vert \beta \vert = \vert \gamma \vert =2} A_{j, \beta} \gamma_j a_\gamma x^{\beta+\gamma-1_j} +\sum_{\vert\alpha\vert=3} a_j A_{j,\alpha} x^{\alpha} \Big) +\CO(x^{4}) \nonumber \\
=&\sum_{j=1}^n a_j\lambda_jx_j+\sum_{\vert\alpha\vert=2} \big( \lambda\cdot\alpha\; a_\alpha + \sum_{j=1}^n A_{j, \alpha} a_j \big) x^\alpha  \nonumber  \\
&+ \sum_{\vert \alpha \vert =3} \Big( \lambda\cdot \alpha\; a_\alpha+\sum_{j=1}^n\big ( \sum_{\fract{\alpha = \beta + \gamma -1_j}{\vert \beta \vert , \vert \gamma \vert =2}} A_{j,\beta}\gamma_ja_{\gamma} + a_jA_{j,\alpha} \big) \Big) x^\alpha + \CO(x^{4}).
\end{align}
Thus, \eqref{q40} gives, for all $\alpha\in \N^n$ with $\vert\alpha\vert=2$,
\begin{equation}\label{q44}
a_\alpha=\frac{1}{\lambda_1-\lambda\cdot \alpha} \sum_{j=1}^n A_{j,\alpha}a_j,
\end{equation}
and, for all $\alpha\in \N^n$ with $\vert\alpha\vert=3$,
\begin{equation}  \label{d58}
a_\alpha=
\frac{1}{\lambda_1-\lambda\cdot \alpha}\; \sum_{j=1}^n  \Big( \sum_{\fract{\beta , \gamma \in \CI_{2}}{\alpha + 1_{j} = \beta + \gamma}} \gamma_j A_{j,\beta} a_{\gamma}
+a_jA_{j,\alpha} \Big) .
\end{equation}
Then,  the lemma follows from \eqref{tr99}.
\end{proof}

\Subsection{Asymptotics near the critical point for the trajectories}
\label{s66}

The informations obtained so far are not sufficient for the computation of the $\varphi_j$'s. We shall obtain here some more knowledge by studying the behaviour of the incoming trajectory $\gamma^-(t)$ as $t\to+\infty$. We recall from \cite[Section 3]{HeSj85_01} (see also \cite[Section 5]{bfrz}),  that the curve $\gamma^{-} (t) = ( x^{-} (t) , \xi^{-} (t) ) \in \Lambda_{-}\cap\Lambda_\omega^-$ satisfy, in the sense of expandible functions,
\begin{equation}\label{tr1003}
\gamma^-(t)=\sum_{j\geq 1}\sum_{m=0}^{M_j'} \gamma^-_{j,m} t^m e^{-\mu_jt},
\end{equation}
Notice that we continue to omit the subscript $k$ for $\gamma^{-}_k=(x_k^-,\xi_k^-)$, $z^{-}_k$, $\dots$ Writing also
\begin{equation}
x^{-} (t) \sim \sum_{j = 1}^{+ \infty} g_{j,m}^{-} (t,z_-)e^{- \mu_{j} t},\quad
 g_{j}^{-} (z^{-},t)=\sum_{m=0}^{M'_{j}} g_{j,m}^{-} (z^{-}) t^{m} ,
\end{equation}
 for some integers $M'_j$, we know that $g_{1}^{-} (z^{-}) = g_{1,0}^{-} (z^{-}) \neq 0$.
Since $\xi^{-} (t) = \partial_{t} x^{-} (t)$, we have
\begin{equation}  \label{d41}
\xi^{-} (t) \sim \sum_{j = 1}^{+ \infty} \sum_{m=0}^{M'_{j}} g_{j,m}^{-} (z^{-}) ( - \mu_{j} t^{m} + m t^{m-1} ) e^{- \mu_{j} t}.
\end{equation}

\begin{proposition}\sl \label{tr1000}
 If $j<\jj$, then $M'_j=0$. 
We also have $M'_{\jj}\leq 1$,
and $M'_{\jj}=0$  when $\CI_1(2\lambda_1)\neq \emptyset$. 
Moreover
\begin{equation}\label{tr1000bis}
(g^-_{\jj,1})^\beta=
\left\{
\begin{array}{cl}
\ds \frac1{4\lambda_1}\sum_{\vert\alpha\vert=2}\frac{\partial^{\alpha+\beta} V(0)}{\alpha !} (g_1^-(z^-))^\alpha
& \mbox{for }\beta\in\CI_1(2\lambda_1),\\
0& \mbox{for }\beta\notin\CI_1(2\lambda_1).
\end{array}
\right .
\end{equation}
and, for $\vert \beta\vert=1$, $\beta\notin \CI_1(2\lambda_1)$,
\begin{equation}\label{tr1000ter}
(g^-_{\jj,0})^\beta=\frac{1}
{(2\lambda_1+\lambda\cdot\beta)(2\lambda_1-\lambda\cdot\beta)}
\sum_{\vert\alpha\vert=2}\frac{\partial^{\alpha+\beta} V(0)}{\alpha!} (g_1^-(z^-))^\alpha.
\end{equation}
 \end{proposition}

\begin{proof} 
 First of all, since $\partial_t\gamma^-(t)=H_p(\gamma^-(t))$, we can write
\begin{equation}\label{tr1001}
\partial_t\gamma^-(t)=F_p(\gamma^-(t)) +\CO(t^{2M'_1}e^{-2\lambda_1t}),
\end{equation}
where 
\begin{equation}\label{tr1002}
F_p=d_{(0,0)}H_p=
\left(
\begin{array}{cc}
0&I\\
\Lambda^2&0
\end{array}
\right ), \ \Lambda^2=\diag(\lambda^2_1,\dots,\lambda^2_n).
\end{equation}
We obtain
\begin{equation}\label{tr1004}
\sum_{1\leq j<\jj}\; \sum_{m=0}^{M_j'} (F_p+\mu_j)\gamma_{j,m}^-t^m=\sum_{1\leq j<\jj}\; \sum_{m=0}^{M_j'} \gamma_{j,m}^-mt^{m-1}e^{-\mu_jt}.
\end{equation}
Now suppose $j<\jj$ and $M'_j\geq 1$. We get, for this $j$, for some $\gamma^-_{j,M_j'}\neq 0$,
\begin{equation}\label{tr1005}
\left\{ \begin{aligned}
&(F_p+\mu_j)\gamma^-_{j,M_j'}=0, \\
&(F_p+\mu_j)\gamma^-_{j,M_j'-1}=M_j'\gamma^-_{j,M_j'},
\end{aligned}
\right .
\end{equation}
so that $\ker (F_p+\mu_j)\cap\im(F_p+\mu_j)\neq\{0\}$. Since $F_p$ is a diagonizable matrix,  this can easily be seen to be a contradiction.

Now we pass to  the study of $M_{\jj}'\,$. So far we have obtained that
\begin{equation}\label{tr1006}
\gamma^-(t)=\sum_{1\leq j<\jj} \gamma^-_{j}e^{-\mu_jt}+\sum_{m=0}^{M'_{\jj}} \gamma^-_{\jj,m}t^me^{-2\lambda_1t}+\CO(t^Ce^{-\mu_{\jj+1}t}),
\end{equation}
and we can write
\begin{equation}\label{tr1007}
H_p(x,\xi)=
\left(
\begin{array}
{c}
\xi\\[6pt]
\ds\Lambda^2x-\sum_{\vert\alpha\vert=2}\frac{\partial^\alpha\nabla V(0)}{\alpha!} x^\alpha +\CO(x^3)
\end{array}
\right).
\end{equation}
Thus we have
\begin{equation}\label{tr1008}
H_p(\gamma^-(t))=F_p \Big( \sum_{j<\jj} \gamma^-_{j}e^{-\mu_jt}+\sum_{m=0}^{M'_{\jj}} \gamma^-_{\jj,m}t^me^{-2\lambda_1t} \Big) +e^{-2\lambda_1t}A(\gamma_1^-)+\CO(e^{-(2\lambda_1+\varepsilon)t}),
\end{equation}
where, noticing that $\mu_j+\mu_{j'}=2\lambda_1$ if and only if $j=j'=1$,
\begin{equation}\label{tr1008etdemi}
A(\gamma_1^-)=
\left (
\begin{array}{c}
0\\[6pt]
\ds-\sum_{\vert\alpha\vert=2}
\frac{\partial^\alpha\nabla V(0)}{\alpha!}(g^-_1)^{\alpha}
\end{array}
\right ).
\end{equation}
For the terms of order $e^{-2\lambda_1t}$, we have, since $\partial_t\gamma^-(t)=H_p(\gamma^-(t))$,
\begin{equation}\label{tr1009}
(F_p+2\lambda_1)\sum_{m=0}^{M'_{\jj}} \gamma^-_{\jj,m}t^m=\sum_{m=0}^{M'_{\jj}} \gamma^-_{\jj,m}mt^{m-1}-A(\gamma_1^-).
\end{equation}
Thus, if we suppose that $M'_{\jj}\geq 2$, we obtain
\begin{equation}\label{tr1010}
\left\{
\begin{aligned}
&(F_p+2\lambda_1)\gamma^-_{\jj,M'_{\jj}}=0,\\
&(F_p+2\lambda_1)\gamma^-_{\jj,M'_{{\jj}-1}}=M'_{\jj}\;\gamma^-_{\jj,M'_{\jj}}\;.
\end{aligned}
\right .
\end{equation}
Then again we have $\gamma^-_{\jj,M'_{\jj}}\in \ker (F_p+2\lambda_1)\cap\im(F_p+2\lambda_1)$, a contradiction.

Eventually, if $\lambda_j\neq2\lambda_1$ for all $j$, then $ \ker (F_p+2\lambda_1)=\{0\}$. Therefore, if we suppose that  $M'_{\jj}=~1$, we see that $\gamma_{\jj,1}\neq 0$ satisfies the first equation in \eqref{tr1010} and  we get a contradiction.

Now we compute $\gamma^-_{\jj}(t)=\gamma^-_{\jj,1}t+\gamma^-_{\jj,0}$. We have
\begin{equation}\label{tr1011}
\left\{
\begin{aligned}
&(F_p+2\lambda_1)\gamma^-_{\jj,1}=0,\\
&(F_p+2\lambda_1)\gamma^-_{\jj,0}=\;\gamma^-_{\jj,1}-A(\gamma_1^-),
\end{aligned}
\right .
\end{equation}
and we see that $\gamma^-_{\jj,1}=\Pi\gamma^-_{\jj,1}=\Pi A(\gamma_1^-)$, where
 $\Pi$ is the projection on the eigenspace of $F_p$ associated to $-2\lambda_1$. We denote by $e_j=(\delta_{i,j}\otimes 0)_{i=1,\dots,n}$ and $\varepsilon_j=(0\otimes \delta_{i,j})_{i=1,\dots,n}$ for $j=1,\dots,n$, so that $(e_1,\dots e_n, \varepsilon_1,\dots, \varepsilon_n)$ is the canonical basis of $\R^{2n}=T_{(0,0)}T^*\R^n$. Then it is easy to check that , for all $j$,  $v_j^\pm=e_j\pm\lambda_j1\varepsilon_j$ is an eigenvector of $F_p$ for the eigenvalue $\pm\lambda_j$. In the basis $\{e_1,\varepsilon_1,\dots,e_n,\varepsilon_n\}$ the projector $\Pi$ is block diagonal 
and, if $K_j=\Vect (e_j,\varepsilon_j)$, we have
\begin{equation}\label{tr1020}
\Pi_{\vert_{K_j}}=
\left\{
\begin{array}{cl}
\begin{pmatrix}1/2&-1/4\lambda_1\\ -\lambda_1&1/2\end{pmatrix}& \mbox{for } j\in \CI_1(2\lambda_1),\\
0&\mbox{for } j\notin \CI_1(2\lambda_1).\\
\end{array}
\right.
\end{equation}
Therefore, we obtain
\begin{equation}\label{tr1021}
(g^-_{\jj,1})^\beta=
\left\{
\begin{array}{cl}
\ds -\frac1{4\lambda_1}\sum_{\vert\alpha\vert=2}\frac{\partial^\beta\partial^\alpha V(0)}{\alpha !} (g_1^-(z^-))^\alpha
& \mbox{for } \beta\in\CI_1(2\lambda_1),\\
0& \mbox{for } \beta\notin\CI_1(2\lambda_1).
\end{array}
\right .
\end{equation}
Now suppose that $k\notin\CI_1(2\lambda_1)$. Then the second equality in \eqref{tr1011} restricted to $K_k$  gives
\begin{equation}\label{tr1022}
\left (
\begin{array}{cc}
2\lambda_1&1\\
\lambda_k^2&2\lambda_1
\end{array}
\right )
\Pi_k\gamma_{\jj,0}=-\Pi_kA(\gamma_1^-),
\end{equation}
where $\Pi_k$ denotes the projection onto $K_k$. Solving this system, one gets
\begin{equation}\label{tr1023}
(g^-_{\jj,0})_k=\frac{1}
{4\lambda_1^2-\lambda^2_k}\Pi_x\Pi_kA(\gamma_1^-),
\end{equation}
and, together with \eqref{tr1008etdemi}, this  ends the proof of Proposition \ref{tr1000}.
\end{proof}

\Subsection{Computation of the $\varphi_{j}^k$'s}

Here we compute the $\varphi^k_j$'s for $j\leq \jj$. We still omit the superscript $k$.
From \cite{bfrz}, we know that $\xi^{-} (t) = \nabla_{x} \varphi \big( t , x^{-} (t) \big)$, so that, using \eqref{d42}, 
\begin{align}\label{tr1400fir} 
\xi^{-} (t) =& \nabla \varphi_{+} ( x^{-} (t) ) + \nabla \varphi_{1} ( x^{-} (t) ) e^{- \lambda_{1} t} + \sum_{2\leq j < \widehat{\jmath}} \nabla \varphi_{j}( 0 ) e^{- \mu_{j} t }     \nonumber  \\
&+ \nabla \varphi_{\widehat{\jmath},2}( 0 ) t^2 e^{- 2 \lambda_{1} t } + \nabla \varphi_{\widehat{\jmath},1}( 0 ) t e^{- 2 \lambda_{1} t } + \nabla \varphi_{\widehat{\jmath},0}( 0 ) e^{- 2 \lambda_{1} t } + \widetilde\CO (e^{- \mu_{\widehat{\jmath}+1} t} ).
\end{align}
Since $\varphi_{+} = - \varphi_{-}$ and $\xi^{-} \in \Lambda_{-}$, we have $\nabla \varphi_{+} ( x^{-} (t) ) = - \xi^{-} (t)$, and we obtain first, by \eqref{d41},
\begin{equation}\label{tr1100}
\nabla \varphi_{j}( 0) =- 2 \mu_{j} g_{j}^{-} (z^{-}),
\end{equation}
for $1 \leq j < \widehat{\jmath}$.

Now we study $\varphi_{\jj}(t,x)=\varphi_{\jj,0}(x)+t\varphi_{\jj,1}(x)+t^2\varphi_{\jj,2}(x)$ when  $\CI_1(2\lambda_1)\neq \emptyset$.  It follows from \eqref{tr1400fir} that we have
\begin{equation}
\left\{
\begin{aligned}
&- 4 \lambda_{1} g_{\widehat{\jmath},1}^{-} (z^{-}) = \nabla \varphi_{\widehat{\jmath},1}( 0 ), \\
&- 4 \lambda_{1} g_{\widehat{\jmath},0}^{-} (z^{-}) +2g_{\widehat{\jmath},1}^{-} (z^{-}) = \nabla \varphi_{\widehat{\jmath},0} ( 0 ) +\nabla^{2}\varphi_{1}(0) g_{1}^{-} (z^{-}) .  
\end{aligned}
\right .
\label{d53}
\end{equation}

On the other hand, we have seen that, by \eqref{d9},  the functions $\varphi_{\jj,2}$, $\varphi_{\jj,1}$  and $\varphi_{\jj,0}$ satisfy 
\begin{equation}\label{tr1401}
\left\{ \begin{aligned}
&(L-2\lambda_1)\varphi_{\jj,2}=0, \\
&(L-2\lambda_1)\varphi_{\jj,1}=-2 \varphi_{\jj,2},\\
&(L-2\lambda_1)\varphi_{\jj,2}=-\varphi_{\jj,1}-\frac12\vert\nabla \varphi_1(0)\vert^2.
\end{aligned} \right.
\end{equation}
In particular $\varphi_{\jj,2}\in \ker (L-2\lambda_1)\cap \im (L-2\lambda_1)$ so  that (see \eqref{tr111}),
\begin{equation}\label{tr130}
\varphi_{\jj,2}(x)=\sum_{\alpha\in \CI_2(\lambda_1)}c_{2,\alpha}x^\alpha +\CO( x^3).
\end{equation}
Going back to \eqref{tr1400fir}, we notice that we obtain now
\begin{align}\label{tr1400} 
\xi^{-} (t) =& \nabla \varphi_{+} ( x^{-} (t) ) + \nabla \varphi_{1} ( x^{-} (t) ) e^{- \lambda_{1} t} + \sum_{2\leq j < \widehat{\jmath}} \nabla \varphi_{j}( 0 ) e^{- \mu_{j} t }     \nonumber  \\
& \nabla \varphi_{\widehat{\jmath},1}( 0 ) t e^{- 2 \lambda_{1} t } + \nabla \varphi_{\widehat{\jmath},0}( 0 ) e^{- 2 \lambda_{1} t } + \widetilde\CO (e^{- \mu_{\widehat{\jmath}+1} t} ),
\end{align}
and this equality  is consistent with Proposition \ref{tr1000}.

Then, \eqref{d35}  and \eqref{d36} give
\begin{equation}\label{tr131}
\varphi_{\jj,1}(x)=\sum_{\alpha\in \CI_1(2\lambda_1)}c_{1,\alpha}x^\alpha+
\sum_{\vert\alpha\vert=2} c_{1,\alpha} x^\alpha+\CO(x^3),
\end{equation}
and, by \eqref{d40}, we have
\begin{equation}\label{tr133}
\Psi((c_{1,\beta})_{\beta\in \CI_1(2\lambda_1)})=(-2c_{2,\alpha})_{\alpha\in\CI_2(\lambda_1)}.
\end{equation}
By \eqref{trtroups}, we also have for $\vert\alpha\vert=2$, $\alpha\notin\CI_2(\lambda_1)$,
\begin{equation}\label{troulah}
c_{1,\alpha}=\frac{1}{2\lambda_1-\lambda\cdot\alpha}
\sum_{\beta\in\CI_1(2\lambda_1)}\frac{\partial^{\alpha+\beta} \varphi_+(0)}{\alpha !} c_{1,\beta}.
\end{equation}

The function $\varphi_{\jj,0}(x)=\sum_{\vert\alpha\vert\leq 2} c_{0,\alpha} x^\alpha+\CO(x^3)$ satisfies (see \eqref{d10})
\begin{equation}\label{tr134}
(L -2\lambda_1) \varphi_{\jj,0} = - \varphi_{\jj,1} - \frac{1}{2}  \big\vert \nabla \varphi_{1} (x) \big\vert^{2}.
\end{equation}
First of all, the compatibility condition  \eqref{d27} gives
\begin{equation}\label{tr137}
\forall \alpha\in \CI_1(2\lambda_1), c_{1,\alpha}=-\nabla\varphi_1(0)\cdot\partial^\alpha\nabla\varphi_1(0), 
\end{equation}
so that in particular, by \eqref{tr133}, the function $\varphi_{\jj,2}$ is known up to $\CO(x^3)$ terms:
\begin{equation}\label{tr138}
\forall \alpha\in\CI_2(\lambda_1), \ c_{2,\alpha}=\frac12\sum_{\beta\in \CI_1(2\lambda_1)}\frac{\partial^{\alpha+\beta}\varphi_+(0)}{\alpha!}\nabla\varphi_1(0)\cdot\partial^\beta\nabla\varphi_1(0),
\end{equation}
and 
\begin{equation}\label{troulahlah}
\forall \alpha\notin\CI_2(\lambda_1), \vert\alpha\vert=2,\;
c_{1,\alpha}=-\frac{1}{2\lambda_1-\lambda\cdot\alpha}
\sum_{\beta\in\CI_1(2\lambda_1)}\frac{\partial^{\alpha+\beta} \varphi_+(0)}{\alpha !}\nabla \varphi_1(0)\cdot\partial^\beta\nabla \varphi_1(0).
\end{equation}

Now \eqref{d35} and \eqref{d36} give
\begin{equation}\label{tr135}
c_{0,0}=\varphi_{\jj,0} (0) = \frac{1}{4 \lambda_{1}} \vert \nabla \varphi_{1} (0) \vert^{2}, 
\end{equation}
and
\begin{equation}\label{tr136}
\forall \alpha\notin\CI_1(2\lambda_1),\vert \alpha\vert=1, c_{0,\alpha}=\frac{1}{2\lambda_1-\lambda\cdot\alpha}\nabla\varphi_1(0)\cdot\partial^\alpha\nabla\varphi_1(0).
\end{equation}
From the other compatibility condition \eqref{d28}, we know that
\begin{align}\label{tr139}
\nonumber
\Big( c_{1,\alpha}+\frac1{\alpha!}\nabla\varphi_1(0)\cdot\partial^\alpha\nabla\varphi_1(0)
&+\frac1{2}\sum_{
\fract{\beta,\gamma\in\CI_1(\lambda_1)}{\beta+\gamma=\alpha}
}
\partial^\beta\nabla\varphi_1(0)\cdot\partial^\gamma\nabla\varphi_1(0)
\\
&+\!\!\!\! \sum_{\fract{\vert\beta\vert=1}{\beta\notin\CI_1(2\lambda_1)}}
\!\!\!
\frac{\partial^{\alpha+\beta}\varphi_+(0)}{\alpha!}\frac{\nabla\varphi_1(0)\cdot\partial^\beta\nabla\varphi_1(0)}{2\lambda_1-\lambda\cdot\beta}
\Big)_{\alpha\in\CI_2(\lambda_1)}\in \im\Psi,
\end{align}
and, from \eqref{d40}, we obtain a relation between the  $(c_{0,\beta})_{\beta\in\CI_1(2\lambda_1)}$ and 
the $(c_{1,\alpha})_{\alpha\in \CI_2(\lambda_1)}$, namely
\begin{align}\label{tr140}
\nonumber
\forall\alpha\in \CI_2(\lambda_1),\; c_{1,\alpha}=&
-\frac1{\alpha!}\partial^\alpha\nabla \varphi_1(0)\cdot \nabla \varphi_1(0)-\frac1{2}\sum_{
\fract{\beta,\gamma\in\CI_1(\lambda_1)}{\beta+\gamma=\alpha}
}
\partial^\beta\nabla\varphi_1(0)\cdot\partial^\gamma\nabla\varphi_1(0)\\
&-\sum_{\beta\in\CI_1(2\lambda_1)}\frac{\partial^{\alpha+\beta}\varphi_+(0)}{\alpha !}c_{0,\beta}
-\!\!\!\! \sum_{\fract{\vert\beta\vert=1}{\beta\notin\CI_1(2\lambda_1)}}
\!\!\!
\frac{\partial^{\alpha+\beta}\varphi_+(0)}{\alpha!}\frac{\nabla\varphi_1(0)\cdot\partial^\beta\nabla\varphi_1(0)}{2\lambda_1-\lambda\cdot\beta}\cdotp
\end{align}
Using the second equation in \eqref{d53}, we obtain, for $\vert\beta\vert=1$,
\begin{equation}\label{tr1402}
c_{0, \beta}=-4\lambda_1(g_{\jj,0}^-(z^-))^ \beta +2(g_{\jj,1}^-(z^-))^\beta - \partial^\beta \nabla \varphi_1(0)\cdot g_1^-(z^-).
\end{equation}

At this point,  we have computed  the functions $\varphi_{\jj,1}(x)$ and $\varphi_{\jj,2}(x)$ up to $\CO(x^3)$, in terms of derivatives of $\varphi_+$ and $\varphi_1$, and of the $g^-_{\jj,m}(z^-)$.
We shall now use the expressions we have obtained in Section \ref{d55} and in Section \ref{s66} to give these functions in terms of $g_1^-$ and of derivatives of $V$ only.

 First of all,  by \eqref{tr130}, \eqref{tr138}, Lemma \eqref{d5} and Lemma \eqref{d6},  we obtain
\begin{align}  \label{d48}
\varphi_{\widehat{\jmath},2} (x)=&-\frac1{8\lambda_1}
  \sum_{\fract{\gamma \in\CI_1(2\lambda_1)}{\alpha,\beta \in\CI_2(\lambda_1)}}
\partial^{\beta+\gamma} V(0)\frac{(g^-_1(z^-))^{\beta}}{\beta!}\;
 \partial^{\alpha+\gamma}V(0) \frac{x^\alpha}{\alpha !}+\CO(x^3) .
 \end{align}
Then we have
\begin{equation}  \label{y19}
\varphi_{\widehat{\jmath},1}(x) =
-4 \lambda_{1} g_{\widehat{\jmath},1}^{-} (z^{-}) \cdot x+
\sum_{\alpha\in \CI_2(\lambda_1)}
c_{1,\alpha}x^\alpha
+\sum_{\fract{\vert\alpha\vert=2}{\alpha\notin \CI_2(\lambda_1)}}
c_{1,\alpha}x^\alpha
+\CO (x^3), 
\end{equation}
where the $c_{1,\alpha}$ are given by \eqref{tr140} and \eqref{tr1402} for $\alpha\in \CI_2(\lambda_1)$, and by \eqref{troulahlah} for $\alpha\notin \CI_2(\lambda_1)$.

$\bullet$ For $\vert\alpha\vert=2$, $\alpha\notin\CI_2(\lambda_1)$, we obtain by \eqref{troulah}, Lemma \ref{d5} and Lemma \ref{d6}, 
\begin{align}\label{tr200}
\nonumber
c_{1,\alpha}=&\frac{4\lambda_1^2}{(2\lambda_1+\lambda\cdot\alpha)
(2\lambda_1-\lambda\cdot\alpha)}   \\
&\times\sum_{\beta\in\CI_1(2\lambda_1)}
\frac{\partial^{\alpha+\beta}V(0)}{\alpha !}
\sum_{j=1}^n\frac{1}{(\lambda_1+\lambda_j)(3\lambda_1+\lambda_j)}
\partial_j\partial^\beta\nabla V(0)\cdot g_1^-(z^-) (g_1^-(z^-))_j.
\end{align}
Since $(g_1^-(z^-))_j=0$ but for $j\in\CI_1(\lambda_1)$, we get, changing notations a bit,
\begin{equation}\label{tr201}
c_{1,\alpha}=
\frac1{(2\lambda_1+\lambda\cdot\alpha)(2\lambda_1-\lambda\cdot\alpha)}
\sum_{
\fract{\gamma\in\CI_1(2\lambda_1)}{\beta\in\CI_2(\lambda_1)}
}
\frac{\partial^{\alpha+\gamma}V(0)}{\alpha !}\frac{\partial^{\beta+\gamma}V(0)}{\beta !}(g_1^-(z^-))^\beta.
\end{equation}

$\bullet$ Now we compute $c_{1,\alpha}$ for $\alpha\in\CI_2(\lambda_1)$.

For the last term in the R.H.S. of \eqref{tr140}, we obtain
\begin{align}\label{tr2031}
 -\sum_{\fract{\vert\beta\vert=1}{\beta\notin\CI_1(2\lambda_1)}} &
\!\!\!
\frac{\partial^{\alpha+\beta}\varphi_+(0)}{\alpha!}\frac{\nabla\varphi_1(0)\cdot\partial^\beta\nabla\varphi_1(0)}{2\lambda_1-\lambda\cdot\beta}=     \nonumber  \\
& \sum_{\fract{\gamma\in\CI_1\setminus\CI_1(2\lambda_1)}{\beta\in\CI_2(\lambda_1)}}
\frac{8\lambda_1^2}{(2\lambda_1-\lambda\cdot\gamma)(\lambda\cdot\gamma)(2\lambda_1+\lambda\cdot\gamma)^2}
\frac{\partial^{\alpha+\gamma}V(0)}{\alpha !}\frac{\partial^{\beta+\gamma}V(0)}{\beta !}(g_1^-(z^-))^\beta.
\end{align}

Using \eqref{tr1000bis} and \eqref{tr1402}, we have also
\begin{eqnarray}\label{tr203}
\nonumber
&&
-\sum_{\beta\in\CI_1(2\lambda_1)}\frac{\partial^{\alpha+\beta}\varphi_+(0)}{\alpha !}c_{0,\beta}
=\hskip 2cm\\
&&
\qquad
-\sum_{\gamma\in\CI_1(2\lambda_1)}
\frac{\partial^{\alpha+\gamma}V(0)}{\alpha !}(g_{\jj,0}^-(z^-))^\gamma
+
\frac1{4\lambda_1^2}
\sum_{
\fract{\gamma\in\CI_1(2\lambda_1)}{\beta\in\CI_2(\lambda_1)}
}
\frac{\partial^{\alpha+\gamma}V(0)}{\alpha !}
\frac{\partial^{\beta+\gamma}V(0)}{\beta !}(g_1^-(z^-))^\beta.
\end{eqnarray}

We pass to the computation of $-\frac 1{\alpha !}\partial^\alpha\nabla \varphi_1(0)\cdot \nabla \varphi_1(0)$ for $\alpha\in \CI_2(\lambda_1)$. We obtain
\begin{align}\label{tr204}
\nonumber
-\frac 1{\alpha !} & \partial^\alpha \nabla \varphi_1(0)\cdot \nabla \varphi_1(0)=
-\sum_{\beta\in\CI_2(\lambda_1)}
\frac{\partial^{\alpha+\beta}V(0)}{\alpha!\beta!}(g_1^-(z^-))^\beta
\\
\nonumber
&
-\frac1{4}\sum_{j,p,k=1}^n
\!\!\!
\sum_{
\fract{\beta,\gamma\in\CI_2}{\beta+\gamma=\alpha+1_p+1_j}
}
\!\!\!\!\!\!
\frac{((\alpha+1_p)_j+1)(\alpha_p+1)}{(\lambda_k+\lambda\cdot \beta)(\lambda_k+\lambda\cdot \gamma)}
\frac{\partial^{\beta+1_k}V(0)}{\beta !}
\frac{\partial^{\gamma+1_k}V(0)}{\gamma !}
(g_1^-(z^-))_j(g_1^-(z^-))_p\\
\nonumber
&
+{2\lambda_1}\sum_{j,p,k=1}^n
\sum_{
\fract{\beta,\gamma\in\CI_2}{\beta+\gamma=\alpha+1_p+1_j}
}
\frac{(\alpha_p+1)\gamma_j}{(\lambda_1-\lambda\cdot\gamma)(\lambda_1+\lambda\cdot\gamma)(\lambda_j+\lambda\cdot\beta)}\times
\\
\nonumber
&\hskip 5cm \times\frac{\partial^{\beta+1_j} V(0)}{\beta!} \frac{\partial^{\gamma+1_k}V(0)}{\gamma !}(g_1^-(z^-))_k(g_1^-(z^-))_p
\\
&=I+II+III.
\end{align}
Writing $\delta=1_j+1_p$, we get
\begin{equation}\label{tr205}
II=
-\frac1{2}\sum_{k=1}^n
\sum_{
\fract{\beta,\gamma,\delta\in\CI_2}{\beta+\gamma=\alpha+\delta}
}
\!\!\!
\frac{(\alpha+\delta)!}{(\lambda_k+\lambda\cdot \beta)(\lambda_k+\lambda\cdot \gamma)}
\frac{\partial^{\beta+1_k}V(0)}{\beta !}
\frac{\partial^{\gamma+1_k}V(0)}{\gamma !}
\frac{(g_1^-(z^-))^\delta}{\alpha !\; \delta !}\cdotp
\end{equation}
Since $\delta\in\CI_2(\lambda_1)$ (otherwise $(g_1^-(z^-))^\delta=0$), we have $\beta,\gamma\in\CI_2(\lambda_1)$ and, changing notations a bit,
\begin{equation}\label{tr204bis}
II=
-\frac1{2}\sum_{\beta\in\CI_2(\lambda_1)}
\frac{(\alpha+\beta)!}{\alpha !}
\sum_{
\fract{\gamma,\delta\in\CI_2(\lambda_1)}{\gamma+ \delta =\alpha+\beta}
}
\sum_{j=1}^n
\frac1{(2\lambda_1+\lambda_j)^2}
\frac{\partial_j\partial^{\gamma}V(0)}{\gamma !}
\frac{\partial_j\partial^{\delta}V(0)}{\delta !}
\frac{(g_1^-(z^-))^\beta}{\beta !}\cdotp
\end{equation}

In the last term $III$, we can suppose that $\gamma=1_j+1_q$ for some $q\in\{1,\dots,n\}$. Then $\gamma_j=\gamma !$ and, writing $\beta=1_a+1_b$ we have
\begin{align}\label{tr206}
III=\lambda_1 & \sum_{j,k,p=1}^n(\alpha_p+1)(g_1^-(z^-))_k(g_1^-(z^-))_p  \nonumber  \\
&\times \hspace{-20pt} \sum_{\fract{a,b,q\in\CI_1}{1_a+1_b+1_q=\alpha+1_p}} \frac{(\alpha_p+1)}{(\lambda_1-\lambda_j-\lambda_q)(\lambda_1+\lambda_j+\lambda_q)(\lambda_j+\lambda_a+\lambda_b)}
{\partial_{j,a,b}V(0)} {\partial_{j,q,k}V(0)}.
\end{align}
Since $\alpha\in\CI_2(\lambda_1)$ and $1_p\in\CI_1(\lambda_1)$ (otherwise $(g_1^-(z^-))_p=0$), we have $1_a,1_b, 1_q\in \CI_1(\lambda_1)$ so that we can write
\begin{equation}\label{tr207}
III=
-\sum_{j,k,p=1}^n(\alpha_p+1)\frac{\lambda_1}{\lambda_j(2\lambda_1+\lambda_j)^2}
(g_1^-(z^-))_k(g_1^-(z^-))_p
\!\!\!\!\!\!
\sum_{
\fract{a,b,q\in\CI_1}{1_a+1_b+1_q=\alpha+1_p}
}
\!\!\!\!\!\!
{\partial_{j,a,b}V(0)} {\partial_{j,q,k}V(0)}.
\end{equation}
Now it is easy to check, noticing that $(\alpha+1_p)_k\in\{1,2,3,4\}$ and examining each case, that
\begin{equation}\label{tr208}
\sum_{
\!\!\!\!\!\!\!\!\!\!\!\!
\fract{a,b,q\in\CI_1}{1_a+1_b+1_q=\alpha+1_p}
}
{\partial_{j,a,b}V(0)} {\partial_{j,q,k}V(0)}
=\frac{(\alpha+1_p)_k}{4}
\!\!\!\!
\sum_{
\fract{a,b,c,d\in\CI_1}{1_a+1_b+1_c+1_d=\alpha+1_p+1_k}
}
{\partial_{j,a,b}V(0)} {\partial_{j,c,d}V(0)}.
\end{equation}
Therefore, we have
\begin{align}\label{tr209}
\nonumber
III=-\frac14
\sum_{j,k,p=1}^n\frac{(\alpha+1_p+1_k)!}{\alpha !}&\frac{\lambda_1}{\lambda_j(2\lambda_1+\lambda_j)^2}
(g_1^-(z^-))_k(g_1^-(z^-))_p
\\
&\times \sum_{
\fract{a,b,c,d\in\CI_1}{1_a+1_b+1_c+1_d=\alpha+1_p+1_k}
}
{\partial_{j,a,b}V(0)} {\partial_{j,c,d}V(0)}.
\end{align}
Eventually, setting $\beta=1_p+1_k$, $\gamma=1_a+1_b$ and $\delta=1_c+1_d$, we get
\begin{equation}\label{tr210}
III=
-\sum_{\beta\in\CI_2(\lambda_1)} 
\frac{(\alpha+\beta)!}{\alpha !}
\sum_{
\fract{\gamma,\delta\in\CI_2(\lambda_1)}{\gamma+\delta=\alpha+\beta}
}
\sum_{j=1}^n
\frac{2\lambda_1}{\lambda_j(2\lambda_1+\lambda_j)^2}
\frac{\partial_{j}\partial^\gamma V(0)}{\gamma !} \frac{\partial_{j}\partial^\delta V(0)}{\delta !}
\frac{(g_1^-(z^-))^\beta}{\beta !}\cdotp
\end{equation}

We are left with the computation of 
\begin{align}\label{zg1}
\nonumber
-\frac1{2}\sum_{
\fract{\beta,\gamma\in\CI_1(\lambda_1)}{\beta+\gamma=\alpha}
}
&\partial^\beta\nabla\varphi_1(0)\cdot\partial^\gamma\nabla\varphi_1(0)=
-\frac1{2}\sum_{j=1}^n\sum_{
\fract{\beta,\gamma\in\CI_1(\lambda_1)}{\beta+\gamma=\alpha}
}
\partial_j\partial^\beta\varphi_1(0)\cdot\partial_j\partial^\gamma\varphi_1(0)\\
= - & \frac1{2}\sum_{j=1}^n
\frac{4\lambda^2_1}{\lambda_j^2(2\lambda_1+\lambda_j)^2}
\sum_{
\fract{\beta,\gamma\in\CI_1(\lambda_1)}{\beta+\gamma=\alpha}
}
\sum_{k,\ell=1}^n
\partial_j\partial_k\partial^\beta V(0)(g_1^-(z^-))_k
\partial_j\partial_\ell\partial^\gamma V(0)(g_1^-(z^-))_\ell.
\end{align}
At this point, we  notice that
\begin{align}\label{zg2}
\nonumber
-\frac1{2}&\sum_{\alpha\in\CI_2(\lambda_1)} 
\sum_{
\fract{\beta,\gamma\in\CI_1(\lambda_1)}{\beta+\gamma=\alpha}
}
\partial^\beta\nabla\varphi_1(0)\cdot\partial^\gamma\nabla\varphi_1(0)x^\alpha
\\
\nonumber
&=
-\frac1{2}\sum_{j=1}^n
\frac{4\lambda^2_1}{\lambda_j^2(2\lambda_1+\lambda_j)^2}
\sum_{
\fractt{\beta,\gamma\in\CI_1(\lambda_1)}{\alpha\in\CI_2(\lambda_1)}{\beta+\gamma=\alpha}
}
\sum_{k,\ell=1}^n
\partial_j\partial_k\partial^\beta V(0)(g_1^-(z^-))_k
\partial_j\partial_\ell\partial^\gamma V(0)(g_1^-(z^-))_\ell\;  x^\alpha
\\
\nonumber
&=
-\frac1{2}\sum_{j=1}^n
\frac{4\lambda^2_1}{\lambda_j^2(2\lambda_1+\lambda_j)^2}
\bigg \{
\sum_{
\alpha,\beta\in\CI_2(\lambda_1)
}
{(\alpha+\beta)!}
\sum_{
\fract{\gamma,\delta\in\CI_2(\lambda_1)}{\gamma+ \delta =\alpha+\beta}
}
\frac{\partial_j\partial^ \gamma V(0)}{\gamma !} 
\frac{\partial_j\partial^ \delta V(0)}{\delta !} \frac{x^\alpha}{\alpha !} \frac{(g_1^-(z^-))^\beta}{\beta !}\\
&
\qquad\qquad - 2\sum_{\alpha,\beta\in\CI_2(\lambda_1)}
\frac{\partial_j\partial^\alpha V(0)}{\alpha !} 
\frac{\partial_j\partial^\beta V(0)}{\beta !} x^\alpha (g_1^-(z^-))^\beta
\bigg\}
\end{align}

From \eqref{tr140}, \eqref{tr2031}, \eqref{tr203} \eqref{tr210}, and  \eqref{zg2}, we finally obtain that\
\begin{align}\label{aargh}
\nonumber
\sum_{\alpha\in\CI_2(\lambda_1)}&c_{1,\alpha}x^\alpha=
\sum_{
\fract{\gamma\in\CI_1\setminus\CI_1(2\lambda_1)}{\alpha,\beta\in\CI_2(\lambda_1)}
}
\frac{8\lambda_1^2}{(2\lambda_1-\lambda\cdot\gamma)(\lambda\cdot\gamma)(2\lambda_1+\lambda\cdot\gamma)^2}
\frac{\partial^{\alpha+\gamma}V(0)}{\alpha !}\frac{\partial^{\beta+\gamma}V(0)}{\beta !}(g_1^-(z^-))^\beta x^\alpha
\\
\nonumber
&-
\sum_{
\fract{\gamma\in\CI_1(2\lambda_1)}
{\alpha\in\CI_2(\lambda_1)}
}
\!\!\!
\frac{\partial^{\alpha+\gamma}V(0)}{\alpha !}(g_{\jj,0}^-(z^-))^\gamma x^\alpha
+
\frac1{4\lambda_1^2}
\sum_{
\fract{\gamma\in\CI_1(2\lambda_1)}{\alpha,\beta\in\CI_2(\lambda_1)}
}
\!\!\!
\frac{\partial^{\alpha+\gamma}V(0)}{\alpha !}
\frac{\partial^{\beta+\gamma}V(0)}{\beta !}(g_1^-(z^-))^\beta x^\alpha \\
\nonumber
&-
\sum_{\alpha,\beta\in\CI_2(\lambda_1)}
\frac{\partial^{\alpha+\beta}V(0)}{\alpha!\beta!}(g_1^-(z^-))^\beta x^\alpha
\\
\nonumber
&-
\frac1{2}\sum_{\alpha,\beta\in\CI_2(\lambda_1)}
{(\alpha+\beta)!}
\sum_{
\fract{\gamma,\delta\in\CI_2}{\gamma+ \delta =\alpha+\beta}
}
\sum_{j=1}^n
\frac1{(2\lambda_1+\lambda_j)^2}
\frac{\partial_j\partial^{\gamma}V(0)}{\gamma !}
\frac{\partial_j\partial^{\delta}V(0)}{\delta !}
\frac{(g_1^-(z^-))^\beta}{\beta !} \frac{x^\alpha}{\alpha !}
\\
\nonumber
&-
\sum_{\alpha, \beta\in\CI_2(\lambda_1)} 
(\alpha+\beta)!
\sum_{
\fract{\gamma,\delta\in\CI_2(\lambda_1)}{\gamma+\delta=\alpha+\beta}
}
\sum_{j=1}^n
\frac{2\lambda_1}{\lambda_j(2\lambda_1+\lambda_j)^2}
\frac{\partial_{j}\partial^\gamma V(0)}{\gamma !} \frac{\partial_{j}\partial^\delta V(0)}{\delta !}
\frac{(g_1^-(z^-))^\beta}{\beta !} \frac{x^\alpha}{\alpha !}
\\
\nonumber
&-2\sum_{
\alpha,\beta\in\CI_2(\lambda_1)
}
{(\alpha+\beta)!}
\sum_{
\fract{\gamma,\delta\in\CI_2(\lambda_1)}{\gamma+\delta=\alpha+\beta}
}
\sum_{j=1}^n
\frac{\lambda^2_1}{\lambda_j^2(2\lambda_1+\lambda_j)^2}
\frac{\partial_j\partial^ \gamma V(0)}{\gamma !} 
\frac{\partial_j\partial^\delta V(0)}{\delta !}\frac{(g_1^-(z^-))^\beta}{\beta !} \frac{x^\alpha}{\alpha !}\\
&
 +4\sum_{\alpha,\beta\in\CI_2(\lambda_1)}\sum_{j=1}^n
\frac{\lambda^2_1}{\lambda_j^2(2\lambda_1+\lambda_j)^2}
\frac{\partial_j\partial^ \alpha V(0)}{\alpha !} 
\frac{\partial_j\partial^\beta V(0)}{\beta !} x^\alpha (g_1^-(z^-))^\beta,
\end{align}
or, more simply,
\begin{align}\label{aargh2}
\nonumber
\sum_{\alpha\in\CI_2(\lambda_1)}&c_{1,\alpha}x^\alpha=
-
\sum_{
\fract{\gamma\in\CI_1(2\lambda_1)}
{\alpha\in\CI_2(\lambda_1)}
}
\!\!\!
\frac{\partial^{\alpha+\gamma}V(0)}{\alpha !}(g_{\jj,0}^-(z^-))^\gamma x^\alpha
+
\sum_{\alpha,\beta\in\CI_2(\lambda_1)}\frac{(g_1^-(z^-))^\beta}{\beta !} \frac{x^\alpha}{\alpha !}
\times\bigg\{
\\
\nonumber
&
\sum_{
{\gamma\in\CI_1\setminus\CI_1(2\lambda_1)}
}
\frac{8\lambda_1^2}{(2\lambda_1-\lambda\cdot\gamma)(\lambda\cdot\gamma)(2\lambda_1+\lambda\cdot\gamma)^2}
{\partial^{\alpha+\gamma}V(0)}{\partial^{\beta+\gamma}V(0)}
\\
\nonumber
&
+
\frac1{4\lambda_1^2}
\sum_{\gamma\in\CI_1(2\lambda_1)}
\!\!\!
{\partial^{\alpha+\gamma}V(0)}
{\partial^{\beta+\gamma}V(0)}
-
\partial^{\alpha+\beta}V(0)
\\
\nonumber
&-
\frac
{(\alpha+\beta)!}{2}
\sum_{
\fract{\gamma,\delta\in\CI_2}{\gamma+ \delta =\alpha+\beta}
}
\sum_{j=1}^n
\frac1{\lambda_j^2}
\frac{\partial_j\partial^{\gamma}V(0)}{\gamma !}
\frac{\partial_j\partial^{\delta}V(0)}{\delta !}
\\
&
 +4\sum_{j=1}^n
\frac{\lambda^2_1}{\lambda_j^2(2\lambda_1+\lambda_j)^2}
{\partial_j\partial^ \alpha V(0)}
{\partial_j\partial^\beta  V(0)}
\bigg \} .
\end{align}

\section{Computations after the critical point}
\label{sec7}

\Subsection{Stationary phase expansion in the outgoing region}

Now we compute the scattering amplitude starting from  \eqref{q6}. First of all, we change the cut-off function $\chi_{+}$ so
that the support of the right hand side of the scalar product in
\eqref{q6} is close to $(0,0)$.

\begin{figure}[!h]
\begin{picture}(0,0)%
\includegraphics{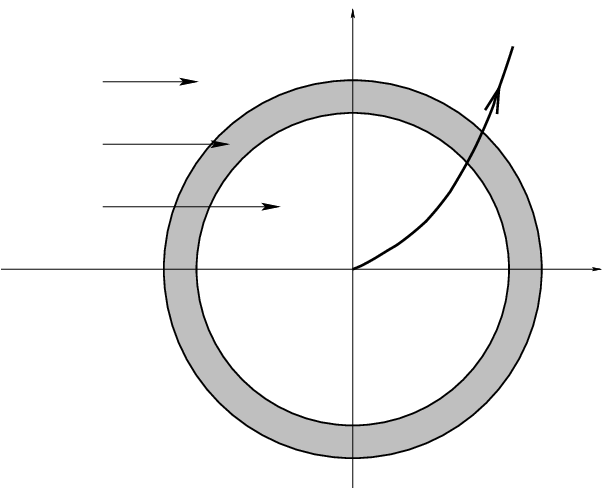}%
\end{picture}%
\setlength{\unitlength}{987sp}%
\begingroup\makeatletter\ifx\SetFigFont\undefined%
\gdef\SetFigFont#1#2#3#4#5{%
  \reset@font\fontsize{#1}{#2pt}%
  \fontfamily{#3}\fontseries{#4}\fontshape{#5}%
  \selectfont}%
\fi\endgroup%
\begin{picture}(11594,9269)(-771,-8183)
\put(9151,-286){\makebox(0,0)[lb]{\smash{{\SetFigFont{8}{9.6}{\rmdefault}{\mddefault}{\updefault}$\gamma_{\ell}^{+}$}}}}
\put(826,-361){\makebox(0,0)[rb]{\smash{{\SetFigFont{8}{9.6}{\rmdefault}{\mddefault}{\updefault}$\chi_{+} = 0$}}}}
\put(901,-1486){\makebox(0,0)[rb]{\smash{{\SetFigFont{8}{9.6}{\rmdefault}{\mddefault}{\updefault}$\supp ( \nabla \chi_{+} )$}}}}
\put(901,-2761){\makebox(0,0)[rb]{\smash{{\SetFigFont{8}{9.6}{\rmdefault}{\mddefault}{\updefault}$\chi_{+} = 1$}}}}
\end{picture}%
$\qquad \qquad$
\begin{picture}(0,0)%
\includegraphics{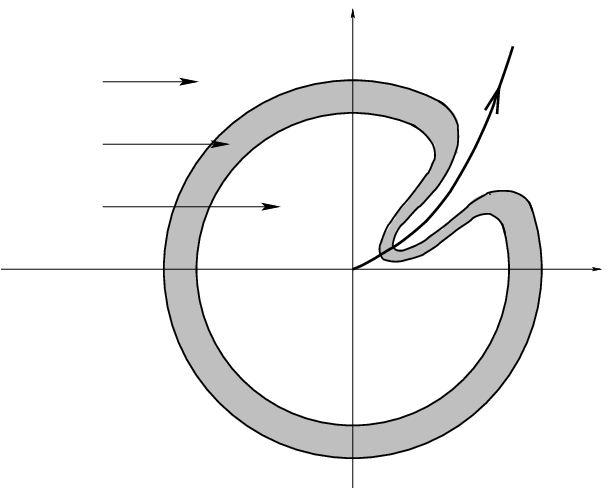}%
\end{picture}%
\setlength{\unitlength}{987sp}%
\begingroup\makeatletter\ifx\SetFigFont\undefined%
\gdef\SetFigFont#1#2#3#4#5{%
  \reset@font\fontsize{#1}{#2pt}%
  \fontfamily{#3}\fontseries{#4}\fontshape{#5}%
  \selectfont}%
\fi\endgroup%
\begin{picture}(11594,9269)(-771,-8183)
\put(9151,-286){\makebox(0,0)[lb]{\smash{{\SetFigFont{8}{9.6}{\rmdefault}{\mddefault}{\updefault}$\gamma_{\ell}^{+}$}}}}
\put(826,-361){\makebox(0,0)[rb]{\smash{{\SetFigFont{8}{9.6}{\rmdefault}{\mddefault}{\updefault}$\widetilde{\chi}_{+} = 0$}}}}
\put(901,-1486){\makebox(0,0)[rb]{\smash{{\SetFigFont{8}{9.6}{\rmdefault}{\mddefault}{\updefault}$\supp ( \nabla \widetilde{\chi}_{+} )$}}}}
\put(901,-2761){\makebox(0,0)[rb]{\smash{{\SetFigFont{8}{9.6}{\rmdefault}{\mddefault}{\updefault}$\widetilde{\chi}_{+} = 1$}}}}
\end{picture}%
\caption{The support of $\chi_{+}$ and $\widetilde{\chi}_{+}$ in $T^{*} \R^{n}$.}
\label{f1}
\end{figure}

Using Maslov's theory, we construct a function $v_{+}$ which coincides with $a_{+} (x,h) e^{i\psi_{+} (x) /h}$ out of a
small neighborhood of $\bigcup_{\ell} \gamma_{\ell}^{+} \cap (B (0,R_{+}+1) \times
\R^{n})$ and such that $v_{+}$ is a solution of $(P-E) v_{+} =0$
microlocally near $\bigcup_{\ell} \gamma_{\ell}^{+}$. Let $\widetilde{\chi}_{+}
(x, \xi ) \in C^{\infty} (T^{*} \R^{n})$ such that
$\widetilde{\chi}_{+} (x, \xi) = \chi_{+} (x)$ out of a small enough
neighborhood of $\bigcup_{\ell} \gamma_{\ell}^{+} \cap (B (0,R_{+}+1) \times
\R^{n})$. In particular, $(P-E) v_{+}$ is microlocally $0$ near the
support of $\chi_{+} - \widetilde{\chi}_{+}$. So, we have
\begin{align}
\bra u_{-} , [\chi_{+} , P] v_{+} \ket =& \bra u_{-} , [ \Op (
\widetilde{\chi}_{+} ) , P] v_{+} \ket + \bra u_{-} , (\chi_{+} - \Op
( \widetilde{\chi}_{+} ) ) (P-E) v_{+} \ket   \nonumber  \\
&- \bra (P-E) u_{-} , (\chi_{+} - \Op ( \widetilde{\chi}_{+} ) ) v_{+} \ket  \nonumber \\
=& \bra u_{-} , [ \Op ( \widetilde{\chi}_{+} ) , P] v_{+} \ket + \CO
(h^\infty) - \bra g_- e^{i\psi_- /h}, (\chi_{+} - \Op (
\widetilde{\chi}_{+} ) ) v_{+} \ket \nonumber  \\ 
=& \bra u_{-} , [ \Op ( \widetilde{\chi}_{+} ) , P] v_{+} \ket + \CO (h^\infty) ,
\end{align}
since the microsupport of $g_- e^{i\psi_- /h}$ and $\chi_{+} -
\widetilde{\chi}_{+}$ are disjoint. Thus, the scattering amplitude is
given by
\begin{equation}\label{q12}
\CA(\omega,\theta, E,h) = c(E)h^{-(n+1)/2} \bra u_{-} , [ \Op (
\widetilde{\chi}_{+} ) , P] v_{+} \ket + \CO (h^\infty) .
\end{equation}

Now we will prove that, modulo $\CO (h^{\infty})$, the only
contribution to the scattering amplitude in \eqref{q12} comes from the
values of the functions $u_{-}$ and $v_{+}$ microlocally on the
trajectories $\gamma_{\ell}^{+}$ and
$\gamma_{j}^{\infty}$. From \eqref{p1m}, the fact that $u_{-} = \CO
(h^{-C})$ and $(P-E) u_{-} =
0$ microlocally out of the microsupport of $g_{-}e^{-i \psi_{-} /h}$, and the usual propagation of
singularities theorem, we get
\begin{equation}
\MS ( u_{-} ) \subset \Lambda_{\omega}^{-} \cup \Lambda_{+} .
\end{equation}
Moreover, we have
\begin{equation}
\MS ( v_{+} ) \subset \Lambda^{+}_{\theta} .
\end{equation}
Now, let $f_{j}^{\infty}$ (resp. $f_{\ell}^{+}$) be $C_{0}^{\infty} (T^{*} \R^{n})$
functions with support in a small enough neighborhood of
$\gamma_{j}^{\infty}$ (resp. $\gamma_{\ell}^{+} \cap
\MS (v_{+})$) such that $f_{j}^{\infty} =1$ (resp. $f_{k}^{+} =1$) in a
neighborhood of $\gamma_{j}^{\infty}$
(resp. $\gamma_{\ell}^{+} \cap \MS (v_{+})$). In particular, we assume
that all these functions have disjoint support. Since $u_{-}$ and $v_{+}$ have
disjoint microsupport out of the support of the $f^{\infty}_{j}$ and
the $f^{+}_{\ell}$, we have
\begin{align}
\CA(\omega,\theta, E,h) =& c(E)h^{-(n+1)/2} \sum_{j} \bra \Op
(f^{\infty}_{j} ) u_{-} , \Op (f^{\infty}_{j} ) [ \Op (
\widetilde{\chi}_{+} ) , P] v_{+} \ket   \nonumber \\
&+ c(E)h^{-(n+1)/2} \sum_{\ell} \bra \Op
(f^{+}_{\ell} ) u_{-} , \Op (f^{+}_{\ell} ) [ \Op (
\widetilde{\chi}_{+} ) , P] v_{+} \ket + \CO (h^\infty)  \nonumber
\\
=& \CA^{reg} + \CA^{sing} .
\end{align}
Concerning the terms which contain $f_{j}^{\infty}$, $\CA^{reg}$, we are exactly in
the same setting as in \cite[Section 4]{RoTa89_01}. The computation
there gives
\begin{equation}
\CA^{reg} = \sum_{j=1}^{N_\infty}
\Big( \sum_{m \geq 0} a_{j,m}^{\rm reg} (\omega ,
\theta ,E) h^{m} \Big) e^{iS^\infty_j /h} + \CO(h^\infty).
\end{equation}
Now we compute $\CA^{sing}$. Proceeding as in Section \ref{q9} for $u_{-}$, one can show that
$v_{+}$ can be written as
\begin{equation}
v_{+} (x) = a_{+} (x,h) e^{i \nu_{\ell}^{+} \pi /2} e^{i \psi_{+} (x) /h} ,
\end{equation}
microlocally near any $\rho \in \gamma_{\ell}^{+}$ close enough
to $(0,0)$. Here $\nu_{\ell}^{+}$ is the Maslov index of $\gamma_{\ell}^{+}$. The phase
$\psi_{+}$ and the classical symbol $a_{+}$ satisfy the usual eikonal and transport
equations. In particular, as in \eqref{ma8} and \eqref{ma11}, we have
\begin{equation}  \label{a11}
\psi_{+} (x_{\ell}^{+} (t) )=-\int_{t}^{+ \infty} \vert \xi_{\ell}^{+}
(u) \vert^2-2E_01_{u>0} \, du=-
\int_{t}^{+\infty}\frac{1}{2} \vert \xi_{\ell}^{+} (u) \vert^2- V
(x_{\ell}^{+} (u) ) -E_0\sgn(u) \, du ,
\end{equation}
and $a_{+} (x,h) \sim \sum_{m} a_{+,m}(x) h^{m}$ with
\begin{equation}  \label{a36}
a_{+,0} (x_{\ell}^{+} (t)) =(2E_0)^{1/4} (D_\ell^+(t))^{-1/2}e^{i tz},
\end{equation}
where
\begin{equation}\label{mk3}
D_\ell^+(t)=\big\vert\det\frac{\partial x_+(t,z,\theta,E_0)}{\partial(t,z)}\vert_{z=z^+_\ell}\big\vert.
\end{equation}

We can chose $\widetilde{\chi}_{+}$ so that the microsupport of the symbol
of $\Op (f^{+}_{\ell} ) [ \Op ( \widetilde{\chi}_{+} ) , P]$ is
contained in a vicinity of such a point $\rho \in
\gamma_{\ell}^{+}$ (see Figure \ref{f1}). Then, microlocally near
$\rho$, we have
\begin{equation}  \label{a6}
\Op (f^{+}_{\ell} ) [ \Op ( \widetilde{\chi}_{+} ) , P] v_{+} =
\widetilde{a}_{+} (x,h) e^{i \nu_{\ell}^{+} \pi /2} e^{i \psi_{+} (x) /h},
\end{equation}
with
\begin{equation}
\widetilde{a}_{+} (x, h) = \sum_{m \geq 0} \widetilde{a}_{+,m} (x) h^{m+1} ,
\end{equation}
and
\begin{equation}  \label{a37}
\widetilde{a}_{+,0} (x) = -i \{ \widetilde{\chi}_{+} , p \} (x, \nabla
\psi_{+} (x)) a_{+,0} (x).
\end{equation}

From \cite[Section 5]{bfrz}, the Lagrangian manifold
\begin{equation*}
\{ (x , \nabla_{x} \varphi^{k} (t, x) ); \ \partial_{t} \varphi^{k} (t,x) =0 \} ,
\end{equation*}
coincides with $\Lambda_{\omega}^{-}$. In particular, since $\MS ( v_{+} ) \subset \Lambda^{+}_{\theta}$ and since there is no curve $\gamma^{\infty}(z^\infty_j)$ sufficiently closed to the critical point, the finite times in \eqref{a4} give a contribution $\CO (h^{\infty})$ to the scattering amplitude \eqref{q6}. In view of the equations \eqref{a4}, \eqref{a5} and \eqref{a6}, the principal contribution of $\CA^{sing}$ will come from the intersection of the manifolds $\Lambda_{\theta}^{+}$ and $\Lambda_{+}$. Recall that, from \ref{a6m}, the manifolds $\Lambda_{\theta}^{+}$ and $\Lambda_{+}$ intersect transversally along $\gamma_{\ell}^{+}$.

In particular, to compute $\CA^{sing}$, we can apply the method of stationary phase in the directions that are transverse to $\gamma_{\ell}^{+}$. For each $\ell$, after a linear and orthonormal change of variables, we can assume that $g^+_{\bf {l}}(z_{\ell}^{+} )$ is collinear to the $x_{\lll}$--direction, and that  $V (x)$ satisfies \ref{a2}. We denote $H_{x_{\lll}}^{\ell} = \{ y = ( y_{1} , \ldots , y_{n} ) \in \R^{n} ; \ y_{\lll} = x_{\lll} \}$ the hyperplane orthogonal to $( 0, \ldots , 0 , x_{\lll}, 0, \ldots ,0 )$. 

We shall compute $\CA^{sing}$ in the case where  there is only one incoming curve $\gamma_{k}^{-}$ and one outgoing curve $\gamma_{\ell}^{+}$. In the general case, $\CA^{sing}$ is simply given by the sum over $k$ and $\ell$ of such contributions. Using \eqref{q6}, \eqref{a4} and \eqref{a6}, we can write
\begin{align}
\CA^{sing} =& \frac{c(E)h^{-(n+1)/2}}{\sqrt{2 \pi h}} \iint e^{i ( \varphi^{k} (t,x) - \psi_{+} (x)) /h} \alpha^{k} (t,x,h) \overline{\widetilde{a}_{+}} (x,h) e^{- i \nu_{\ell}^{+} \pi /2}  d t \, d x    \nonumber  \\
=& \frac{c(E)h^{-(n+1)/2}}{\sqrt{2 \pi h}} \int_{x_{\lll}} \iint_{y \in H_{x_{\lll}}} e^{i ( \varphi^{k} (t,x) - \psi_{+} (x) ) /h} \alpha^{k} (t,x,h) \overline{\widetilde{a}_{+}} (x,h) e^{- i \nu_{\ell}^{+} \pi /2}  d t \, d y \, d x_{\lll}  .  \label{a9}
\end{align}
Let $\Phi (y) = \varphi^{k} (t,x_{\lll} , y) - \psi_{+} (x_{\lll} , y)$ be the phase function in \eqref{a9}. From \eqref{d2}--\eqref{d4}, we can write
\begin{equation} \label{a12}
\Phi (y) = S^{-}_{k} + (\varphi_{+} - \psi_{+} ) (x_{\lll} ,y) + \widetilde{\psi} (t,x_{\lll} ,y) ,
\end{equation}
where $\widetilde{\psi} = \CO (e^{- \lambda_{1} t})$ is an expandible function. Since the manifolds $\Lambda_{\theta}^{+}$ and $\Lambda_{+}$ intersect transversally along $\gamma_{\ell}^{+}$, the phase function $y \to (\varphi_{+} - \psi_{+} ) (x_{\lll} ,y)$ has a non degenerate critical point $y^{\ell} (x_{\lll}) \in H_{x_{\lll}}^{\ell} \cap \Pi_{x} \gamma_{\ell}^{+}$,
 and $x_{\lll} \mapsto y^{\ell} (x_{\lll})$ is $C^{\infty}$ for $x_{\lll} \neq 0$. Then, from the implicit function theorem, the function $\Phi$ has a unique critical point $y^{\ell} (t ,x_{\lll}) \in H_{x_{\lll}}^{\ell}$ for $t$ large enough depending on $x_{\lll}$. The function $(t,x_{\lll} ) \mapsto y^{\ell} (t ,x_{\lll})$ is expandible and we have
\begin{equation}  \label{a10}
y^{\ell} (t ,x_{\lll}) = y^{\ell} (x_{\lll}) - \Hess (\varphi_{+} - \psi_{+})^{-1} \big( y^{\ell} (x_{\lll}) \big) \nabla \varphi_{1} \big( y^{\ell} (x_{\lll}) \big) e^{- \mu_{1} t} + \widetilde{\CO} \big( e^{- \mu_{2} t} \big) .
\end{equation}
As a consequence, $\Phi \big( y^{\ell} (t ,x_{\lll}) \big)$ is also expandible.

Since $\varphi_{+}$ and $\psi_{+}$ satisfy the same eikonal equation, we get (see \eqref{ma5})
\begin{equation}
\partial_{t} ( \varphi_+ - \psi_{+} ) (x_{\ell}^{+} (t)) = \vert \xi_{\ell}^{+} (t) \vert^{2} - \vert \xi_{\ell}^{+} (t) \vert^{2} =0 .
\end{equation}
Thus, $( \varphi_+ - \psi_{+} ) (y^{\ell} (x_{\lll}) )$ does not depend of $x_{\lll}$ and is equal to
\begin{align}
( \varphi_+ - \psi_{+} ) (y^{\ell} (x_{\lll}) ) =& \lim_{t\to - \infty} ( \varphi_+ - \psi_{+} ) (x_{\ell}^{+} (t))  \nonumber \\
=& \int_{- \infty}^{+ \infty} \vert \xi_{\ell}^{+} (s) \vert^2-2E_01_{s>0} \, ds    \nonumber  \\
=& \int_{- \infty}^{+\infty}\frac{1}{2} \vert \xi_{\ell}^{+} (s) \vert^2- V (x_{\ell}^{+} (s) ) -E_0\sgn(s) \, ds\nonumber   \\
=& S_{\ell}^{+} ,
\end{align}
where we have used \eqref{a11}. Therefore, the phase function $\Phi$ at the critical point $y^{\ell} (t ,x_{\lll})$ is equal to
\begin{align}
\Phi \big( y^{\ell} (t ,x_{\lll}) \big) =& S^{-}_{k} + S_{\ell}^{+} + \sum_{\fract{m \in \N}{\mu_{m} \leq 2 \lambda_{1}}} \varphi_{m} \big( t , y^{\ell} (x_{\lll}) \big) e^{- \mu_{m} t}    \nonumber   \\
&- \frac{1}{2} \big( \Hess (\varphi_{+} - \psi_{+})^{-1} \big( y^{\ell} (x_{\lll}) \big) \nabla \varphi_{1} \big( y^{\ell} (x_{\lll}) \big) \cdot \nabla \varphi_{1} \big( y^{\ell} (x_{\lll}) \big) \big) e^{- 2 \mu_{1} t}  + \widetilde{\CO} (e^{- \widetilde{\mu} t}) ,  \label{a17}
\end{align}
where $\widetilde{\mu}$ is the first of the $\mu_{j}$'s such that $\mu_{j} > 2 \lambda_{1}$.

Using the method of the stationary phase for the integration with respect to $y \in H_{x_{\lll}}^{\ell}$ in \eqref{a9}, we get
\begin{equation}   \label{a21}
\CA^{sing} = \frac{c(E)h^{-(n+1)/2}}{\sqrt{2 \pi h}} (2 \pi h)^{(n-1)/2} \iint e^{i \Phi ( y^{\ell} (t ,x_{\lll}) ) /h} f^{\ell} (t ,x_{\lll} ,h) \, d t \, d x_{\lll} + \CO (h^{\infty}) .
\end{equation}
The $\CO (h^{\infty})$  term follows from the fact that the error term stemming from the stationary phase method can be integrated with respect to  time $t$, since $\alpha^{k} \in \CS^{0,2\re\Sigma(E)}$, with $\re \Sigma(E)>0$ (see the beginning of Section \ref{a20}). The symbol $f^{\ell} (t ,x_{\lll} ,h)$ is a classical expandible function of order $\CS^{1 , 2 \re \Sigma(E)}$ in the sense of Definition \ref{a13}:
\begin{equation}  \label{a22}
f^{\ell} (t ,x_{\lll} ,h) \sim \sum_{m \geq 0} f_{m}^{\ell} (t, x_{\lll} , \ln h ) h^{1+ m} ,
\end{equation}
where the $f_{m}^{\ell}$ are polynomials with respect to $\ln h$ and
\begin{equation}
f_{0}^{\ell} (t,x_{\lll} , \ln h ) = \alpha^{k}_{0} \big( t , y^{\ell} (t,x_{\lll}) \big) \overline{\widetilde{a}_{+,0}} \big( y^{\ell} (t,x_{\lll}) \big) e^{- i \nu_{\ell}^{+} \pi /2} \frac{e^{i \sgn \Phi_{\vert H_{x_{\lll}}^{\ell}}'' ( y^{\ell} (t,x_{\lll}) ) \pi/4}}{\big\vert \det \Phi_{\vert H_{x_{\lll}}^{\ell}}'' \big( y^{\ell} (t,x_{\lll}) \big) \big\vert^{1/2}} .  \label{a14}
\end{equation}
Using Proposition \ref{a15}, we compute the Hessian of $\Phi$, and we get
\begin{align*}
\psi_{+} '' \big( y^{\ell} (x_{\lll}) \big) =& \diag ( -\lambda_{1} , \ldots , - \lambda_{\lll -1} , \lambda_\lll , - \lambda_{\lll +1} , \ldots , - \lambda_{n}) + o (1) ,  \\
\varphi_{+} '' \big( y^{\ell} (x_{\lll}) \big) =& \diag (\lambda_{1} , \ldots , \lambda_{n} ) + o (1) .
\end{align*}
Then, for $x_{\lll}$ small enough and $t$ large enough depending on $x_{\lll}$, we have
\begin{align}
\big\vert \det \Phi_{\vert H_{x_{\lll}}^{\ell}}'' \big( y^{\ell} (t,x_{\lll}) \big) \big\vert^{1/2} &= \sqrt{ \prod_{j\neq \lll} 2 \lambda_{j}} + o(1) ,   \label{a38} \\
\sgn \Phi_{\vert H_{x_{\lll}}^{\ell}}'' \big( y^{\ell} (t,x_{\lll}) \big) &= n-1 , \label{a39}
\end{align}
as $x_{\lll}$ goes to $0$.

\Subsection{Behaviour of the phase function $\Phi$}

Suppose that $j\in \N$ is such that $j < \widehat{\jmath}$. From \eqref{y1}, we have
\begin{equation}
\varphi_{j}^{k} (x_{\ell}^{+} (s_{0})) = e^{- \mu_{j} (s - s_{0})} \varphi_{j}^{k} (x_{\ell}^{+} (s)) .
\end{equation}
Combining \eqref{d42} with \eqref{tr1100}, we get
\begin{align}
\varphi_{j}^{k} (x_{\ell}^{+} (s_{0})) =&  e^{\mu_{j} s_{0}} e^{- \mu_{j} s} \big( -2 \mu_{j} \<g_{j}^{-} (z_{k}^{-}) \vert g_{j}^{+} (z_{\ell}^{+}\>) e^{\mu_{j} s} + \CO (e^{2 \lambda_{1} s}) \big) \nonumber  \\
=& -2 \mu_{j} \big\< g_{j}^{-} (z_{k}^{-}) \big\vert g_{j}^{+} (z_{\ell}^{+}) \big\> e^{\mu_{j} s_{0}} . \label{a18}
\end{align}

We  suppose first  that we are in the case \textbf{(a)} of assumption \ref{a7}. Then, \eqref{a17} becomes
\begin{equation}   \label{a23}
\Phi \big( y^{\ell} (t ,x_{\lll}) \big) = S^{-}_{k} + S_{\ell}^{+} - 2 \mu_{{\bf k}} \big\< g_{{\bf k}}^{-} (z_{k}^{-}) \big\vert g_{{\bf k}}^{+} (z_{\ell}^{+}) \big\> e^{\mu_{{\bf k}} s (x_{\lll})} e^{- \mu_{{\bf k}} t} + \widetilde{\CO} (e^{- \mu_{{\bf k} +1} t}) .
\end{equation}
Here $s (x_{\lll})$ is such that $x_{\ell}^{+} ( s (x_{\lll})) = x^{\ell} (x_{\lll})$ and the $\widetilde{\CO} (e^{- \mu_{{\bf k} +1} t})$ is in fact expandible, uniformly with respect to $x_{\lll}$ when $x_{\lll}$ varies in a compact set avoiding $0$.

Suppose now that we are in the case \textbf{(b)} of assumption \ref{a7}. Of course, from \eqref{a18}, we have $\varphi_{j} \big( y^{\ell} (x_{\lll}) \big) =0$ for all $j < \widehat{\jmath}$. On the other hand, Corollary \ref{tr501} and \eqref{tr1401} imply
\begin{equation}
\varphi_{\widehat{\jmath} ,2}^{k} (x_{\ell}^{+} (s_{0})) = e^{- 2 \lambda_{1} (s - s_{0})} \varphi_{\widehat{\jmath} ,2}^{k} (x_{\ell}^{+} (s)) .
\end{equation}
Combining this with \eqref{d48}, we get
\begin{align}
\varphi_{\widehat{\jmath} ,2}^{k} (x_{\ell}^{+} (s_{0})) =& e^{2 \lambda_{1} s_{0}} e^{- 2 \lambda_{1} s} \Big( - \frac{1}{8 \lambda_{1}} \sum_{\fract{j \in \CI_{1}(2\lambda_1)}{\alpha, \beta \in \CI_{2}(\lambda_1)}} \frac{\partial^{\alpha + 1_{j}} V (0)}{\alpha !} \frac{\partial^{\beta + 1_{j}} V (0)}{\beta !} \big( g_{1}^{-} (z_{k}^{-}) \big)^{\alpha} \big( g_{1}^{+} (z_{\ell}^{+}) \big)^{\beta} e^{2 \lambda_{1} s}  \nonumber  \\
&+ \CO (e^{3 \lambda_{1} s}) \Big) \nonumber  \\
=& - \frac{1}{8 \lambda_{1}} \sum_{\fract{j \in \CI_{1}(2\lambda_1)}{\alpha, \beta \in \CI_{2}(\lambda_1)}} \frac{\partial^{\alpha + 1_{j}} V (0)}{\alpha !} \frac{\partial^{\beta + 1_{j}} V (0)}{\beta !} \big( g_{1}^{-} (z_{k}^{-}) \big)^{\alpha} \big( g_{1}^{+} (z_{\ell}^{+}) \big)^{\beta} e^{2 \lambda_{1} s_{0}} . \label{a19}
\end{align}
In particular, \eqref{a17} becomes, in that case,
\begin{align}
\Phi \big( y^{\ell} (t ,x_{\lll}) \big) =& S^{-}_{k} - S_{\ell}^{+} - \frac{1}{8 \lambda_{1}} \sum_{\fract{j \in \CI_{1}(2\lambda_1)}{\alpha, \beta \in \CI_{2}(\lambda_1)}} \frac{\partial^{\alpha + 1_{j}} V (0)}{\alpha !} \frac{\partial^{\beta + 1_{j}} V (0)}{\beta !} \big( g_{1}^{-} (z_{k}^{-}) \big)^{\alpha} \big( g_{1}^{+} (z_{\ell}^{+}) \big)^{\beta} e^{2 \lambda_{1} s (x_{\lll})}   \nonumber  \\
&\times t^{2} e^{-2 \lambda_{1} t} + \CO ( t e^{- 2 \lambda_{1} t})   \nonumber   \\
=& S^{-}_{k} + S_{\ell}^{+} + {\mathcal M}_{2}(k,\ell) t^{2} e^{-2 \lambda_{1} t} + \CO ( t e^{- 2 \lambda_{1} t}) .  \label{y2}
\end{align}
As in \eqref{a23}, the term $\CO (t e^{- 2 \lambda_{1} t})$ is in fact expandible uniformly with respect to $x_{\lll}$ when $x_{\lll}$ varies in a compact set avoiding $0$.

Eventually, we suppose  that we are in the case \textbf{(c)} of assumption \ref{a7}. Then  we obtain from \eqref{a18} and \eqref{a19} that $\varphi_{j} \big( y^{\ell} (x_{\lll}) \big) =0$ for all $j < \widehat{\jmath}$ and $\varphi_{\widehat{\jmath} ,2} \big( y^{\ell} (x_{\lll}) \big) =0$. With the last identity in mind, Equation \eqref{tr1401} on $\varphi_{\widehat{\jmath} ,1}^{k}$ implies
\begin{equation}
\varphi_{\widehat{\jmath} ,1}^{k} (x_{\ell}^{+} (s_{0})) = e^{- 2 \lambda_{1} (s - s_{0})} \varphi_{\widehat{\jmath} ,1}^{k} (x_{\ell}^{+} (s)) .
\end{equation}
In order to compute $\varphi_{\widehat{\jmath} ,1}^{k} (x_{\ell}^{+} (s))$, we put the expansion \eqref{fzz1} for $x_{\ell}^{+} (s)$ (with Proposition \ref{tr1000} in mind) into \eqref{y19}. The third term in \eqref{y19} will be, at least, $\CO (e^{(\mu_{2} + \mu_{1}) s}) = o (e^{2 \lambda_{1} s})$. Thank to \eqref{tr1000bis} and thanks to the fact that ${\mathcal M}_{2}(k,\ell) =0$, the first term in \eqref{y19} will give no contribution of order $s e^{2 \lambda_{1} s}$ and will be of the form
\begin{equation}  \label{y20}
-4 \lambda_{1} g_{\widehat{\jmath},1}^{-} (z^{-}) \cdot x_{\ell}^{+} (s) = - \sum_{\fract{j \in \CI_{1}}{\alpha \in \CI_{2} (\lambda_{1})}}  \frac{\partial_j\partial^{\alpha } V(0)}{\alpha !} (g_1^-(z^-))^\alpha (g_{\widehat{\jmath} ,0}^{+} (z^{+} ))_{j} e^{2 \lambda_{1} s} + \widetilde{\CO} (e^{\mu_{\widehat{\jmath}+1} s})
\end{equation}
It remains to study the contribution the second term in \eqref{y19}, as given in \eqref{aargh2}. As previously, the first term of the third line in \eqref{aargh2} will give a term of order $o (e^{2 \lambda_{1} s})$. The other terms will contribute to the order $e^{2 \lambda_{1} s}$ for
\begin{align}
&-\sum_{\fract{j \in \CI_{1}}{\alpha \in \CI_{2} (\lambda_{1})}}
 \frac{\partial_{j} \partial^{\alpha} V(0)}{\alpha !}(g_{\jj,0}^-(z^-))_{j} (g_{1}^{+} (z^{+} ))^{\alpha} + \sum_{\alpha , \beta \in \CI_{2} (\lambda_{1})} \frac{(g_1^-(z^-))^\alpha}{\alpha !} \frac{(g_1^+(z^+))^\beta}{\beta !} \times \nonumber  \\
&\qquad \times \Big( -\partial^{\alpha+\beta}V(0) + \sum_{j \in \CI_1 \setminus \CI_1(2\lambda_1)} \frac{4 \lambda_1^2}{\lambda_{j}^2 (4 \lambda_1^2 - \lambda^2_{j})}
\partial^{\alpha+\gamma}V(0)\partial^{\beta+\gamma}V(0)   \nonumber  \\
&\qquad \qquad
-\sum_{ \fractt{j \in \CI_{1}}{\gamma,\delta\in\CI_2(\lambda_1)}{\gamma+ \delta =\alpha+\beta}
} \frac{(\gamma+\delta)!}{\gamma !\; \delta !} \frac{1}{2\lambda_j^2} {\partial_j\partial^{\gamma}V(0)} {\partial_j\partial^{\delta}V(0)} \Big) .  \label{y21}
\end{align}
Thus, combining \eqref{y20} and \eqref{y21}, the  discussion above leads to 
\begin{align}
\varphi_{\widehat{\jmath} ,1}^{k} (x_{\ell}^{+} (s_{0})) =& e^{2 \lambda_{1} s_{0}} e^{- 2 \lambda_{1} s} \big( {\mathcal M}_{1}(k, \ell ) e^{2 \lambda_{1} s} + o (e^{2 \lambda_{1} s}) \big) \nonumber  \\
=& {\mathcal M}_{1}(k, \ell ) e^{2 \lambda_{1} s_{0}} .
\end{align}
In particular, \eqref{a17} becomes, in that case,
\begin{align}
\Phi \big( y^{\ell} (t ,x_{\lll}) \big) =& S^{-}_{k} + S_{\ell}^{+} + {\mathcal M}_{1}(k, \ell ) e^{2 \lambda_{1} s (x_{\lll})} t e^{-2 \lambda_{1} t} + \CO ( e^{- 2 \lambda_{1} t}) .
\end{align}
As above, the $\CO ( e^{- 2 \lambda_{1} t})$ is expandible  uniformly with respect to the variable $x_{\lll}$ when $x_{\lll}$ varies in a compact set avoiding $0$.

\Subsection{Integration with respect to time}

Now we perform the integration with respect to time $t$ in \eqref{a21}. We  follow the ideas of \cite[Section 5]{HeSj85_01} and \cite[Section 6]{bfrz}. Since $y^{\ell} (t , x_{\lll} )$ is expandible (see \eqref{a10}), and since $\Phi$ is $C^{\infty}$ outside of $x_{\lll} = 0$, the symbol $f^{\ell} (t ,x_{\lll} ,h)$ is expandible. 

We compute only the contribution of the principal symbol (with respect to $h$) of $f^{\ell}$, since  the other terms can be treated the same way, and the remainder term will give a contribution $\CO (h^{\infty})$ to the scattering amplitude. In other word, we compute
\begin{equation}  \label{a24}
\CA^{sing}_{0} = \frac{c(E)h^{-(n+1)/2}}{\sqrt{2 \pi h}} (2 \pi h)^{(n-1)/2} h \iint e^{i \Phi ( y^{\ell} (t ,x_{\lll}) ) /h} f^{\ell}_{0} (t ,x_{\lll} ) \, d t \, d x_{\lll} + \CO (h^{\infty}) .
\end{equation}

First, we assume that we are in the case \textbf{(a)} of the assumption \ref{a7}. In that case, $\Psi$ is given by \eqref{a23}. For $x_{\lll}$ fixed in a compact set outside from $0$, we set
\begin{align}
\tau =& \Phi \big( y^{\ell} (t ,x_{\lll}) \big) - (S^{-}_{k}+ S_{\ell}^{+})   \nonumber \\
=& - 2 \mu_{{\bf k}} \big\< g_{{\bf k}}^{-} (z_{k}^{-}) \big\vert g_{{\bf k}}^{+} (z_{\ell}^{+}) \big\> e^{\mu_{{\bf k}} s (x_{\lll})} e^{- \mu_{{\bf k}} t} + R (t, x_{\lll} )  ,  \label{y3}
\end{align}
and we perform the change of variable $t \to \tau$ in \eqref{a24}, and we assume for a moment that 
\begin{equation}\label{zm1}
\big\< g_{{\bf k}}^{-} (z_{k}^{-}) \big\vert g_{{\bf k}}^{+} (z_{\ell}^{+}) \big\><0.
\end{equation}
 Here $R (t, x_{\lll} ) = \widetilde{\CO} (e^{- \mu_{{\bf k} +1} t})$ is expandible. 
As in \cite[Section 5]{HeSj85_01} and \cite[Section 6]{bfrz}, we get
\begin{align}
e^{-t} \sim& \big( - 2 \mu_{{\bf k}} \big\< g_{{\bf k}}^{-} (z_{k}^{-}) \big\vert g_{{\bf k}}^{+} (z_{\ell}^{+}) \big\> e^{\mu_{{\bf k}} s (x_{\lll})} \big)^{-1/ \mu_{{\bf k}}} \tau^{1 / \mu_{{\bf k}}} 
\Big( 1 + \sum_{j =1}^{\infty} \tau^{\widehat{\mu}_{j} / \mu_{{\bf k}}} b_{j} (- \ln \tau , x_{\lll}) \Big)  \label{y4} \\
t \sim& - \frac{1}{\mu_{{\bf k}}} \ln \tau + \frac{1}{\mu_{{\bf k}}} \ln \big( - 2 \mu_{{\bf k}} \big\< g_{{\bf k}}^{-} (z_{k}^{-}) \big\vert g_{{\bf k}}^{+} (z_{\ell}^{+}) \big\> e^{\mu_{{\bf k}} s (x_{\lll})} \big) + \sum_{j =1}^{\infty} \tau^{\widehat{\mu}_{j} / \mu_{{\bf k}}} b_{j} (- \ln \tau , x_{\lll})   \\
\tau \frac{d t}{d \tau} \sim& - \frac{1}{\mu_{{\bf k}}} + \sum_{j =1}^{\infty} \tau^{\widehat{\mu}_{j} / \mu_{{\bf k}}} b_{j} (- \ln \tau , x_{\lll}) ,  \label{a32}
\end{align}
where the $b_{j}$'s change from line to line. These expansions  are valid in the following sense:

\begin{definition} \sl \label{y5}
Let $f ( \tau ,y)$ be defined on $] 0, \varepsilon [ \times U$ where $U \subset \R^{m}$. We say that $f = \widehat{\CO} (g ( \tau ))$ (resp. $f = \widehat{o} (g ( \tau ))$), where $g ( \tau )$ is a non-negative function defined in $]0, \varepsilon[$ if and only if for all $\alpha \in \N$ and $\beta \in \N^{m}$,
\begin{equation}
( \tau \partial_{\tau} )^{\alpha} \partial_{y}^{\beta}  f( \tau , y) = \CO (g (\tau)) ,
\end{equation}
(resp. $o (g (\tau))$) for all $( \tau , y ) \in ]0, \varepsilon [ \times U$.
\end{definition}

\noindent
Thus, an expression like $f \sim \sum_{j =1}^{\infty} \tau^{\widehat{\mu}_{j} / \mu_{{\bf k}}} f_{j} (- \ln \tau , x_{\lll})$, where $f_{j} (- \ln \tau , x_{\lll})$ is a polynomial with respect to $\ln \tau$,  like in \eqref{y4}--\eqref{a32}, means that, for all $J \in \N$,
\begin{equation}  \label{y6}
f ( \tau ,x) - \sum_{j =0}^{J} \tau^{\widehat{\mu}_{j} / \mu_{{\bf k}}} f_{j} (- \ln \tau , x_{\lll}) = \widehat{\CO} ( \tau^{\widehat{\mu}_{J} / \mu_{{\bf k}}} ).
\end{equation}
We shall say that such symbols $f$ are called expandible near $0$.

Since $f^{\ell}_{0} (t ,x_{\lll} ,h)$ is expandible (see Definition \ref{expdeb}) with respect to $t$, this symbol is also expandible near $0$ with respect to $\tau$ in the previous sense. In particular, we get
\begin{equation}  \label{a27}
\widetilde{f}^{\ell}_{0} ( \tau ,x_{\lll} ) = - f^{\ell}_{0} (t ,x_{\lll} ) \tau \frac{d t}{d \tau} \sim \sum_{j=0}^{\infty} \tau^{(\Sigma(E) + \widehat{\mu_{j}})/ \mu_{{\bf k}}} \widetilde{f}^{\ell}_{0,j} ( - \ln \tau , x_{\lll} ),
\end{equation}
where the $\widetilde{f}^{\ell}_{0,j}$'s are polynomials with respect to $\- \ln \tau$. The principal symbol $\widetilde{f}^{\ell}_{0,0}$ is independent on $\- \ln \tau$ and we have
\begin{equation}\label{zl1}
\widetilde{f}^{\ell}_{0,0} ( x_{\lll} ) = 
\frac{1}{\mu_{{\bf k}}} \big( - 2 \mu_{{\bf k}}
 \big\< g_{{\bf k}}^{-} (z_{k}^{-}) \big\vert g_{{\bf k}}^{+} (z_{\ell}^{+}) \big\> 
 e^{\mu_{{\bf k}} s (x_{\lll})} \big)^{-\Sigma(E) / \mu_{{\bf k}}} f^{\ell}_{0,0} (x_{\lll}) .
\end{equation}
In that case, \eqref{a24} becomes
\begin{equation}  \label{a26}
\CA^{sing}_{0} = \frac{c(E) h^{-1/2}}{(2 \pi)^{1 - n/2}} e^{i (S^{-}_{k} + S_{\ell}^{+}) /h} \iint_{0}^{+ \infty} e^{i \tau /h} \widetilde{f}^{\ell}_{0} (\tau ,x_{\lll}) \, \frac{d \tau}{\tau} \, d x_{\lll} + \CO (h^{\infty}) .
\end{equation}
Note that $\widetilde{f}^{\ell}_{0} (\tau ,x_{\lll})$ has in fact a compact support with respect to $\tau$. Now, using Lemma \ref{a25}, we can perform  the integration with respect to $t$ of each term in the right hand side of \eqref{a27}, modulo a term $\CO (h^{\infty})$ (see \eqref{a28}--\eqref{a29} in Lemma \ref{a25}). Then, we get
\begin{equation}  \label{a30}
\CA^{sing}_{0} = \frac{c(E) h^{-1/2}}{(2 \pi)^{1 - n/2}} e^{i (S^{-}_{k} + S_{\ell}^{+}) /h} \sum_{j=0}^{+ \infty} \widehat{f}_{j} ( \ln h) h^{(\Sigma(E) + \widehat{\mu}_{j})/ \mu_{{\bf k}}} ,
\end{equation}
where $\widehat{f}_{j} (\ln h)$ is a polynomial in respect to $\ln h$. $\widehat{f}_{0}$ does not depend on $h$ and we have
\begin{equation}  \label{a33}
\widehat{f}_{0}  = \Gamma ( \Sigma(E)/ \mu_{{\bf k}} ) (-i)^{-\Sigma(E)/\mu_{{\bf k}}} \int \widetilde{f}^{\ell}_{0,0} ( x_{\lll} ) \, d x_{\lll} .
\end{equation}

To finish the proof, it remains to perform the integration with respect to $x_{\lll}$ in \eqref{a33}. From \eqref{a14} and  \eqref{zl1}, it becomes
\begin{align}
\widehat{f}_{0} =& \Gamma ( \Sigma(E)/ \mu_{{\bf k}} ) 
\frac{1}{\mu_{{\bf k}}} \int \big( 2 i \mu_{{\bf k}} 
\big\< g_{{\bf k}}^{-} (z_{k}^{-}) \big\vert g_{{\bf k}}^{+} (z_{\ell}^{+}) \big\> 
e^{\mu_{{\bf k}} s (x_{\lll})} \big)^{-\Sigma(E)/ \mu_{{\bf k}}}   \nonumber  \\
&
\times \alpha_{0,0} (y^{\ell} \big( x_{\lll}) \big) 
\overline{\widetilde{a}_{+,0}} \big( y^{\ell} (x_{\lll}) \big) 
e^{- i \nu_{\ell}^{+} \pi /2} \frac{e^{i \sgn \Phi_{\vert H_{x_{\lll}}^{\ell}}'' ( y^{\ell} (x_{\lll}) ) \pi/4}}{\big\vert \det \Phi_{\vert H_{x_{\lll}}^{\ell}}'' \big( y^{\ell} (x_{\lll}) \big) \big\vert^{1/2}} \, d x_{\lll} .  \label{a34}
\end{align}
Now we make the change of variable $x_{\lll} \mapsto s$ given by $y^{\ell} (x_{\lll} )  = x_{\ell}^{+} (s)$ (then $s ( x_{\lll} ) = s$). In particular,
\begin{equation}  \label{a40}
d x_{\lll} = \partial_{s} ( x_{\ell ,\lll}^{+} (s) ) d s = \lambda_{\lll} \vert g^{+}_{\lll} (z_{\ell}^{+}) \vert e^{\lambda_{\lll} s} (1 + o(1)) d s ,
\end{equation}
as $s \to - \infty$. In this setting, we get
\begin{equation}  \label{a41}
\alpha_{0,0} (x_{\ell}^{+} (s)) = \alpha_{0,0} (0) (1 + o (1)) ,
\end{equation}
as $s \to - \infty$, where $\alpha_{0,0} (0)$ is given in \eqref{a35}. We also have, from \eqref{a36} and \eqref{a37},
\begin{equation}  \label{a42}
\overline{\widetilde{a}_{+,0}} (x_{\ell}^{+} (s)) = - i \partial_{s} \big( \widetilde{\chi}_{+} (\gamma_{\ell}^{+} (s)) \big)(2E_0)^{1/4}(D_\ell^+)^{-1/2}e^{isz}.
\end{equation}
Then, putting \eqref{a38}, \eqref{a39}, \eqref{a40}, \eqref{a41} and \eqref{a42} in \eqref{a34}, we obtain
\begin{align}
\widehat{f}_{0} =&
 \Gamma ( \Sigma(E)/ \mu_{{\bf k}} ) \frac{-i}{\mu_{{\bf k}}} \int \big( 2 i \mu_{{\bf k}}
 \big\< g_{{\bf k}}^{-} (z_{k}^{-}) \big\vert g_{{\bf k}}^{+} (z_{\ell}^{+}) \big\> \big)^{-\Sigma(E)/ \mu_{{\bf k}}} 
 \alpha_{0,0} (0) \partial_{s} \big( \widetilde{\chi}_{+} (\gamma_{\ell}^{+} (s)) \big) e^{- i \nu_{\ell}^{+} \pi /2}   \nonumber  \\
&\times \frac{e^{i (n-1) \pi /4}}{\sqrt{ \ds \prod_{j\neq \lll} 2 \lambda_{j}}} \lambda_{\lll} \vert g^{+}_{\lll} (z_{\ell}^{+}) \vert (2E_0)^{1/4}(D_\ell^+)^{-1/2}e^{isz}   
 e^{-\Sigma(E) s} e^{\lambda_{\lll} s} (1+ o (1)) \, d s   \nonumber
\\
=& - \frac{e^{i (n+1) \pi /4}}{\mu_{{\bf k}}} 
 \Big( \prod_{j \neq \lll} 2 \lambda_j \Big)^{-1/2} 
 \lambda_{\lll} \vert g^{+}_{\lll} (z_{\ell}^{+}) \vert \Gamma ( \Sigma(E)/ \mu_{{\bf k}} ) \big( 2 i \mu_{{\bf k}} \big\< g_{{\bf k}}^{-} (z_{k}^{-}) \big\vert g_{{\bf k}}^{+} (z_{\ell}^{+}) \big\> \big)^{-\Sigma(E)/ \mu_{{\bf k}}}  \nonumber  \\
&\times e^{- i \nu_{\ell}^{+} \pi /2} \alpha_{0,0} (0)
 (2E_0)^{1/4}(D^+_\ell)^{-1/2} \int \partial_{s} \big( \widetilde{\chi}_{+} (\gamma_{\ell}^{+} (s)) \big) (1 + o(1)) \, d s .
\end{align}
Here the $o(1)$ does not depend on $\widetilde{\chi}_{+}$. Now, we choose a family of cut-off functions $(\widetilde{\chi}_{+}^{j})_{j \in \N}$ such that the support of $\partial_{t} \big( \widetilde{\chi}_{+}^{j} (\gamma_{\ell}^{+} (t)) \big)$ goes to $- \infty$ as $j \to + \infty$. We also assume that $\partial_{t} \big( \widetilde{\chi}_{+}^{j} (\gamma_{\ell}^{+} (t)) \big)$ is non negative (see Figure \ref{f1}). Then
\begin{align}
\widehat{f}_{0}  =&
 - \frac{e^{i (n+1) \pi /4}}{\mu_{{\bf k}}}  \Big( \prod_{j\neq \lll} 2 \lambda_{j} \Big)^{-1/2} \lambda_{\lll} \Gamma ( \Sigma(E)/ \mu_{{\bf k}} ) 
 e^{- i \nu_{\ell}^{+} \pi /2} e^{i\pi/4} ( 2 \lambda_1)^{3/2} e^{-i\nu_{k}^{-}\pi/2} \nonumber  
 \\
&\times  \vert g^{-}_{1} (z_{k}^{-})\vert\; \vert g^{+}_{\lll} (z_{\ell}^{+}) \vert \big( 2 i \mu_{{\bf k}} \big\< g_{{\bf k}}^{-} (z_{k}^{-}) \big\vert g_{{\bf k}}^{+} (z_{\ell}^{+}) \big\> \big)^{-\Sigma(E)/ \mu_{{\bf k}}}\\
&(2E_0)^{1/2} (D_k^-D^+_\ell)^{-1/2}\times (1 + o(1)) .  \label{y14}
\end{align}
as $j \to + \infty$. Since $\widehat{f}_{0}$ is also independent of $\widetilde{\chi}_{+}$, we obtain Theorem \ref{main} from \eqref{a30} and \eqref{a33}, in the case \textbf{(a)} and under the assumption \eqref{zm1}. When $\< g_{{\bf k}}^{-} (z_{k}^{-}) \vert g_{{\bf k}}^{+} (z_{\ell}^{+}) \>>0$, we set $\tau$ as the opposite of the R.H.S. of \eqref{y3}, and we obtain the result along the same lines (see  Remark \ref{mk1}).

Now we assume that we are in the case \textbf{(b)} of the assumption \ref{a7}. In that case, the phase function $\Psi$ is given by \eqref{y2}. For $x_{\lll}$ fixed in a compact set outside from $0$, we set, mimicking \eqref{y3},
\begin{align}
\tau =& \Phi \big( y^{\ell} (t ,x_{\lll}) \big) - (S^{-}_{k} + S_{\ell}^{+})   \nonumber \\
=& {\mathcal M}_{2}(k, \ell ) e^{2 \lambda_{1} s (x_{\lll})} t^{2} e^{-2 \lambda_{1} t} + R (t, x_{\lll} )
\end{align}
where $R (t, x_{\lll} ) = \CO ( t e^{- 2 \lambda_{1} t})$ is expandible with respect to $t$. As above, we assume that ${\mathcal M}_{2}(k, \ell )$ is positive (the other case can be studied the same way).

Following \eqref{y4}, we want to write $s := e^{-t}$ as a function of $\tau$. Since $t \mapsto \tau (t)$ is expandible with respect to $t$, we have
\begin{equation}  \label{y8}
\tau = {\mathcal M}_{2}(k, \ell ) e^{2 \lambda_{1} s (x_{\lll})} ( \ln s )^{2} s^{2 \lambda_{1}} (1 + r (s, x_{\lll})) ,
\end{equation}
where $r (s, x_{\lll} ) = \widehat{o} (1)$. In particular, $\partial_{s} \tau >0$ for $s$ positive small enough and then, for $\varepsilon >0$ small enough, $s \mapsto \tau (s)$ is invertible for $0< s < \varepsilon$. We denote $s (\tau )$ the inverse of this function. We look for $s ( \tau )$ of the form
\begin{equation} \label{y7}
s ( \tau ) = ( 2 \lambda_{1} )^{1/ \lambda_{1}} \Big( \frac{\tau}{{\mathcal M}_{2}(k, \ell ) e^{2 \lambda_{1} s (x_{\lll})}} \Big)^{1/2 \lambda_{1}} \frac{u ( \tau , x_{\lll})}{( - \ln \tau )^{1/ \lambda_{1}}} ,
\end{equation}
where $u ( \tau , x_{\lll})$ has to be determined. Using \eqref{y8}, the equation on $u$ is
\begin{align}
\tau =& {\mathcal M}_{2}(k, \ell ) e^{2 \lambda_{1} s (x_{\lll})} ( \ln s )^{2} s^{2 \lambda_{1}} (1 + r (s, x_{\lll}))  \nonumber  \\
=& \tau u^{2 \lambda_{1}} \Big( 1 - \frac{\ln \big( (2 \lambda_{1})^{-2} {\mathcal M}_{2}(k, \ell ) e^{2 \lambda_{1} s (x_{\lll})} \big)}{\ln \tau} + 2 \lambda_{1} \frac{\ln u}{\ln \tau} - 2 \frac{\ln ( - \ln \tau )}{\ln \tau} \Big)^{2}  \nonumber  \\
&\times \Big( 1 + r \Big( ( 2 \lambda_{1} )^{1/ \lambda_{1}} \Big( \frac{\tau}{{\mathcal M}_{2}(k, \ell ) e^{2 \lambda_{1} s (x_{\lll})}} \Big)^{1/2 \lambda_{1}} \frac{u}{(- \ln \tau )^{1/ \lambda_{1}}} , x_{\lll} \Big) \Big)  \nonumber  \\
=& \tau F ( \tau , u , x_{\lll} ),
\end{align}
where $F = u^{2 \lambda_{1}} (1 + \widetilde{r}(\tau , u, x_{\lll} ))$ and $\widetilde{r} = \widehat{o} (1)$ for $u$ close to $1$ (here $(u, x_{\lll} )$ are the variables $y$ in Definition \ref{y5}). In other word, to find $u$, we have to solve $F (t ,u ,x_{\lll} ) =1$.

First we remark that $u \mapsto F (\tau ,u , x_{\lll} )$ is real-valued and continuous. Since, for $\delta >0$ and $\tau$ small enough, $F ( \tau , 1- \delta , x_{\lll} ) < 1 < ( \tau , 1+ \delta , x_{\lll} )$, there exists $u \in [1- \delta , 1+ \delta ]$ such that $F ( \tau , 1+ \delta , x_{\lll} ) =1$. Thank to the discussion before \eqref{y7}, the function $s (\tau)$ is of the form \eqref{y7} with $u (\tau , x_{\lll} ) \in [1- \delta , 1+ \delta ]$, for $\tau$ small enough.

For $\tau >0$, the function $F$ is $C^{\infty}$ and, since $\widetilde{r} = \widehat{o} (1)$, we have
\begin{equation}
\partial_{u} \big( F (\tau , u, x_{\lll} ) -1 \big) (u(\tau , x_{\lll})) = 2 \lambda_{1} u^{2\lambda_{1} -1} (1 + o_{\tau} (1)) > \lambda_{1} ,
\end{equation}
for $\tau$ small enough. The notation $o_{\tau} (1)$ means a term which goes to $0$ as $\tau$ goes to $0$. Here we have used the fact that $u(\tau , x_{\lll} )$ is close to $1$. In particular, the implicit function theorem implies that $u (\tau , x_{\lll})$ is $C^{\infty}$.

We write $u = 1 + v (\tau , x_{\lll})$ and we known that $v \in C^{\infty}$ and $v = o_{\tau} (1)$. Differentiating the equality
\begin{equation}
1 = F (\tau ,u (\tau , x_{\lll}) , x_{\lll} ) = \big( u (\tau , x_{\lll}) \big)^{2 \lambda_{1}} \big( 1 + \widetilde{r}(\tau , u (\tau , x_{\lll}) , x_{\lll} ) \big) ,
\end{equation}
one can show that $v = \widehat{o} (1)$. Thus we have
\begin{align}
e^{-t} =& s ( \tau ) = ( 2 \lambda_{1} )^{1/ \lambda_{1}} \Big( \frac{\tau}{{\mathcal M}_{2}(k, \ell ) e^{2 \lambda_{1} s (x_{\lll})}} \Big)^{1/2 \lambda_{1}} \frac{1 + \widehat{r} (\tau , x_{\lll} )}{( - \ln \tau )^{1/ \lambda_{1}}},     \label{y9}    \\
t =& - \frac{\ln \tau}{2 \lambda_{1}} (1 + \widehat{r} (\tau , x_{\lll} ) ),    \\
\tau \frac{d t}{d \tau} =& - \frac{1}{2 \lambda_{1}} + \widehat{r} (\tau , x_{\lll} ),  \label{y10}
\end{align}
where $\widehat{r} (\tau , x_{\lll} ) = \widehat{o} (1)$ change from line to line.

Since $f^{\ell}_{0} (t ,x_{\lll} ,h)$ is expandible with respect to $t$, we get, from \eqref{y9}--\eqref{y10},
\begin{equation}  \label{y11}
\widetilde{f}^{\ell}_{0} ( \tau ,x_{\lll} ) = - f^{\ell}_{0} (t ,x_{\lll} ) \tau \frac{d t}{d \tau} = \tau^{\Sigma(E)/2 \lambda_{1}} ( -\ln \tau )^{-\Sigma(E)/\lambda_{1}}\big( \widetilde{f}^{\ell}_{0,0} (  x_{\lll} ) + \widehat{r} (\tau , x_{\lll} ) \big),
\end{equation}
where $\widehat{r} = \widehat{o} (1)$ and
\begin{equation}
\widetilde{f}^{\ell}_{0,0} ( x_{\lll} ) = (2 \lambda_{1})^{\Sigma(E)/\lambda_{1} -1} \big( {\mathcal M}_{2}(k, \ell ) e^{2 \lambda_{1} s (x_{\lll})} \big)^{-\Sigma(E) / 2\lambda_{1}} f^{\ell}_{0,0} (x_{\lll}) .
\end{equation}
In that case, \eqref{a24} becomes
\begin{equation}  \label{y12}
\CA^{sing}_{0} = \frac{c(E) h^{-1/2}}{(2 \pi)^{1 - n/2}} e^{i (S^{-}_{k} + S_{\ell}^{+}) /h} \iint_{0}^{+ \infty} e^{i \tau /h} \widetilde{f}^{\ell}_{0} (\tau ,x_{\lll}) \, \frac{d \tau}{\tau} \, d x_{\lll} + \CO (h^{\infty}) .
\end{equation}
Note that $\widetilde{f}^{\ell}_{0} (\tau ,x_{\lll})$ has in fact a compact support with respect to $\tau$. Now, using Lemma \ref{a25}, we can perform the integration with respect to  $t$ in \eqref{y12}, modulo an error term given by \eqref{a28}--\eqref{a29} in Lemma \ref{a25}. Then, we get
\begin{align}
\CA^{sing}_{0} =& \frac{c(E) h^{-1/2}}{(2 \pi)^{1 - n/2}} e^{i (S^{-}_{k} + S_{\ell}^{+}) /h} \Gamma ( \Sigma(E)/ 2 \lambda_{1} ) (-i)^{-\Sigma(E)/ 2 \lambda_{1}}    \nonumber  \\
&\qquad \times h^{\Sigma(E)/ 2 \lambda_{1}} (-\ln h)^{-\Sigma(E)/ \lambda_{1}} \Big( \int \widetilde{f}^{\ell}_{0,0} ( x_{\lll} ) \, d x_{\lll} + o(1) \Big) ,  \label{y13}
\end{align}
as $h$ goes to $0$. The rest of the proof follows that of \eqref{y14}.

At last, the proof of Theorem \ref{main} in the case \textbf{(c)} can be obtained along the same lines, and we omit it.

\appendix

\section{Proof of Proposition \ref{app1}}
\label{secappa}

We prove that $\Lambda_{\theta}^{+} \cap \Lambda_{+} \neq
\emptyset$. From the assumption (A2), the Lagrangian manifold $\Lambda_{+}$ can be
described, near $(0,0)\in T^* (\R^{n})$, as
\begin{equation}
\Lambda_{+} = \{ (x, \xi) ; \ x = \nabla \widetilde{\varphi}_{+} ( \xi )  \},
\end{equation}
for $\vert \xi \vert < 2 \varepsilon$, with $\varepsilon > 0$ small enough. For $\eta
\in \S^{n-1}$, let $( x (t, \eta ) , \xi (t,
\eta ) )$ be the bicharacteristic curve with initial condition
$(\widetilde{\varphi} ( \varepsilon \eta ) , \varepsilon \eta)$. We have
\begin{equation}  \label{az3}
\Lambda_{+} = \{ ( x (t, \eta ) , \xi (t, \eta ) ) ; \ t \in \R, \
\eta \in \S^{n-1} \} \cup \{ (0,0) \}.
\end{equation}
The function $\xi (t, \eta )$ is continuous on $\R \times \S^{n-1}$.
From the classical scattering theory (see \cite[Section 1.3]{DeGe97_01}),
we know that this function $\xi (t, \eta )$ converges uniformly to
\begin{equation}  \label{az1}
\xi ( \infty , \eta ) := \lim_{t \to + \infty} \xi (t, \eta ),
\end{equation}
as $t \to + \infty$ and $\xi (\infty , \eta ) \in \sqrt{2 E} \,
\S^{n-1}$.

Then, the function
\begin{equation}
F (t, \eta ) = \frac{ \xi ( \frac{t}{1-t} , \eta ) }{\vert  \xi (
\frac{t}{1-t} , \eta ) \vert } ,
\end{equation}
is well defined for $0 \leq t \leq 1$ with the convention $F(1, \eta )
= \xi (\infty , \eta ) / \sqrt{2 E}$. Here we used that $\vert \xi (t, \eta ) \vert
\neq 0$ for each $t \in [0, + \infty]$, $\eta \in \S^{n-1}$. The previous properties of $\xi (t, \eta)$ imply the continuity of $F (t, \eta )$ on $[0,1]\times \S^{n-1}$.

From \eqref{az3}, to prove that $\Lambda_{\theta}^{+} \cap \Lambda_{+} \neq \emptyset$ for all $\theta \in \S^{n-1}$, it is enough (equivalent) to show the surjectivity of $\eta \to F (1, \eta)$. But if $\eta \to F (1, \eta)$ is not onto, then $\im F (1, \cdot ) \subset \S^{n-1} \setminus \{$a point$\}$. And since $\S^{n-1} \setminus \{$a point$\}$ is a contractible space, $F (1 , \cdot )$ is homotopic to a constant map
\begin{equation}
f : \S^{n-1} \to \S^{n-1} .
\end{equation}
On the other hand, $F : [0,1] \times \S^{n-1} \longrightarrow \S^{n-1}$ gives a homotopy between $F (0, \cdot) = {\rm Id}_{\S^{n-1}}$ and $F (1, \cdot)$. In particular, we have
\begin{equation}
1 = \deg (F (0, \cdot)) = \deg (F (1, \cdot)) = \deg (f ( \cdot)) =0,
\end{equation}
which is impossible (see \cite[Section 23]{Fu95_01} for more details).

\section{A lower bound for the resolvent}

\label{d45}

Let $\chi \in C^{\infty} (]0,+\infty[)$ be a non-decreasing function such that
\begin{equation}
\chi (x) = \left\{ \begin{aligned}
&x&&\text{ for } 0< x <1  \\
&2 &&\text{ for } 2 <  x ,
\end{aligned} \right.
\end{equation}
Let also $\varphi\in C^\infty_0(\R)$ an even function such that $0\leq \varphi\leq 1$, $1_{[-1,1]}\prec\varphi$, and $\supp\varphi\subset [ -2,2]$.
We set
\begin{equation}
u (x) =\prod_{j=1}^{n} e^{i \lambda_{j} x_{j}^{2} /2h} 
\varphi \Big( \frac{x_j}{h^{\alpha}} \Big)  
\chi \Big( \frac{h^{\beta}}{ \vert x_{j} \vert^{1/2}} \Big)
=\prod_{j=1}^{n} u_j(x),
\end{equation}
where  $0< \alpha < 2 \beta$ will be fixed later on. 
The  $u_j$'s are of course $C^{\infty}$ functions, and we have
\begin{equation}
(P-E_{0} ) u =- \frac{h^{2}}{2} \Delta u(x) - \sum_{j=1}^{n} \frac{\lambda^2_{j}}{2} x^2_{j} u(x)+ \CO ( x^{3}  u (x)  ) .
\end{equation}

\begin{lemma}\sl \label{b11}
For any $h$ small enough, we have
\begin{gather}
h^{\beta n}\vert\ln h\vert^{n/2} \lesssim \Vert u \Vert_{L^{2} (\R^{n})} 
\lesssim h^{\beta n}\vert\ln h\vert^{n/2},
\\
\big\Vert \vert x\vert^3u(x) \big\Vert_{L^{2} (\R^{n})} 
\lesssim h^{3\alpha}h^{\beta n}\vert\ln h\vert^{n/2}.
\end{gather}
\end{lemma}

\begin{proof} First of all, the second estimate follow easily from the first one: we have
\begin{equation*}
\big\Vert \vert x\vert^3u(x) \big\Vert^2=\int_{\R^n}  \vert x\vert^6\vert u(x)\vert^2d x \lesssim h^{6\alpha}\Vert u\Vert^2,
\end{equation*}
since $u$ vanishes if $\vert x\vert>2h^{\alpha}$.
Thanks to the fact that $u$ is a product of $n$ functions of one variable, it is enough to estimate
\begin{equation*}
I=\int\varphi^2 \Big( \frac{t}{h^\alpha} \Big) \chi^2 \Big( \frac{h^\beta}{t^{1/2}} \Big) dt=2 \int_0^{2h^\alpha} \varphi^2 \Big( \frac{t}{h^\alpha} \Big) \chi^2 \Big( \frac{h^\beta}{t^{1/2}} \Big) dt.
\end{equation*}
We have
\begin{equation*}
2\int_{h^{2\beta}}^{h^\alpha}\chi^2 \Big( \frac{h^\beta}{t^{1/2}} \Big) d t
\leq 
I
\leq 
2\int_{h^{2\beta}}^{2h^\alpha}\chi^2 \Big( \frac{h^\beta}{t^{1/2}} \Big) dt+2 \int_0^{h^\beta} \chi^2 \Big( \frac{h^\beta}{t^{1/2}} \Big) d t,
\end{equation*}
so that 
\begin{equation*}
2\int_{h^{2\beta}}^{h^\alpha}\frac{h^{2\beta}}{t}d t
\leq 
I
\leq 
2\int_{h^{2\beta}}^{2h^\alpha}\frac{h^{2\beta}}{t}dt+2\int_0^{h^\beta}4 \, d t.
\end{equation*}
The first estimate follows from the fact that $2\beta-\alpha>0$, once we have noticed that
\begin{equation*}
\int_{h^{2\beta}}^{Ah^\alpha}\frac{h^{2\beta}}{t}dt=h^{2\beta} \big( (2\beta-\alpha)\vert \ln h\vert+\alpha\ln A \big).
\end{equation*}
\end{proof}

On the other hand, we have
\begin{equation*}
-\frac{h^2}2\Delta u(x)-\sum_{j=1}^{n} \frac{\lambda_{j}^{2}}{2} x^2_{j} u(x)= \sum_{k=1}^n\prod_{j\neq k} u_j(x_j) \Big( - \frac{h^2}2u''_k(x_k) - \frac{\lambda_{k}^{2}}{2} x^2_{k} u_k(x_k) \Big).
\end{equation*}
From Lemma \ref{b11}, we get
\begin{align}
\nonumber
\big\Vert (P-E_{0})u \big\Vert
\lesssim &
h^{\beta(n-1)}\vert \ln h\vert^{(n-1)/2}  \sup_{1 \leq k \leq n} \big\Vert h^2 u''_k(t)+ \lambda_{k}^{2} t^2u_k(t) \big\Vert + h^{3\alpha} h^{\beta n} \vert \ln h\vert^{n/2}
\\
\lesssim & \Big( h^{-\beta}\vert \ln h\vert^{-1/2} \sup_{1 \leq k \leq n} \big\Vert h^2 u''_k(t)+ \lambda_{k}^{2} t^2u_k(t) \big\Vert +h^{3\alpha} \Big) \Vert u \Vert .  \label{n4}
\end{align}
We also have
\begin{equation}\label{n1}
h^2 u''_k(t) + \lambda_{k} t^2u_k(t)=
e^{i\lambda_kt^2/2h} \big( h^2v''_h(t) + ih \lambda_{k} (2t\partial_t+1)v_h(t) \big),
\end{equation}
where we have set $v_h(t)=\varphi \big( \frac{t}{h^{\alpha}} \big)
\chi \big( \frac{h^{\beta}}{ \vert t \vert^{1/2}} \big)$. Notice that the right hand side of \eqref{n1} is an even function, so that we only have to consider $t>0$.
The point here, is that we have, for $t>0$,
\begin{equation}\label{n2}
(2t\partial_t+1) \Big( \chi \Big( \frac{h^{\beta}}{  t^{1/2}} \Big) \Big)=
-\frac{h^\beta}{t^{1/2}}\chi'\Big( \frac{h^{\beta}}{  t^{1/2}} \Big)
+\chi \Big( \frac{h^{\beta}}{  t^{1/2}} \Big)=
\left\{
\begin{aligned}
&2        &&\mbox{ if } 0<t<\frac {h^{2\beta}}{4},\\
&\CO(1) \ &&\mbox{ if }   \frac {h^{2\beta}}{4}<t< h^{2\beta},\\
&0        &&\mbox{ if }  h^{2\beta}<t.
\end{aligned}
\right.
\end{equation}
Therefore, we obtain
\begin{align}
\big\Vert (2t\partial_t+1)v_{h} \big\Vert^2 =&
2\int_0^{2h^{\alpha}} \Big(\varphi \Big( \frac{t}{h^{\alpha}} \Big)  
(2t\partial_t+1)\Big(\chi \Big( \frac{h^{\beta}}{ \vert t \vert^{1/2}} \Big)\Big)\Big)^2 d t \nonumber  \\
&+ 2\int_0^{2h^{\alpha}} \Big(2t\partial_t\Big(\varphi \Big( \frac{t}{h^{\alpha}} \Big)\Big)
\chi \Big( \frac{h^{\beta}}{ \vert t \vert^{1/2}} \Big)\Big )^2d t   \nonumber \\
\lesssim& \int_{0}^{h^{2\beta}} d t + \int_{h^\alpha}^{2h^{\alpha}}\frac{t^2}{h^{2\alpha}}\Big(\varphi' \Big( \frac{t}{h^{\alpha}} \Big) \chi \Big( \frac{h^{\beta}}{ \vert t \vert^{1/2}} \Big)\Big)^2 d t \lesssim h^{2\beta} . \label{n3}
\end{align}
On the other hand, an easy computation gives, still for $t>0$,
\begin{align}
\nonumber
v''_h(t)=&h^{-2\alpha}\varphi''\Big( \frac{t}{h^{\alpha}} \Big)  
\chi \Big( \frac{h^{\beta}}{  t^{1/2}} \Big)
-
\frac{h^{\beta-\alpha}}{t^{3/2}}
\varphi' \Big( \frac{t}{h^{\alpha}} \Big)  
\chi' \Big( \frac{h^{\beta}}{ t^{1/2}} \Big)
\\
&+
\frac{3h^\beta}{4t^{5/2}}\varphi \Big( \frac{t}{h^{\alpha}} \Big)  
\chi' \Big( \frac{h^{\beta}}{ t^{1/2}} \Big)
+
\frac{h^{2\beta}}{4t^3}
\varphi \Big( \frac{t}{h^{\alpha}} \Big)  
\chi'' \Big( \frac{h^{\beta}}{ t^{1/2}} \Big).
\end{align}
Computing the $L^2$--norm of each of these terms as in Lemma \ref{b11} and \eqref{n3}, we obtain
\begin{equation} \label{n5}
\Vert h^2v_h''\Vert\lesssim h^{2+\beta-2\alpha}+ h^{2+\beta-2\alpha}+h^{2-3\beta} + h^{2-3\beta},
\end{equation}
and, eventually, from \eqref{n4}, \eqref{n1}, \eqref{n3} and \eqref{n5},
\begin{equation*}
\big\Vert(P-E_{0} )u \big\Vert \lesssim \Big(h^{-\beta}\vert \ln h\vert^{-1/2} \big( h^{1 + \beta} + h^{2+\beta-2\alpha} + h^{2-3\beta} \big) +h^{3\alpha}\Big)\Vert u\Vert.
\end{equation*}
Therefore we obtain Proposition \ref{r222} if we can find $\alpha>0$ and $\beta>0$ such that
\begin{equation*}
2-2\alpha>1,\  2-4\beta>1,\  3\alpha>1\ \text{and}\  2\beta>\alpha,
\end{equation*}
and one can check that $\alpha=5/12$ and $\beta=11/48$ satisfies these four  inequalities.

\section{Lagrangian manifolds which are transverse to $\Lambda_{\pm}$}

 Let $\Lambda \subset p^{-1} (E_{0})$ be a Lagrangian manifold such that $\Lambda \cap \Lambda_{-}$ is transverse along a Hamiltonian curve $\gamma (t) = (x(t) , \xi  (t))$. Then, where exists $a \neq 0$ and $\nu \in \{ 1 , \ldots , n \}$ such that
\begin{equation}  \label{q24}
\gamma (t) = ( a + \CO (e^{- \varepsilon t}) ) e^{- \lambda_{\nu} t} ,
\end{equation}
as $t \to + \infty$. The vector $a$ is an eigenvector of
\begin{equation}
\left( \begin{array}{cc}
0 & Id \\
V'' (0) & 0
\end{array} \right) ,
\end{equation}
for the eigenvalue $\lambda_{\nu}$. Thus, up to a linear change of variable in $\R^{n}$, we can always assume that $\Pi_{x} a$ is collinear to the $x_{\nu}$--direction. The goal of this section is to prove the following geometric result.

\begin{proposition} \sl
For $t$ large enough, $\Lambda$ projects nicely on $\R^{n}_{x}$ near $\gamma (t)$. In particular, there exists $\psi \in C^{\infty} (\R^{n})$ defined near $\Pi_x\gamma$, unique up to a constant, such that $\Lambda = \Lambda_{\psi} := \{ (x, \nabla \psi (x)); \ x \in \R^{n} \}$. Moreover, we have
\begin{equation}
\psi '' (x (t)) =  \left(
\begin{array}{cccccccc}
\lambda_{1} & & & & & & &\\
& & \ddots & & & & & \\
& & & \lambda_{\nu -1} & & & & \\
& & & & - \lambda_{\nu} & & & \\
& & & & & \lambda_{\nu +1} & & \\
& & & & & & \ddots & \\
& & & & & & & \lambda_{n}
\end{array} \right)+ \CO ( e^{ - \varepsilon t}),
\end{equation}
as $t \to + \infty$. \label{a15}
\end{proposition}

\begin{remark} \sl
The same result hold in the outgoing region: If $\gamma = \Lambda \cap \Lambda_{+}$ is transverse, $\Lambda$ projects nicely on $\R_{x}^{n}$ near $\gamma (t)$, $t \to - \infty$. Then $\Lambda = \Lambda_{\psi}$ for some function $\psi$ satisfying $\psi '' (x(t)) = \diag ( - \lambda_{1} , \ldots , - \lambda_{\nu-1} , \lambda_{\nu} , - \lambda_{\nu+1} , \ldots , - \lambda_{n}) + \CO (e^{\varepsilon t})$.
\end{remark}

\begin{proof}
We follow the proof of \cite[Lemma 2.1]{HeSj85_01}. There exist symplectic local coordinates $(y, \eta )$ centered at $(0,0)$ such that $\Lambda_{-}$ (resp. $\Lambda_{+}$) is given by $y = 0$ (resp. $\eta =0$) and
\begin{align}
y_{j} =& \frac{1}{\sqrt{2 \lambda_{j}}} ( \xi_{j} + \lambda_{j} x_{j} ) + \CO ((x, \xi )^{2} ) , \\
\eta_{j} =& \frac{1}{\sqrt{2 \lambda_{j}}} ( \xi_{j} - \lambda_{j} x_{j} ) + \CO ((x, \xi )^{2} ) .
\end{align}
Then, $p (x, \xi ) = A (y , \eta ) y \cdot \eta$ with $A_{0} := A (0,0) = \diag (\lambda_{1} , \ldots , \lambda_{n} )$.

\begin{equation}  \label{q23}
\frac{d}{dt} \left( \begin{array}{c}
\delta_{y} \\
\delta_{\eta}
\end{array} \right) = \left( \begin{array}{cc}
A_{0} + \CO (e^{-\lambda_{1}t}) & 0 \\
\CO (e^{-\lambda_{1}t}) & A_{0} + \CO (e^{-\lambda_{1}t})
\end{array} \right) \left( \begin{array}{c}
\delta_{y} \\
\delta_{\eta}
\end{array} \right)
\end{equation}
We denote by $U(t,s)$ the linear operator such that $U (t,s) \delta$ solves \eqref{q23} with $U (s,s) = Id$.

Since $\Lambda \cap \Lambda_{-} = \gamma$ is transverse, there exists $E_{n-1} (t_{0}) \subset T_{\gamma (t_{0})} \Lambda$, a vector space of dimension $n-1$ disjoint from $T_{\gamma (t_{0})} \Lambda_{-}$. For convenience, we set $E_{n} (t_{0}) = E_{n-1} (t_{0}) \oplus \R v$ for some $v \notin T_{\gamma (t_{0})} \Lambda + T_{\gamma (t_{0})} \Lambda_{-}$. Let $E_{\bullet} (t) = U (t, t_{0}) E_{\bullet} (t_{0})$. From \cite[Lemma 2.1]{HeSj85_01}, there exists $B_{t} = \CO (e^{- \lambda_{1} t})$ such that $E_{n} (t)$ is given by $\delta_{\eta} = B_{t} \delta_{x}$. Now, if $\delta \in E_{n-1} (t)$, we have $\sigma ( H_{p} , \delta )=0$ since $E_{n-1}(t) \oplus \R H_{p} = T_{\gamma (t)} \Lambda$ and $\Lambda$ is a Lagrangian manifold. From \eqref{q24}, we have
\begin{equation}
H_{p} (\gamma (t)) = \dot{\gamma} (t) = - \lambda_{\nu} ( \widetilde{a} e_{\eta_{\nu}}+ \CO (e^{-\varepsilon t})) e^{ - \lambda_{\nu} t} ,
\end{equation}
where $e_{\eta_{\nu}}$ is the basis vector corresponding to $\eta_{\nu}$ and then
\begin{equation}
0 = \sigma (e^{\lambda_{\nu} t} H_{p} , \delta ) = \lambda_{\nu} \widetilde{a} \delta_{y_{\nu}} + \CO (e^{-\varepsilon t}) \vert \delta \vert .
\end{equation}
It follows that $\delta \in E_{n-1} (t)$ if and only if $(\delta_{y_{\nu}} , \delta_{\eta} ) = \widetilde{B}_{t} \delta_{y'}$ with $\widetilde{B}_{t} = \CO (e^{-\varepsilon t})$. Using $T_{\gamma (t)} \Lambda = E_{n-1}(t) \oplus \R H_{p}$, we obtain that $T_{\gamma (t)} \Lambda$ has a basis formed of vector $f_{j} (t)$ such that
\begin{align}
f_{j} =& e_{y_{j}} + \CO (e^{- \varepsilon t}) \qquad \text{for } j \neq \nu \\
f_{\nu} =& e_{\eta_{\nu}} + \CO (e^{- \varepsilon t}).
\end{align}
In the $(x,\xi)$-coordinates, $T_{\gamma (t)} \Lambda$ has a basis formed of vector $\widetilde{f}_{j} (t)$ of the form
\begin{align}
\widetilde{f}_{j} =& e_{\xi_{j}} + \lambda_{j} e_{x_{j}} + \CO (e^{- \varepsilon t}) \qquad \text{for } j \neq \nu \\
\widetilde{f}_{\nu} =& e_{\xi_{\nu}} - \lambda_{j} e_{x_{\nu}} + \CO (e^{- \varepsilon t}) ,
\end{align}
and the lemma follows.
\end{proof}

\section{Asymptotic behaviour of certain integrals}

\begin{lemma} \label{a25}
Let $\alpha\in\C$, $\re\alpha >0$, $\beta \in \R$ and $\chi \in C_{0}^{\infty} (] - \infty , 1/2 [ )$ be such that $\chi =1$ near $0$. As $\lambda$ goes to $+ \infty$, we have
\begin{equation}  \label{l1}
\int_{0}^{\infty} e^{i \lambda t} t^{\alpha} ( - \ln t )^{\beta} \chi (t) \, \frac{d t}{t} = \Gamma ( \alpha ) ( \ln \lambda )^{\beta} (-i \lambda)^{- \alpha} ( 1 + o (1)) .
\end{equation}
Moreover, if $\beta \in \N$, we get
\begin{equation}  \label{l2}
\int_{0}^{\infty} e^{i \lambda t} t^{\alpha} ( - \ln t )^{\beta} \chi (t) \, \frac{d t}{t} = ( - i \lambda )^{- \alpha} \sum_{j = 0}^{\beta} C_{\beta}^{j} \Gamma^{( j )} ( \alpha)  (-1)^{j} \big( \ln (- i \lambda ) \big)^{\beta - j}  + \CO (\lambda^{- \infty}) .
\end{equation}
Finally, if $s (t) \in C^{\infty} (] 0 , + \infty [)$ satisfies
\begin{equation}  \label{a28}
\vert \partial^{j}_{t} s (t) \vert =  o \big( t^{\alpha -j} ( - \ln t )^{\beta} \big) ,
\end{equation}
for all $j \in \N$ and $t \to 0$, then
\begin{equation}  \label{a29}
\int_{0}^{\infty} e^{i \lambda t} s (t) \chi (t) \, \frac{d t}{t} = o \big( ( \ln \lambda )^{\beta} \lambda^{- \alpha} \big) .
\end{equation}
Here $(-i \lambda)^{- \alpha} = e^{i \alpha \pi /2} \lambda^{- \alpha}$ and $\ln (- i \lambda) = \ln \lambda - i \pi /2$.
\end{lemma}

\begin{remark}\sl\label{mk1}
Notice that one obtains the behaviour of these quantities as $\lambda\to -\infty$ by taking the complex conjugate in these expressions.
\end{remark}

\begin{proof}
We begin with \eqref{l2} and assume first that $\beta = 0$. Then, we can write
\begin{align}
\int_{0}^{\infty} e^{i \lambda t} t^{\alpha} \chi (t) \, \frac{d t}{t} =& \lim_{\varepsilon \to 0} \int_{0}^{\infty} e^{i ( \lambda + i \varepsilon ) t} t^{\alpha} \chi (t) \, \frac{d t}{t}   \nonumber  \\
=& \lim_{\varepsilon \to 0} \big( I_{1} (\alpha , \varepsilon ) + I_{2} ( \alpha , \varepsilon ) \big),
\end{align}
where
\begin{align}
I_{1} ( \alpha , \varepsilon ) =& \int_{0}^{\infty} e^{- ( \varepsilon - i \lambda ) t} t^{\alpha} \, \frac{d t}{t}\; ,  \label{l3}  \\
I_{2} ( \alpha , \varepsilon ) =& \int_{0}^{\infty} e^{i ( \lambda + i \varepsilon ) t} t^{\alpha} ( 1 - \chi (t)) \, \frac{d t}{t} \cdotp  \label{l4}
\end{align}
It is clear that
\begin{equation}
I_{1} ( \alpha , \varepsilon ) = (\varepsilon - i \lambda )^{- \alpha} \Gamma ( \alpha ) ,
\end{equation}
where $z^{- \alpha}$ is well defined on $\C \setminus ] - \infty , 0]$ and real positive on $]0, + \infty [$. In particular
\begin{equation}
\lim_{\varepsilon \to 0} I_{1} ( \alpha , \varepsilon ) = ( - i \lambda )^{- \alpha} \Gamma ( \alpha ).
\end{equation}
Concerning, $I_{2} ( \alpha , \varepsilon )$, we remark that $r (t , \alpha ) = t^{\alpha -1} (1- \chi (t))$ is a symbol which satisfies
\begin{equation}  \label{l6}
\vert \partial_{t}^{j}  \partial_{\alpha}^{k} r (t , \alpha ) \vert \lesssim \< t \>^{\re\alpha -1 -j} \< \ln t \>^{k} ,
\end{equation}
for all $j,k \in \N$ uniformly for $t \in [0, + \infty [$ and $\alpha$ in a compact set of $\{\re z>0\}$. Then, making integration by parts in \eqref{l4}, we obtain
\begin{equation}
I_{2} ( \alpha , \varepsilon ) = \frac{1}{( \varepsilon - i \lambda)^{j}} \int_{0}^{+ \infty} e^{(i \lambda - \varepsilon) t} \partial_{t}^{j} r (t, \alpha ) \, d t ,  \label{l5}
\end{equation}
for all $j \in \N$. Now, if $j$ is large enough ($j >\re \alpha$), $\partial_{t}^{j} r (t, \alpha )$ is integrable in time uniformly with respect to $\varepsilon$. In particular, for such $j$,
\begin{equation}
\lim_{\varepsilon \to 0} I_{2} ( \alpha , \varepsilon ) = e^{i j \pi / 2} \lambda^{-j} \int_{0}^{+ \infty} e^{i \lambda t} \partial_{t}^{j} r (t, \alpha ) \, d t ,
\end{equation}
and then (see \eqref{l6} or Cauchy's formula)
\begin{align}
\partial_{\alpha}^{k} \lim_{\varepsilon \to 0} I_{2} ( \alpha , \varepsilon ) =& e^{i j \pi / 2} \lambda^{-j} \int_{0}^{+ \infty} e^{i \lambda t} \partial_{t}^{j} \partial_{\alpha}^{k} r (t, \alpha ) \, d t   \nonumber   \\
=& \CO ( \lambda^{- \infty} ),  \label{l7}
\end{align}
for all $k \in \N$. Then we obtain \eqref{l2} for $\beta = 0$. To obtain the result for $\beta \in \N$, it is enough to see that
\begin{align}
\int_{0}^{\infty} e^{i \lambda t} t^{\alpha} ( \ln t )^{\beta} \chi (t) \, \frac{d t}{t} =& \partial_{\alpha}^{\beta} \int_{0}^{\infty} e^{i \lambda t} t^{\alpha} \chi (t) \, \frac{d t}{t}  \nonumber   \\
=& \partial_{\alpha}^{\beta} \big( ( - i \lambda )^{- \alpha} \Gamma ( \alpha ) \big) + \partial_{\alpha}^{\beta} \lim_{\varepsilon \to 0} I_{2} ( \alpha , \varepsilon )   \nonumber  \\
=& ( - i \lambda )^{- \alpha} \sum_{j = 0}^{\beta} C_{\beta}^{j} \Gamma^{( j )} ( \alpha) \big( - \ln (- i \lambda ) \big)^{\beta - j}  + \CO (\lambda^{- \infty}) ,
\end{align}
from \eqref{l7}. Thus, \eqref{l2} is proved.

Let $u$ be a function $C^{\infty} (]0, + \infty [)$ be such that
\begin{equation}
\vert \partial_{t}^{j} u (t) \vert \lesssim t^{\re\alpha -j} (- \ln t )^{\beta} ,
\end{equation}
near $0$. Let $\varphi \in C^{\infty} (\R)$ such that $\varphi =1$ for $t <1$ and $\varphi =0$ for $t>2$. For $\delta >0$, we have
\begin{equation}
\int_{0}^{+ \infty} e^{i \lambda t} u(t) \chi (t) \big( 1 - \varphi (t / \delta ) \big) \, \frac{d t}{t} = (-i \lambda)^{-N} \int_{0}^{\infty} e^{i \lambda t} \partial_{t}^{N} \Big( u(t) \chi (t) \big( 1 - \varphi (t / \delta ) \big) t^{-1} \Big) d t ,
\end{equation}
for all $N$.

If one of the derivatives falls on $1 - \varphi (t / \delta )$, the support of this contribution is inside $[ \delta , 2 \delta ]$. Therefore, the corresponding term will be bounded by $\delta^{\re\alpha - N -1} ( \ln \delta )^{\beta}$ and will contribute like $\delta^{\re\alpha -N} ( - \ln \delta )^{\beta}$ to the integral.

If ones of the derivatives falls on $\chi (t)$, the support of the integrand will be a compact set outside of $0$ and then this function will be $\CO (1)$. The contribution to the integral of such term will be like $1$.

If all the derivatives fall on $u(t) t^{-1}$, we corresponding term will satisfies
\begin{align}
\int_{0}^{\infty} e^{i \lambda t} \partial_{t}^{N} \big( u(t) t^{-1} \big) \chi (t) \big( 1 - \varphi (t / \delta ) \big) d t =& \CO (1) \int_{\delta}^{+ \infty} t^{\re\alpha -1 -N} (- \ln t)^{\beta} (1 - \chi (t)) d t  \nonumber  \\
\lesssim& ( - \ln \delta)^{\beta} \delta^{\re\alpha - N} ,
\end{align}
for $N$ large enough ($N > \re\alpha$).

From this 3 cases, we deduce
\begin{equation}
\int_{0}^{+ \infty} e^{i \lambda t} u(t) \chi (t) \big( 1 - \varphi (t / \delta ) \big) \, \frac{d t}{t} = \CO \big( ( - \ln \delta)^{\beta} \delta^{\alpha - N} \lambda^{-N} \big) .
\end{equation}
Taking $ \delta = (\varepsilon \lambda)^{-1}$, we get
\begin{equation}  \label{l8}
\int_{0}^{+ \infty} e^{i \lambda t} u(t) \chi (t) \big( 1 - \varphi (t / \delta ) \big) \, \frac{d t}{t} = \CO \big( \varepsilon ( \ln \lambda )^{\beta} \lambda^{- \alpha} \big) ,
\end{equation}
as $\lambda \to + \infty$.

We now assume \eqref{a28}, and we want to prove \eqref{a29}. Since, for $t$ small enough
\begin{equation}
t^{\re\alpha -1} ( - \ln t)^{\beta} \lesssim \big( t^{\re\alpha} (- \ln t)^{\beta} \big)' ,
\end{equation}
we get
\begin{align}
\Big\vert \int_{0}^{+ \infty} e^{i \lambda t} s (t) \chi (t) \varphi (t / \delta ) \, \frac{d t}{t} \Big\vert =& o_{\delta \to 0} (1) \int_{0}^{2 \delta} t^{\re\alpha -1} ( - \ln t)^{\beta} d t   \nonumber  \\
=& o_{\delta \to 0} (1) \delta^{\re\alpha} ( - \ln \delta )^{\beta} .
\end{align}
Here $o_{\delta \to 0} (1)$ stands for  a term which goes to $0$ as $\delta$ goes to $0$. If $\delta = (\varepsilon \lambda)^{-1}$, we get
\begin{equation}  \label{l9}
\Big\vert \int_{0}^{+ \infty} e^{i \lambda t} s (t) \chi (t) \varphi (t/ \delta ) \, \frac{d t}{t} \Big\vert = o_{\lambda \to + \infty} (1) \lambda^{-\alpha} ( \ln \lambda )^{\beta} ,
\end{equation}
when $\lambda \to + \infty$ and $\varepsilon$ fixed. Taking $\varepsilon$ small enough in \eqref{l8}, and  then $\lambda$ large enough in \eqref{l9}, we get \eqref{a29}.

We are left with \eqref{l1}. We need to compute
\begin{equation}
{\mathcal I} = \int_{0}^{+ \infty} e^{i \lambda t} t^{\alpha} ( - \ln t)^{\beta} \varphi(t / \delta) \, \frac{d t}{t} .
\end{equation}
Performing the change of variable $s = \lambda t$, we get
\begin{align}
{\mathcal I} =& \lambda^{- \alpha} \int_{0}^{2 / \varepsilon} e^{i s} s^{\alpha} (\ln \lambda - \ln s)^{\beta} \varphi( \varepsilon s ) \, \frac{d s}{s}  \nonumber   \\
=& ( \ln \lambda )^{\beta} \lambda^{- \alpha} \int_{0}^{2 / \varepsilon} e^{i s} s^{\alpha} (1 - \ln s / \ln \lambda )^{\beta} \varphi( \varepsilon s ) \, \frac{d s}{s} .
\end{align}
We remark that, in the previous equation, $- \ln s / \ln \lambda > - \ln (2 / \varepsilon ) / \ln \lambda > -1/2$ for $\lambda$ large enough. Using $(1 + u)^{\beta} = 1 + \CO ( \vert u \vert + \vert u \vert^{\max (1 , \beta)} )$ for $u > -1/2$, we get
\begin{align}
{\mathcal I} =& ( \ln \lambda )^{\beta} \lambda^{- \alpha} \int_{0}^{2 / \varepsilon} e^{i s} s^{\alpha} \varphi( \varepsilon s ) \, \frac{d s}{s}    \nonumber   \\
&+ ( \ln \lambda )^{\beta} \lambda^{- \alpha} \int_{0}^{2 / \varepsilon} s^{\re\alpha} \CO \Big( \frac{\vert \ln s \vert}{\ln \lambda} + \Big( \frac{\vert \ln s \vert}{\ln \lambda} \Big)^{\max (1, \beta)} \Big) \varphi( \varepsilon s ) \, \frac{d s}{s}   \nonumber  \\
=& ( \ln \lambda )^{\beta} \int_{0}^{+ \infty} e^{i \lambda t} t^{\alpha} ( - \ln t)^{\beta} \varphi(t / \delta) \, \frac{d t}{t} + \CO_{\varepsilon} \big( ( \ln \lambda )^{\beta -1} \lambda^{- \alpha} \big) .  \label{l10}
\end{align}
Note that the $\CO_{\varepsilon}$ in \eqref{l10} depends on $\varepsilon$.

Then, using \eqref{l8}, \eqref{l10} and  \eqref{l8} again, we get
\begin{align}
\int_{0}^{\infty} e^{i \lambda t} t^{\alpha} & ( - \ln t )^{\beta} \chi (t) \, \frac{d t}{t} \nonumber  \\
=& {\mathcal I} + \CO \big( \varepsilon ( \ln \lambda )^{\beta} \lambda^{- \alpha} \big)   \nonumber  \\
=& ( \ln \lambda )^{\beta} \int_{0}^{+ \infty} e^{i \lambda t} t^{\alpha} ( - \ln t)^{\beta} \varphi(t / \delta) \, \frac{d t}{t} + \CO \big( ( \ln \lambda )^{\beta -1} \lambda^{- \alpha} \big) + \CO \big( \varepsilon ( \ln \lambda )^{\beta} \lambda^{- \alpha} \big)  \nonumber  \\
=& ( \ln \lambda )^{\beta} \int_{0}^{+ \infty} e^{i \lambda t} t^{\alpha} ( - \ln t)^{\beta} \, \frac{d t}{t} + \CO \big( \varepsilon ( \ln \lambda )^{\beta} \lambda^{- \alpha} \big) + \CO_{\varepsilon} \big( ( \ln \lambda )^{\beta -1} \lambda^{- \alpha} \big) \nonumber   \\
&+ \CO \big( \varepsilon ( \ln \lambda )^{\beta} \lambda^{- \alpha} \big) .
\end{align}
Choosing $\varepsilon$ small enough, then $\lambda$ large enough, and using \eqref{l2} with $\beta =0$ to compute the first term, we obtain
\begin{equation}
\int_{0}^{\infty} e^{i \lambda t} t^{\alpha} ( - \ln t )^{\beta} \chi (t) \, \frac{d t}{t} = \Gamma ( \alpha ) ( \ln \lambda )^{\beta} (-i \lambda)^{- \alpha} ( 1 + o (1)) ,
\end{equation}
and this finishes the proof for \eqref{l1}.
\end{proof}

\bibliographystyle{amsplain}\providecommand{\bysame}{\leavevmode\hbox to3em{\hrulefill}\thinspace}
\providecommand{\MR}{\relax\ifhmode\unskip\space\fi MR }
\providecommand{\MRhref}[2]{%
  \href{http://www.ams.org/mathscinet-getitem?mr=#1}{#2}
}
\providecommand{\href}[2]{#2}

\end{document}